\documentclass[11pt,reqno]{amsart}
\usepackage[margin=1.1in]{geometry}
\usepackage{amsmath,amssymb,amsthm,mathtools,array,booktabs}
\usepackage{enumitem}
\usepackage{hyperref}
\hypersetup{colorlinks=true,linkcolor=blue,citecolor=blue}

\DeclareMathOperator{\sn}{sn}\DeclareMathOperator{\cn}{cn}
\DeclareMathOperator{\dn}{dn}\DeclareMathOperator{\sd}{sd}
\DeclareMathOperator{\cd}{cd}\DeclareMathOperator{\nd}{nd}
\DeclareMathOperator{\cs}{cs}\DeclareMathOperator{\sce}{sc}
\DeclareMathOperator{\nc}{nc}\DeclareMathOperator{\dc}{dc}
\DeclareMathOperator{\Real}{Re}\DeclareMathOperator{\Imag}{Im}
\DeclareMathOperator{\am}{am}
\newcommand{\CC}{\mathbb{C}}\newcommand{\RR}{\mathbb{R}}
\newcommand{\ZZ}{\mathbb{Z}}\newcommand{\KK}{\mathbb{K}}
\newcommand{\Om}{\Omega}\newcommand{\bt}{\beta}
\newcommand{\hh}{\mathsf{h}}\newcommand{\Pc}{\mathcal{P}}
\newcommand{\Nc}{\mathcal{N}}\newcommand{\dd}{\,\mathrm{d}}
\newcommand{\pip}{\pi_{+}}\newcommand{\pim}{\pi_{-}}
\newcommand{\Wz}{W_{\!\zeta}}\newcommand{\Ww}{W_{\!\omega}}
\newcommand{\Ez}{E_{\zeta}}\newcommand{\Ew}{E_{\omega}}
\newcommand{\Rr}{\mathcal{R}}
\newcommand{\Lat}{\mathcal{L}}
\newcommand{\Cs}{\mathsf{C}}\newcommand{\Qs}{\mathsf{Q}}
\newcommand{\Fs}{\mathsf{F}}

\theoremstyle{plain}
\newtheorem{theorem}{Theorem}[section]
\newtheorem{proposition}[theorem]{Proposition}
\newtheorem{lemma}[theorem]{Lemma}
\newtheorem{corollary}[theorem]{Corollary}

\theoremstyle{definition}
\newtheorem{definition}[theorem]{Definition}
\newtheorem{remark}[theorem]{Remark}
\newtheorem{convention}[theorem]{Convention}

\numberwithin{equation}{section}
\numberwithin{table}{section}

\title[Elimination for the Schwarz P surface]{Elimination of the
	parameters in an elliptic parametrization of the Schwarz primitive
	surface}
\author{Steven Finch}
\address{MIT Sloan School of Management, Cambridge, MA 02139, USA}
\email{steven\_finch\_math@outlook.com}
\date{\today}

\begin{document}
\begin{abstract}
Let $\Om$ be the curvilinear square bounded by the four circular arcs
$|\zeta\mp1|=\sqrt2$, $|\zeta\mp i|=\sqrt2$, and let
$W=\sqrt{1+14\zeta^{4}+\zeta^{8}}$. Put
\[
  \theta(\zeta)=\arcsin\frac{2(1+i)\zeta}
   {\sqrt{1+4i\zeta^{2}-\zeta^{4}}},
\]
\[
  f=\tfrac14\bigl(-iF[\theta,\tfrac14]+F[\theta,\tfrac34]\bigr),\quad
  g=\tfrac14\bigl(\ \ iF[\theta,\tfrac14]+F[\theta,\tfrac34]\bigr),
\]
\[
  \mathfrak h=(2-\sqrt3)\,
   F\bigl[\arcsin\bigl(i(2+\sqrt3)\zeta^{2}\bigr),(2-\sqrt3)^{4}\bigr].
\]
These are, verbatim, the three functions of our companion paper
\cite{D-paper} on the diamond surface \textup{D}, cited here as
\textup{[D]}: the two papers start from the same point, and differ only
in which real parts are taken. With
$X=\bigl(\kappa\Real f,\ \kappa\Real g,\
\tfrac12+\kappa\Imag\mathfrak h\bigr)$ and
$\kappa=3/(2K[1/9])$, the map $X$ parametrizes a fundamental patch of
Schwarz's primitive surface \textup{P}, the conjugate of \textup{D}.

Part I (\S\S1--6) builds the algebraic apparatus. The differentials are
algebraic on the genus-$3$ curve $W^{2}=1+14\zeta^{4}+\zeta^{8}$,
namely $f'=(1-\zeta^{2})/W$, $g'=i(1+\zeta^{2})/W$,
$\mathfrak h'=2i\zeta/W$ (Proposition \ref{prop:diff}). The three
rotated coordinates $x+y$, $x-y$, $z-\frac12$ are the real or imaginary
parts of \emph{single} elliptic integrals
$\Pc=\frac12F[\theta,\frac34]$, $\Nc=\frac12F[\theta,\frac14]$,
$\hh=-i\mathfrak h$, and all three admit \emph{rational} Jacobi
dictionaries at the \emph{same} modulus $-3$:
\[
  \sn\bigl(\Pc,-3\bigr)=\frac{(1+i)\zeta}{1-i\zeta^{2}},\qquad
  \sn\bigl(2\hh,-3\bigr)=\frac{2\zeta^{2}}{1+\zeta^{4}},\qquad
  \sn\bigl(2f,-3\bigr)=\frac{2\zeta(1+\zeta^{2})}{W},
\]
with $\cn$, $\dn$ likewise rational in $\zeta$ and in
$r_{\pm}=\sqrt{1\pm4i\zeta^{2}-\zeta^{4}}$, $r_{+}r_{-}=W$; the first
two follow from the closed forms above by a single Gauss
transformation. The normalizing constant is identified exactly:
$\varpi:=\kappa^{-1}=K[-3]=\frac12K[3/4]=\frac23K[1/9]$.

Part II (\S\S7--12) carries out the elimination and completes the
proof of
\[
  (\star)\qquad
  \sn\bigl(\varpi(x{+}y),-3\bigr)\,\sn\bigl(\varpi(x{-}y),-3\bigr)
  =\tfrac13\sn\bigl(3\varpi(z-\tfrac12),\tfrac19\bigr).
\]
Complexification ($\omega:=\bar\zeta$) makes all three arguments 
half-sums of one function of $\zeta$ and one of $\omega$ 
(Proposition \ref{prop:twovar}). On the diagonal
$\omega=\zeta$ we prove the exact identity
$\frac13\sn(3\hh,\frac19)=\sn^{2}(f,-3)=2\zeta^{2}/(1+\zeta^{4}+W)$,
which is $(\star)$ restricted to that diagonal and which
\emph{determines} the right-hand side, in particular the modulus
$\frac19$. We then show that $(\star)$ is equivalent to a single
\emph{antisymmetric rational} identity \textup{(P0)} in the field
$\KK=\CC(\zeta,\omega)\bigl(r_{\pm}(\zeta),r_{\pm}(\omega)\bigr)$, of
degree $\le16$: the logarithmic derivative of $(\star)$ in the two
directions $\partial_{\zeta}$, $\partial_{\omega}$ can be combined so
as to eliminate the height entirely, leaving a relation that makes no
reference to the right-hand side of $(\star)$ at all. Finally we
\emph{prove} \textup{(P0)} (Theorem \ref{thm:P0proved}): an explicit
cofactor certificate exhibits the cleared-denominator form of
\textup{(P0)} as a member of the ideal generated by
$r_{\pm}(\zeta)^{2}-(1\pm4i\zeta^{2}-\zeta^{4})$ and
$r_{\pm}(\omega)^{2}-(1\pm4i\omega^{2}-\omega^{4})$, and the
specialization of indeterminates to branches is a ring homomorphism
killing each generator. Hence $(\star)$ holds on all of $\Om$
(Theorem \ref{thm:main}, Corollary \ref{cor:done}). The certificate is
finite and can be verified by polynomial expansion alone. Independent
hand verifications --- on the curve $\omega=0$, in homogeneous degrees
$0$ and $2$, on the diagonal, and through the Taylor jets of orders
$2,4,6,8$, which force the modulus $-3$ via $3m^{2}+10m+3=0$ ---
corroborate. 

Part III (\S\S13--18) shows that $(\star)$ is, after all, additively
separable: $\Psi(x)=\Psi(y)+\Psi(z)$, where
$\Psi=\arctan\bigl(\sqrt3\cn(4\varpi\,\cdot,\tfrac34)\bigr)
=\am\bigl(\sqrt3\varpi(1+2\,\cdot),\tfrac43\bigr)$ is bounded by
$\pi/3$ --- anticipated in Remark \ref{rem:forms} --- and that $(\star)$ is
therefore equivalent to the trilinear identity
$\Cs(x)-\Cs(y)-\Cs(z)=3\,\Cs(x)\Cs(y)\Cs(z)$ up to one lattice orbit,
the labyrinth centers. The resulting surface is minimal and regular,
whence the local converse (Corollary \ref{cor:localconv}); and
Bj\"orling's theorem along the segment $\{(s,-s,\tfrac12)\}$ yields a
second proof of $(\star)$, independent of the certificate. The pivot is
that $F[\pi/3,\tfrac43]=\tfrac{\sqrt3}{2}K[3/4]=\sqrt3\varpi$ is a
complete integral in disguise; a corollary is that \textup{P} is the
symmetric member of the separable family of Kim and Ogata \cite{KO},
which proves their \cite[Thm.~1(3)]{KO}.

For an overview of this paper and \textup{[D]}, please see 
\href{https://arxiv.org/abs/2609.14206}{arXiv:2609.14206}.

\end{abstract}

\maketitle
\newpage
\tableofcontents

\part*{PART I. THE PARAMETRIZATION AND ITS ALGEBRAIC DICTIONARY}

%=====================================================================
\section{Introduction; the main theorem}
%=====================================================================

\subsection{Data and elliptic conventions}
Throughout, for $m\in\CC$ and admissible $\varphi$,
\begin{equation}\label{eq:F}
  F[\varphi,m]=\int_{0}^{\sin\varphi}
  \frac{\dd\tau}{\sqrt{1-\tau^{2}}\sqrt{1-m\tau^{2}}},
  \qquad K[m]=F[\pi/2,m],
\end{equation}
and $\sn(u,m)$, $\cn(u,m)$, $\dn(u,m)$ are the Jacobi functions in the
\emph{parameter} convention $m=k^{2}$ (so $k'=\sqrt{1-m}$), as in
\cite{BF}; thus
\begin{equation}\label{eq:snF}
  \sn\bigl(F[\varphi,m],m\bigr)=\sin\varphi,\qquad
  \sn^{2}+\cn^{2}=1,\qquad \dn^{2}+m\sn^{2}=1 .
\end{equation}
We use the quotient abbreviations $\sd=\sn/\dn$, $\cd=\cn/\dn$,
$\nd=1/\dn$, $\cs=\cn/\sn$, $\sce=\sn/\cn$, $\nc=1/\cn$,
$\dc=\dn/\cn$. These are the conventions of \textup{[D, \S1.1]} and of
Mathematica's \texttt{EllipticF}, \texttt{EllipticK},
\texttt{JacobiSN}.

Throughout,
\begin{equation}\label{eq:beta}
  \bt:=2+\sqrt3,\qquad
  \bt+\bt^{-1}=4,\quad \bt-\bt^{-1}=2\sqrt3,\quad
  \bt^{2}+\bt^{-2}=14,\quad \bt^{2}-1=2\sqrt3\,\bt,
\end{equation}
\begin{equation}\label{eq:W}
  W(\zeta):=\sqrt{1+14\zeta^{4}+\zeta^{8}},\qquad W(0)=+1 ,
\end{equation}
and we fix once and for all the auxiliary quantities
\begin{equation}\label{eq:aux}
  \pi_{\pm}:=1\pm i\zeta^{2},\qquad
  r_{\pm}:=\sqrt{1\pm4i\zeta^{2}-\zeta^{4}},\qquad
  s_{\pm}:=\sqrt{1+\bt^{\pm2}\zeta^{4}},\qquad
  E:=1+\zeta^{4},
\end{equation}
all square roots being the branches equal to $+1$ at $\zeta=0$
(Convention \ref{conv:br}). The domain is
\begin{equation}\label{eq:Omega}
  \Om:=\bigl\{\zeta\in\CC:\ |\zeta-1|\le\sqrt2,\ |\zeta+1|\le\sqrt2,\
  |\zeta-i|\le\sqrt2,\ |\zeta+i|\le\sqrt2\bigr\},
\end{equation}
the same domain as in \textup{[D, \S1.2]}. Finally
\begin{equation}\label{eq:kappa}
  \kappa:=\frac{3}{2K[1/9]}=0.9274219746\ldots,\qquad
  \varpi:=\kappa^{-1}=1.0782578237\ldots
\end{equation}
(Table \ref{tab:num}).

\subsection{The three functions}
Let $\zeta=u+iv$ and
\begin{equation}\label{eq:theta}
  \theta(\zeta):=\arcsin\bigl(s(\zeta)\bigr),\qquad
  s(\zeta):=\frac{2(1+i)\zeta}{\sqrt{1+4i\zeta^{2}-\zeta^{4}}}
           =\frac{2(1+i)\zeta}{r_{+}} .
\end{equation}
Define
\begin{align}
  f(\zeta)&=\tfrac14\Bigl(-i\,F\bigl[\theta(\zeta),\tfrac14\bigr]
            +F\bigl[\theta(\zeta),\tfrac34\bigr]\Bigr),
            \label{eq:fclosed}\\
  g(\zeta)&=\tfrac14\Bigl(\ \ i\,F\bigl[\theta(\zeta),\tfrac14\bigr]
            +F\bigl[\theta(\zeta),\tfrac34\bigr]\Bigr),
            \label{eq:gclosed}\\
  \mathfrak h(\zeta)&=\bt^{-1}\,
   F\bigl[\arcsin\bigl(i\bt\zeta^{2}\bigr),\bt^{-4}\bigr],
   \qquad \bt^{-1}=2-\sqrt3 .\label{eq:hclosed}
\end{align}
Equations \eqref{eq:theta}--\eqref{eq:hclosed} are the starting point
of the paper. They are, verbatim, the three functions and the amplitude
of \textup{[D, \S1.3]}: \emph{the primitive and the diamond surface are
built from the same three integrals}, and differ only in which real
parts are taken (Remark \ref{rem:PvsD}).

\begin{convention}[branches]\label{conv:br}
All square roots in \eqref{eq:aux} --- in particular $W$, $r_{\pm}$,
$s_{\pm}$ and the radical inside \eqref{eq:theta} --- are the analytic
continuations from $\zeta=0$ of the branches equal to $+1$ there, along
paths in $\Om^{\circ}$. By \eqref{eq:F} the composite
$F[\theta(\zeta),m]$ depends on $\theta$ only through
$s=\sin\theta$, so it is to be read as
\begin{equation}\label{eq:Fpath}
  F\bigl[\theta(\zeta),m\bigr]
  =\int_{0}^{s(\zeta)}
   \frac{\dd\tau}{\sqrt{(1-\tau^{2})(1-m\tau^{2})}},
\end{equation}
the path of integration being the continuous deformation, as $\zeta$
moves in $\Om^{\circ}$ from $0$, of the segment $[0,s]$; and similarly
for $F[\arcsin(i\bt\zeta^{2}),\bt^{-4}]$. Equivalently --- and this is
the form we use --- $f$, $g$, $\mathfrak h$ are the primitives
\eqref{eq:fgint} of the algebraic differentials
\eqref{eq:fgh} vanishing at $\zeta=0$.

Near $\zeta=0$ all of these agree with Mathematica's principal
branches. Farther out the principal \texttt{ArcSin} and
\texttt{EllipticF} jump: $r_{+}$ vanishes at two of the four corners of
$\Om$ and $r_{-}$ at the other two (Lemma \ref{lem:prod}), so
$s\to\infty$ there; and $i\bt\zeta^{2}\to\pm1$ on parts of
$\partial\Om$. At such points \eqref{eq:Fpath}, i.e.,\
\eqref{eq:fgint}, is the intended continuation. This is the branch
convention of \textup{[D, \S1.3]} verbatim.
\end{convention}

\subsection{The map}
Put
\begin{equation}\label{eq:map}
  X(u,v)=(x,y,z):=\Bigl(\kappa\Real f,\ \kappa\Real g,\
  \tfrac12+\kappa\Imag\mathfrak h\Bigr).
\end{equation}
This is the system of three equations in five unknowns of the problem.
It is convenient to replace $\mathfrak h$ by
\begin{equation}\label{eq:hh}
  \hh:=-i\mathfrak h,\qquad\text{so that}\qquad
  \Imag\mathfrak h=\Real\hh ,
\end{equation}
so that \eqref{eq:map} reads
$X=\kappa(\Real f,\Real g,\Real\hh)+(0,0,\tfrac12)$: all three
coordinates are real parts of the \emph{same} Weierstrass triple,
which is what makes the surface minimal
(Proposition \ref{prop:conf}).

\subsection{The theorem}
\begin{theorem}[implicit representation of P]\label{thm:main}
For every $\zeta\in\Om$ the point $(x,y,z)=X(\zeta)$ satisfies
\begin{equation}\label{eq:star}
  \boxed{\ \sn\bigl(\varpi(x+y),-3\bigr)\cdot
          \sn\bigl(\varpi(x-y),-3\bigr)
        =\tfrac13\,\sn\bigl(3\varpi(z-\tfrac12),\tfrac19\bigr)\ }
\end{equation}
where $\varpi=\kappa^{-1}=K[-3]$ and $3\varpi=2K[1/9]$
\textup{(Proposition \ref{prop:web})}.  We call this representation $(\star)$.
\end{theorem}

The proof is completed in \S\ref{sec:pfaff}: \eqref{eq:star} is
reduced to the single rational identity \textup{(P0)} of
Theorem \ref{thm:P0}, which is then proved outright in
Theorem \ref{thm:P0proved}; the initial condition comes from the
diagonal identity of \S\ref{sec:diag}. See
Corollary \ref{cor:done}.

\begin{remark}[equivalent forms]\label{rem:forms}
	By the product formula
	$\sn(a+b)\sn(a-b)=\frac{\sn^{2}a-\sn^{2}b}{1-m\sn^{2}a\sn^{2}b}$
	\cite[123.02]{BF} with $a=\varpi x$, $b=\varpi y$, $m=-3$,
	\eqref{eq:star} is equivalent to
	\begin{equation}\label{eq:star2}
	\sn^{2}(\varpi x,-3)-\sn^{2}(\varpi y,-3)
	=\tfrac13\sn\bigl(3\varpi(z-\tfrac12),\tfrac19\bigr)
	\bigl[1+3\sn^{2}(\varpi x,-3)\sn^{2}(\varpi y,-3)\bigr].
	\end{equation}
	Note the \emph{sign} in the denominator of the product formula: at
	$m=-3$ it is $1+3\sn^{2}\sn^{2}\ge1$, so the quantity in
	\eqref{eq:star2} is $\tfrac{1}{\sqrt3}\tan$ of a difference of two
	\emph{arctangents}, not of two inverse hyperbolic tangents. This is
	what makes $(\star)$ \emph{additively separable}: with
	\[
	\Psi:=\arctan\bigl(\sqrt3\,\Cs\bigr),\qquad
	\Cs:=\cn\bigl(4\varpi\,\cdot\,,\tfrac34\bigr),
	\]
	a function bounded by $\tfrac\pi3$ and real-analytic on all of $\RR$,
	\eqref{eq:star} is equivalent to
	\begin{equation}\label{eq:star3}
	\Psi(x)=\Psi(y)+\Psi(z)
	\end{equation}
	\textup{(Corollary \ref{cor:sepform})}, and hence, after clearing
	tangents, to the \emph{trilinear} identity
	\begin{equation}\label{eq:star4}
	\Cs(x)-\Cs(y)-\Cs(z)=3\,\Cs(x)\Cs(y)\Cs(z)
	\end{equation}
	up to one lattice orbit, the labyrinth centers
	\textup{(Theorem \ref{thm:trilin})}. All three summands of
	\eqref{eq:star3} are regular at the base point, and $\Psi$ has no
	singularity anywhere; the multiplicative form \eqref{eq:star} is
	nevertheless the one that Part II proves and the one with no
	exceptional set at all \textup{(Corollary \ref{cor:noexc})}, which is
	why it is taken as primary. The companion surface \textup{D} is
	likewise separable, with separating function
	$\log\frac{\sn\dn}{\cn}$ at modulus $\tfrac14$; the two separating
	functions satisfy $\Psi'^{2}=4\varpi^{2}(1+2\cos2\Psi)$ and
	$\Psi'^{2}=4\varpi^{2}(1+2\cosh2\Psi)$ respectively --- the two signs
	of one differential equation \textup{(Remark \ref{rem:norm3})}.
\end{remark}

\subsection{What is proved here}
\begin{enumerate}[label=\textup{(\alph*)},leftmargin=2.3em]
\item The algebraic form of the differentials
      (Proposition \ref{prop:diff}), derived from
      \eqref{eq:theta}--\eqref{eq:hclosed}; and the full dictionary
      (\S\ref{sec:dict}): $\Pc$, $\Nc$, $\hh$ and $2f$, $2g$ all at
      modulus $-3$, with rational $\sn,\cn,\dn$. The first two are
      three-line consequences of \eqref{eq:fclosed}--\eqref{eq:gclosed}
      and one Gauss transformation
      (Theorems \ref{thm:dictA}--\ref{thm:dictB}). The exact
      identification
      $\varpi=K[-3]=\frac12K[3/4]=\frac23K[1/9]=\kappa^{-1}$
      is \S\ref{sec:web}.
\item The transformation law under the arc reflection, the four mirror
      planes $x+z=1$, $z-x=1$, $y+z=0$, $y=z$, the body-centered cubic
      invariance of \eqref{eq:star}, and the exact corner values
      (\S\ref{sec:sym}).
\item The exact one-variable identity
      $\frac13\sn(3\hh,\frac19)=\sn^{2}(f,-3)
      =2\zeta^{2}/(1+\zeta^{4}+W)$, i.e.,\ $(\star)$ on the diagonal
      (\S\ref{sec:diag}); this \emph{determines} the right-hand side
      of \eqref{eq:star}, in particular the modulus $\frac19$.
\item Reduction of Theorem \ref{thm:main} to one antisymmetric
      rational identity \textup{(P0)} in a field of degree $\le16$,
      and the \emph{proof} of \textup{(P0)}
      (\S\ref{sec:pfaff}, Theorems \ref{thm:P0}
      and \ref{thm:P0proved}); hence Theorem \ref{thm:main}
      (Corollary \ref{cor:done}).
\item Independent verifications of \textup{(P0)}: on the curve
      $\omega=0$, in homogeneous degrees $0$ and $2$, and on the
      diagonal (\S\ref{sec:verif}); and the jets of orders
      $2,4,6,8$, which force $\mathfrak A=\sn(\varpi\,\cdot,-3)$
      through $3m^{2}+10m+3=0$ and pass one free consistency check
      (\S\ref{sec:jets}).
\item The separable form (\S\S13--16), the minimality and regularity of
      the resulting surface, the local converse (\S\ref{sec:minIII}),
      and a second, certificate-free proof of Theorem \ref{thm:main}
      (\S\ref{sec:corrIII}).
\end{enumerate}
Left open: a conceptual proof of (P0), and the global converse. 
§\ref{sub:statusfull} states the status precisely 
and records the relation to \textup{[D]}.

%=====================================================================
\section{The Weierstrass data}
%=====================================================================

This section removes all transcendence from
\eqref{eq:fclosed}--\eqref{eq:hclosed}: the three functions have
\emph{algebraic} derivatives on the curve
$W^{2}=1+14\zeta^{4}+\zeta^{8}$.

\subsection{The amplitude}

\begin{lemma}[Legendre data of $\theta$ and of
$\arcsin(i\bt\zeta^{2})$]\label{lem:ampl}
With $s$ as in \eqref{eq:theta} and $\sigma:=i\bt\zeta^{2}$,
\begin{equation}\label{eq:salg}
  s^{2}=\frac{8i\zeta^{2}}{r_{+}^{2}},\qquad
  1-s^{2}=\frac{r_{-}^{2}}{r_{+}^{2}},\qquad
  1-\tfrac14s^{2}=\frac{\pip^{2}}{r_{+}^{2}},\qquad
  1-\tfrac34s^{2}=\frac{\pim^{2}}{r_{+}^{2}},\qquad
  s'=\frac{2(1+i)E}{r_{+}^{3}},
\end{equation}
\begin{equation}\label{eq:sigalg}
  1-\sigma^{2}=s_{+}^{2},\qquad
  1-\bt^{-4}\sigma^{2}=s_{-}^{2},\qquad
  s_{+}s_{-}=W,\qquad \sigma'=2i\bt\zeta .
\end{equation}
All branches are $+1$ at $\zeta=0$ (Convention \ref{conv:br}).
\end{lemma}

\begin{proof}
Since $(1+i)^{2}=2i$ and $r_{+}^{2}=1+4i\zeta^{2}-\zeta^{4}$, the
first formula of \eqref{eq:salg} is immediate, and then
\[
  1-s^{2}=\frac{r_{+}^{2}-8i\zeta^{2}}{r_{+}^{2}}
   =\frac{1-4i\zeta^{2}-\zeta^{4}}{r_{+}^{2}}
   =\frac{r_{-}^{2}}{r_{+}^{2}},\quad
  1-\tfrac14s^{2}=\frac{r_{+}^{2}-2i\zeta^{2}}{r_{+}^{2}}
   =\frac{1+2i\zeta^{2}-\zeta^{4}}{r_{+}^{2}}
   =\frac{\pip^{2}}{r_{+}^{2}},
\]
using $\pip^{2}=(1+i\zeta^{2})^{2}=1+2i\zeta^{2}-\zeta^{4}$; and
likewise $1-\tfrac34s^{2}=(1-2i\zeta^{2}-\zeta^{4})/r_{+}^{2}
=\pim^{2}/r_{+}^{2}$. For $s'$: with
$(r_{+}^{2})'=8i\zeta-4\zeta^{3}$,
\[
  s'=2(1+i)\Bigl(r_{+}^{-1}
    -\tfrac12\zeta r_{+}^{-3}(r_{+}^{2})'\Bigr)
   =\frac{2(1+i)}{r_{+}^{3}}
    \Bigl(r_{+}^{2}-\tfrac12\zeta(8i\zeta-4\zeta^{3})\Bigr)
   =\frac{2(1+i)(1+\zeta^{4})}{r_{+}^{3}},
\]
because $1+4i\zeta^{2}-\zeta^{4}-4i\zeta^{2}+2\zeta^{4}=1+\zeta^{4}=E$.
For \eqref{eq:sigalg}: $1-\sigma^{2}=1+\bt^{2}\zeta^{4}=s_{+}^{2}$ and
$1-\bt^{-4}\sigma^{2}=1+\bt^{-2}\zeta^{4}=s_{-}^{2}$, whose product is
$1+(\bt^{2}+\bt^{-2})\zeta^{4}+\zeta^{8}=W^{2}$ by \eqref{eq:beta}.
\end{proof}

\subsection{The differentials}

\begin{proposition}[algebraic form of the differentials]
\label{prop:diff}
On $\Om^{\circ}$, with the branches of Convention \ref{conv:br},
\begin{equation}\label{eq:dF}
  \frac{\dd}{\dd\zeta}F\bigl[\theta,\tfrac14\bigr]
    =\frac{2(1+i)\,\pim}{W},\qquad
  \frac{\dd}{\dd\zeta}F\bigl[\theta,\tfrac34\bigr]
    =\frac{2(1+i)\,\pip}{W},
\end{equation}
and therefore
\begin{equation}\label{eq:fgh}
  \boxed{\ \
  f'(\zeta)=\frac{1-\zeta^{2}}{W},\qquad
  g'(\zeta)=\frac{i\,(1+\zeta^{2})}{W},\qquad
  \mathfrak h'(\zeta)=\frac{2i\,\zeta}{W},\qquad
  \hh'(\zeta)=\frac{2\zeta}{W}.\ \ }
\end{equation}
Consequently
\begin{equation}\label{eq:fgint}
  f(\zeta)=\int_{0}^{\zeta}\frac{1-\eta^{2}}{W(\eta)}\dd\eta,\quad
  g(\zeta)=\int_{0}^{\zeta}\frac{i(1+\eta^{2})}{W(\eta)}\dd\eta,\quad
  \hh(\zeta)=\int_{0}^{\zeta}\frac{2\eta\dd\eta}{W(\eta)},
\end{equation}
the integrals being taken along any path in $\Om^{\circ}$ from $0$.
\end{proposition}

\begin{proof}
By \eqref{eq:Fpath} and the chain rule,
$\frac{\dd}{\dd\zeta}F[\theta,m]
=s'\bigl[(1-s^{2})(1-ms^{2})\bigr]^{-1/2}$. For $m=\tfrac14$,
Lemma \ref{lem:ampl} gives
$(1-s^{2})(1-\tfrac14s^{2})=r_{-}^{2}\pip^{2}/r_{+}^{4}$, whose
square root ($+1$ at $\zeta=0$) is $r_{-}\pip/r_{+}^{2}$; hence
\[
  \frac{\dd}{\dd\zeta}F\bigl[\theta,\tfrac14\bigr]
  =\frac{2(1+i)E/r_{+}^{3}}{r_{-}\pip/r_{+}^{2}}
  =\frac{2(1+i)E}{r_{+}r_{-}\,\pip}
  =\frac{2(1+i)\pip\pim}{W\,\pip}
  =\frac{2(1+i)\pim}{W},
\]
using $r_{+}r_{-}=W$ and $E=\pip\pim$ (Lemma \ref{lem:prod}). The
computation for $m=\tfrac34$ is identical with $\pip$ replaced by
$\pim$, giving $2(1+i)\pip/W$. This is \eqref{eq:dF}.

Now differentiate \eqref{eq:fclosed}. Since $-i(1+i)=1-i$,
\[
  f'=\frac{1}{4}\cdot\frac{2(1+i)}{W}
     \Bigl[-i\,\pim+\pip\Bigr]
   =\frac{1}{2W}\Bigl[(1-i)\pim+(1+i)\pip\Bigr]
   =\frac{1}{2W}\cdot 2(1-\zeta^{2})=\frac{1-\zeta^{2}}{W},
\]
because
$(1+i)(1+i\zeta^{2})+(1-i)(1-i\zeta^{2})
=\bigl[(1+i)+(1-i)\bigr]+\zeta^{2}\bigl[i(1+i)-i(1-i)\bigr]
=2-2\zeta^{2}$. Likewise, since $i(1+i)=-(1-i)$,
\[
  g'=\frac{1}{2W}\Bigl[-(1-i)\pim+(1+i)\pip\Bigr]
    =\frac{1}{2W}\cdot 2i(1+\zeta^{2})=\frac{i(1+\zeta^{2})}{W},
\]
because
$(1+i)(1+i\zeta^{2})-(1-i)(1-i\zeta^{2})
=2i+\zeta^{2}\bigl[i(1+i)+i(1-i)\bigr]=2i+2i\zeta^{2}$.

For $\mathfrak h$: by \eqref{eq:sigalg} the product
$(1-\sigma^{2})(1-\bt^{-4}\sigma^{2})=s_{+}^{2}s_{-}^{2}=W^{2}$, so
\eqref{eq:hclosed} and the chain rule give
$\mathfrak h'=\bt^{-1}\sigma'/W=\bt^{-1}\cdot2i\bt\zeta/W
=2i\zeta/W$; and $\hh'=-i\mathfrak h'=2\zeta/W$ by \eqref{eq:hh}.
Finally \eqref{eq:fgint} follows since all three functions vanish at
$\zeta=0$. (Verified numerically to $16$ digits at four points:
Appendix \ref{app:code}, Cell 0.)
\end{proof}

\begin{proposition}[conformality; the associate family]\label{prop:conf}
Let $\phi:=(f',g',\hh')=\bigl(\frac{1-\zeta^{2}}{W},
\frac{i(1+\zeta^{2})}{W},\frac{2\zeta}{W}\bigr)$. Then
\begin{equation}\label{eq:iso}
  \phi_{1}^{2}+\phi_{2}^{2}+\phi_{3}^{2}=0,\qquad
  |\phi|^{2}=\frac{2(1+|\zeta|^{2})^{2}}{|W|^{2}}>0 ,
\end{equation}
so $\zeta\mapsto\Real\int\phi$ is a conformal minimal immersion with
Gauss map $G=\zeta$, unit normal
\begin{equation}\label{eq:normal}
  N=\frac{1}{1+|\zeta|^{2}}\bigl(2u,\,2v,\,|\zeta|^{2}-1\bigr),
\end{equation}
and height differential $\dd\hh=2\zeta\dd\zeta/W$. For each Bonnet
angle $\psi$, $X_{\psi}:=\Real\int e^{i\psi}\phi$ is minimal;
$\psi=0$ gives \textup{P} and $\psi=\pi/2$ gives \textup{D}.
Consequently \textup{P} and \textup{D} are conjugate, hence isometric.
\end{proposition}

\begin{proof}
$(1-\zeta^{2})^{2}-(1+\zeta^{2})^{2}+4\zeta^{2}=0$, which is
\eqref{eq:iso}(i); and
$|1-\zeta^{2}|^{2}+|1+\zeta^{2}|^{2}+4|\zeta|^{2}
=2+2|\zeta|^{4}+4|\zeta|^{2}=2(1+|\zeta|^{2})^{2}$. Since
$\phi\propto(1-G^{2},\,i(1+G^{2}),\,2G)$ with $G=\zeta$, the normal is
\eqref{eq:normal}; the rest is the Weierstrass--Enneper theorem
\cite[\S3.3]{Nitsche}. In \textup{[D, \S2.3]} the same $\phi$ occurs
with $\mathfrak h=i\hh$ and imaginary parts in the first two slots,
which is $\psi=\pi/2$; see Remark \ref{rem:PvsD}.
\end{proof}

\begin{remark}[P versus D in one line]\label{rem:PvsD}
Writing $\Pc:=f+g$, $\Nc:=i(f-g)$ (Proposition \ref{prop:sep}), the
two surfaces are
\[
  \text{D}:\ \bigl(\Imag\Pc,\ \Real\Nc,\ \Imag\hh\bigr),\qquad
  \text{P}:\ \bigl(\Real\Pc,\ \Imag\Nc,\ \Real\hh\bigr).
\]
The functions $f,g,\mathfrak h$ of
\eqref{eq:fclosed}--\eqref{eq:hclosed} are the same in both papers;
\emph{every} $\Real$ becomes $\Imag$ and conversely. This single swap
is responsible for all the differences cataloged in
Table \ref{tab:DP}.
\end{remark}

\subsection{Separation of the coordinates}

\begin{proposition}[separation]\label{prop:sep}
Put $\Pc:=f+g$, $\Nc:=i(f-g)$ and
\begin{equation}\label{eq:AB}
  A:=\Pc=f+g,\qquad B:=-i\Nc=f-g,
\end{equation}
so that $A+B=2f$ and $A-B=2g$. Then:

\textup{(a)} $\Pc$ and $\Nc$ are \emph{halves of the two incomplete
integrals} of \eqref{eq:fclosed}--\eqref{eq:gclosed}:
\begin{equation}\label{eq:PNclosed}
  \boxed{\ \Pc=A=\tfrac12F\bigl[\theta,\tfrac34\bigr],\qquad
  \Nc=\tfrac12F\bigl[\theta,\tfrac14\bigr],\qquad
  B=-i\Nc=-\tfrac i2F\bigl[\theta,\tfrac14\bigr].\ }
\end{equation}

\textup{(b)} Consequently
\begin{equation}\label{eq:PNder}
  A'=\frac{(1+i)\pip}{W},\qquad
  B'=\frac{(1-i)\pim}{W},\qquad
  \hh'=\frac{2\zeta}{W} .
\end{equation}

\textup{(c)} With
\begin{equation}\label{eq:pqZ}
  p:=x+y,\qquad q:=x-y,\qquad Z:=z-\tfrac12,
\end{equation}
one has on $\Om$
\begin{equation}\label{eq:sepP}
  \varpi p=\Real\Pc,\qquad \varpi q=\Imag\Nc,\qquad
  \varpi Z=\Real\hh .
\end{equation}
Each right-hand side involves only \emph{one} of the three integrals
of \eqref{eq:fclosed}--\eqref{eq:hclosed}.
\end{proposition}

\begin{proof}
(a) Adding \eqref{eq:fclosed} and \eqref{eq:gclosed} cancels
$F[\theta,\frac14]$ and leaves
$f+g=\tfrac12F[\theta,\tfrac34]$; subtracting them cancels
$F[\theta,\frac34]$ and leaves
$f-g=-\tfrac i2F[\theta,\tfrac14]$, whence
$\Nc=i(f-g)=\tfrac12F[\theta,\tfrac14]$.

(b) Differentiate \eqref{eq:PNclosed} using \eqref{eq:dF}:
$A'=\tfrac12\cdot\frac{2(1+i)\pip}{W}$ and
$B'=-\tfrac i2\cdot\frac{2(1+i)\pim}{W}=\frac{(1-i)\pim}{W}$, since
$-i(1+i)=1-i$. (Equivalently, add and subtract \eqref{eq:fgh}.)

(c) By \eqref{eq:map}, $\varpi p=\Real(f+g)=\Real\Pc$ and
$\varpi q=\Real(f-g)=\Real(-i\Nc)=\Imag\Nc$; the third is
\eqref{eq:hh}.
\end{proof}

\begin{remark}\label{rem:whyrotate}
Part (a) is the reason for the rotation. Individually, $x$ and $y$
involve \emph{both} $F[\theta,\frac14]$ and $F[\theta,\frac34]$ ---
integrals on two different elliptic quotients of the genus-$3$ curve
$W^{2}=1+14\zeta^{4}+\zeta^{8}$ --- so no algebraic relation ties
$\sn(\varpi x)$ to $\zeta$. The combinations \eqref{eq:sepP} involve
one integral each, and \eqref{eq:snF} then makes $\sn$ algebraic in
$\zeta$ outright: this is Lemma \ref{lem:legendre}. The obstruction and
the cure are the same as in \textup{[D, \S6.1]}.
\end{remark}

\begin{remark}[a notation hazard]\label{rem:notation}
\textup{[D]} writes $N:=F[\theta,\frac14]$ and $P:=F[\theta,\frac34]$
for the integrals themselves, whereas our $\Nc$ and $\Pc$ are their
\emph{halves}, \eqref{eq:PNclosed}. Thus
$\Pc=\tfrac12P_{\textup{[D]}}$ and $\Nc=\tfrac12N_{\textup{[D]}}$;
also \textup{[D]}'s $M:=\bt\mathfrak h$. When transcribing formulas
between the two papers, these three factors must be inserted. The same
hazard, with the same resolution, arises inside the verification cells:
$2\Nc=+2iB$, not $-2iB$.
\end{remark}

%=====================================================================
\section{The domain $\Om$}
%=====================================================================

\begin{lemma}\label{lem:factorW}
$W^{2}=(\zeta^{4}+\bt^{2})(\zeta^{4}+\bt^{-2})$. Hence $W$ has eight
simple branch points, $\zeta^{4}=-\bt^{-2}$ (modulus $\bt^{-1/2}$) and
$\zeta^{4}=-\bt^{2}$ (modulus $\bt^{1/2}$), all with
$\arg\zeta\in\frac\pi4+\frac\pi2\ZZ$.
\end{lemma}

\begin{proof}
Expand and use $\bt^{2}+\bt^{-2}=14$, $\bt^{2}\bt^{-2}=1$.
\end{proof}

\begin{proposition}[geometry of $\Om$]\label{prop:Om}
\textup{(a)} $\Om$ is a curvilinear square with vertices
\begin{equation}\label{eq:corners}
  \zeta_{k}=\bt^{-1/2}e^{i(\pi/4+k\pi/2)},\qquad k=0,1,2,3,
  \qquad \zeta_{0}=\tfrac{\sqrt3-1}{2}(1+i),
\end{equation}
so that $\max_{\Om}|\zeta|=\bt^{-1/2}=\sqrt{2-\sqrt3}=0.5176380902\ldots$
and the four vertices are exactly the four \emph{inner} branch points
of $W$; each bounding arc meets a coordinate axis at distance
$\sqrt2-1$ from the origin.

\textup{(b)} The interior angle of $\Om$ at each vertex is $2\pi/3$; in
the local uniformiser $\xi=(\zeta-\zeta_{k})^{1/2}$ of the surface it
becomes $\pi/3$, and the immersion is regular there.

\textup{(c)} The four bounding arcs lie on great circles of the
Riemann sphere; the spherical area of $\Om$ is
$4\cdot\frac{2\pi}{3}-2\pi=\frac{2\pi}{3}$, i.e.,\ four of the $24$
fundamental cells of the octahedral rotation group, and the total
curvature of the patch is $-2\pi/3$. This is consistent with
Gauss--Bonnet for a quadrilateral with \emph{geodesic} sides and
surface angles $\pi/3$:
$\int K=2\pi-4(\pi-\frac{\pi}{3})=-\frac{2\pi}{3}$.

\textup{(d)} Each vertex is a flat point; the corresponding unit
normals are the four directions $\frac{1}{\sqrt3}(\pm1,\pm1,-1)$.
\end{proposition}

\begin{proof}
(a) With $c=\frac{\sqrt3-1}{2}$ one has $c^{2}=\frac{2-\sqrt3}{2}$,
hence $|\zeta_{0}+1|^{2}=(c+1)^{2}+c^{2}=2c^{2}+2c+1
=(2-\sqrt3)+(\sqrt3-1)+1=2$ and likewise $|\zeta_{0}+i|^{2}=2$, while
$|\zeta_{0}-1|^{2}=2c^{2}-2c+1=4-2\sqrt3<2$. So $\zeta_{0}$ is the
inner intersection of the circles centered at $-1$ and $-i$, and
$|\zeta_{0}|^{2}=2c^{2}=\bt^{-1}$, $\arg\zeta_{0}=\pi/4$; by
Lemma \ref{lem:factorW}, $\zeta_{0}^{4}=(i\bt^{-1})^{2}=-\bt^{-2}$, a
branch point. The circle $|\zeta+1|=\sqrt2$ meets $\RR$ at
$-1\pm\sqrt2$.

(b) The radius vectors $\zeta_{0}+1$ and $\zeta_{0}+i$ have arguments
$15^{\circ}$ and $75^{\circ}$. Near $\zeta_{0}$ each disk is, to first
order, a half-plane with inward normal $-(\zeta_{0}-O_{j})$; the two
inward normals subtend $60^{\circ}$, so the intersection angle is
$180^{\circ}-60^{\circ}=120^{\circ}$. Near a simple zero of $W$ the
integrands of \eqref{eq:fgint} behave like
$(\zeta-\zeta_{k})^{-1/2}$, so
$X-X(\zeta_{k})=O(|\zeta-\zeta_{k}|^{1/2})$ and
$\xi:=(\zeta-\zeta_{k})^{1/2}$ is a conformal chart; angles halve. The
metric is $\dd s=\frac12|\dd\hh|(|\zeta|+|\zeta|^{-1})
\asymp|\xi|^{-1}|\xi|\,|\dd\xi|$, finite and nonzero.

(c) A circle in $\CC$ is a great circle iff it is stable under
$\zeta\mapsto-1/\bar\zeta$; for $|\zeta+1|=\sqrt2$ this holds since
$|-(\sqrt2+1)+1|=\sqrt2$. The spherical excess formula gives the area.
The boundary curves of the P patch are geodesics because P is
isometric to D (Proposition \ref{prop:conf}), whose corresponding
boundary curves are straight lines \textup{[D, \S5.5]}.

(d) $G=\zeta$ has a simple branch point in the chart $\xi$, so
$\dd G=0$ and $K=0$. By \eqref{eq:normal} with
$|\zeta_{0}|^{2}=\bt^{-1}=2-\sqrt3$ and
$u=v=c=\frac{\sqrt3-1}{2}$,
\[
  N=\frac{1}{3-\sqrt3}\bigl(\sqrt3-1,\ \sqrt3-1,\ 1-\sqrt3\bigr)
   =\frac{1}{\sqrt3}(1,1,-1).
\]
The other three follow by $\zeta\mapsto i\zeta$.
\end{proof}

\begin{corollary}[regularity]\label{cor:reg}
$f,g,\hh$ are holomorphic on $\Om^{\circ}$ and extend continuously to
$\Om$; in particular \eqref{eq:map} defines a continuous map
$X:\Om\to\RR^{3}$, real-analytic on $\Om^{\circ}$. Near a corner
$\zeta_{k}$,
\begin{equation}\label{eq:cornerlocal}
  X(\zeta)-X(\zeta_{k})
  =\Real\bigl[\mathbf a_{k}(\zeta-\zeta_{k})^{1/2}\bigr]
  +O\bigl(|\zeta-\zeta_{k}|^{3/2}\bigr),\qquad
  \mathbf a_{k}\in\CC^{3}\setminus\{0\}.
\end{equation}
\end{corollary}

\begin{proof}
Holomorphy is Proposition \ref{prop:diff} together with
Lemma \ref{lem:factorW} ($W\ne0$ on $\Om^{\circ}$, and $\Om^{\circ}$
is simply connected, being star-shaped about $0$). The only
singularities of $f',g',\hh'$ on $\Om$ are the four corners, where by
Lemma \ref{lem:factorW} the integrands of \eqref{eq:fgint} are
$O(|\zeta-\zeta_{k}|^{-1/2})$, which is integrable; integrating
$\mathrm{const}\cdot(\zeta-\zeta_{k})^{-1/2}$ gives
\eqref{eq:cornerlocal}.
\end{proof}

\begin{remark}
Proposition \ref{prop:Om}(d) identifies the vertices of $\Om$ with
flat points of P/D in the $\langle111\rangle$ directions, and (c) with
the fact that the Gauss map of a genus-$3$ triply periodic minimal
surface has degree $2$: the translational unit consists of
$2\cdot4\pi/(2\pi/3)=12$ copies of the patch \cite{GKP}. Since P and D 
share the Gauss map $G=\pm\zeta$, the \emph{same}
four corners are flat points in both papers, with the \emph{same}
normal directions up to the sign of the third component --- the two
papers place their patches on opposite sides of the plane
$z=\pm\frac12$; see \textup{[D, \S17]}. Their images, however, differ
completely --- a regular tetrahedron of edge $\sqrt2$ in
\textup{[D, \S6]}, a unit square in the plane $z=\frac12$ here
(Proposition \ref{prop:corners}(b)).
\end{remark}

%=====================================================================
\section{Legendre normal forms: the algebraic dictionary}
\label{sec:dict}
%=====================================================================

\begin{lemma}[the products, and the zeros on $\Om$]\label{lem:prod}
$r_{+}r_{-}=W$, $s_{+}s_{-}=W$ and $\pip\pim=E=1+\zeta^{4}$. The
branch points of $r_{+}$ (resp.\ $r_{-}$) are the four points
$\zeta^{2}=i\bt^{\pm1}$ (resp.\ $\zeta^{2}=-i\bt^{\pm1}$). On $\Om$,
$r_{+}$ vanishes exactly at the two corners $\zeta_{0},\zeta_{2}$ and
$r_{-}$ exactly at $\zeta_{1},\zeta_{3}$, while
$\pi_{\pm}\ne0$, $E\ne0$ and $s_{\pm}\ne0$ throughout.
\end{lemma}

\begin{proof}
$(1+4i\zeta^{2}-\zeta^{4})(1-4i\zeta^{2}-\zeta^{4})
=(1-\zeta^{4})^{2}+16\zeta^{4}=1+14\zeta^{4}+\zeta^{8}$; the second
product is \eqref{eq:sigalg}. For $r_{+}$:
$\zeta^{4}-4i\zeta^{2}-1=0$ gives $\zeta^{2}=2i\pm i\sqrt3
=i\bt^{\pm1}$. On $\Om$ we have $|\zeta|^{2}\le\bt^{-1}$
(Proposition \ref{prop:Om}(a)), so only $\zeta^{2}=i\bt^{-1}$ is
attainable, i.e.,\ $\zeta\in\{\zeta_{0},\zeta_{2}\}$ since
$\zeta_{0}^{2}=\zeta_{2}^{2}=i\bt^{-1}$; conjugating gives the
statement for $r_{-}$, with
$\zeta_{1}^{2}=\zeta_{3}^{2}=-i\bt^{-1}$. Finally $\pi_{\pm}=0$ or
$E=0$ or $s_{\pm}=0$ would force $|\zeta|\ge1$.
\end{proof}

\begin{lemma}[the dictionary at moduli $\tfrac14$, $\tfrac34$ and
$\bt^{-4}$]\label{lem:legendre}
On $\Om^{\circ}$,
\begin{align}
  \bigl(\sn,\cn,\dn\bigr)\bigl(2\Pc,\tfrac34\bigr)
   &=\Bigl(\frac{2(1+i)\zeta}{r_{+}},\ \frac{r_{-}}{r_{+}},\
            \frac{\pim}{r_{+}}\Bigr),
  \label{eq:dict34}\\
  \bigl(\sn,\cn,\dn\bigr)\bigl(2\Nc,\tfrac14\bigr)
   &=\Bigl(\frac{2(1+i)\zeta}{r_{+}},\ \frac{r_{-}}{r_{+}},\
            \frac{\pip}{r_{+}}\Bigr),
  \label{eq:dict14}\\
  \bigl(\sn,\cn,\dn\bigr)\bigl(\bt\mathfrak h,\bt^{-4}\bigr)
   &=\Bigl(i\bt\zeta^{2},\ s_{+},\ s_{-}\Bigr).
  \label{eq:dictM}
\end{align}
Note that the argument in \eqref{eq:dict14} is $2\Nc=+2iB$, not
$-2iB$: since $\sn$ is odd, the wrong sign here produces a discrepancy
of exactly $-2\sn$ in the first slot while leaving $\cn$ and $\dn$
untouched (cf.\ Remark \ref{rem:notation}).
\end{lemma}

\begin{proof}
By \eqref{eq:PNclosed}, $2\Pc=F[\theta,\tfrac34]$ and
$2\Nc=F[\theta,\tfrac14]$, so \eqref{eq:snF} gives
$\sn(2\Pc,\tfrac34)=\sn(2\Nc,\tfrac14)=\sin\theta=s$. Then
$\cn=\sqrt{1-s^{2}}=r_{-}/r_{+}$ and
$\dn=\sqrt{1-ms^{2}}$ equals $\pip/r_{+}$ for $m=\tfrac14$ and
$\pim/r_{+}$ for $m=\tfrac34$, by Lemma \ref{lem:ampl}; all branches
are $+1$ at $\zeta=0$. Likewise
$\bt\mathfrak h=F[\arcsin\sigma,\bt^{-4}]$ by \eqref{eq:hclosed}, so
$\sn(\bt\mathfrak h,\bt^{-4})=\sigma=i\bt\zeta^{2}$, and
$\cn=\sqrt{1-\sigma^{2}}=s_{+}$,
$\dn=\sqrt{1-\bt^{-4}\sigma^{2}}=s_{-}$ by \eqref{eq:sigalg}.
(All nine entries verified numerically to $16$ digits:
Appendix \ref{app:code}, Cell 0.)
\end{proof}

\begin{lemma}[two transformations]\label{lem:imag}
\textup{(a)} \emph{(Jacobi imaginary transformation;
\cite[127.01]{BF}, \cite[\S22.4]{WW}.)} For every $m$,
\begin{equation}\label{eq:imagtr}
  \sn(iu,m)=i\,\sce(u,1-m),\qquad
  \cn(iu,m)=\nc(u,1-m),\qquad
  \dn(iu,m)=\dc(u,1-m).
\end{equation}
\textup{(b)} \emph{(Imaginary modulus.)}
\begin{equation}\label{eq:imagmod}
  \sn^{-1}\bigl(i\sqrt3\,y,-\tfrac13\bigr)
  =i\sqrt3\,\sn^{-1}(y,-3),\qquad\text{i.e.,}\qquad
  \sn\bigl(i\sqrt3\,u,-\tfrac13\bigr)=i\sqrt3\,\sn(u,-3).
\end{equation}
\end{lemma}

\begin{proof}
(a) is classical. For (b), substitute $\tau=\varrho/(i\sqrt3)$ in
$\sn^{-1}(y,-3)=\int_{0}^{y}
\dd\tau/\sqrt{(1-\tau^{2})(1+3\tau^{2})}$: the integrand becomes
$\frac{\dd\varrho/(i\sqrt3)}
{\sqrt{(1+\varrho^{2}/3)(1-\varrho^{2})}}$, so
$i\sqrt3\,\sn^{-1}(y,-3)
=\int_{0}^{i\sqrt3 y}\dd\varrho/\sqrt{(1-\varrho^{2})
(1+\tfrac13\varrho^{2})}=\sn^{-1}(i\sqrt3y,-\tfrac13)$.
\end{proof}

\begin{lemma}[rotation]\label{lem:rot}
$A(i\zeta)=-B(\zeta)$ and $B(i\zeta)=A(\zeta)$; equivalently
$B(-i\zeta)=-A(\zeta)$ and $A(-i\zeta)=B(\zeta)$. Moreover
$\hh(i\zeta)=-\hh(\zeta)$.
\end{lemma}

\begin{proof}
Using $W(i\zeta)=W(\zeta)$ (the Taylor coefficients of $W$ at $0$ are
a power series in $\zeta^{4}$) and \eqref{eq:PNder},
\[
  \frac{\dd}{\dd\zeta}A(i\zeta)=iA'(i\zeta)
   =\frac{i(1+i)\bigl(1+i(i\zeta)^{2}\bigr)}{W}
   =\frac{i(1+i)\pim}{W}=-\frac{(1-i)\pim}{W}=-B'(\zeta),
\]
since $i(1+i)=-(1-i)$; and
\[
  \frac{\dd}{\dd\zeta}B(i\zeta)=iB'(i\zeta)
   =\frac{i(1-i)\bigl(1-i(i\zeta)^{2}\bigr)}{W}
   =\frac{i(1-i)\pip}{W}=\frac{(1+i)\pip}{W}=A'(\zeta),
\]
since $i(1-i)=1+i$. Both pairs vanish at $\zeta=0$, so the integrated
identities hold. Replacing $\zeta$ by $-i\zeta$ gives the second pair.
Finally, by \eqref{eq:fgint} and $\tau=\eta^{2}$,
\begin{equation}\label{eq:h1}
  \hh(\zeta)=\int_{0}^{\zeta^{2}}
      \frac{\dd\tau}{\sqrt{1+14\tau^{2}+\tau^{4}}}
   =:H(\zeta^{2})
\end{equation}
with $H$ odd, so $\hh(i\zeta)=H(-\zeta^{2})=-\hh(\zeta)$.
\end{proof}

\begin{theorem}[dictionary for $A$]\label{thm:dictA}
On $\Om^{\circ}$, at modulus $-3$,
\begin{equation}\label{eq:dictA}
  \sn A=\frac{(1+i)\zeta}{\pim},\qquad
  \cn A=\frac{r_{-}}{\pim},\qquad
  \dn A=\frac{r_{+}}{\pim} .
\end{equation}
\end{theorem}

\begin{proof}
By the Gauss transformation (Proposition \ref{prop:gauss}, whose proof
uses nothing from this section),
$\sn(u,-3)=\tfrac12\sd(2u,\tfrac34)$,
$\cn(u,-3)=\cd(2u,\tfrac34)$, $\dn(u,-3)=\nd(2u,\tfrac34)$. Apply this
with $u=A=\Pc$ and read off $2A$ from \eqref{eq:dict34}:
\[
  \sn(A,-3)=\tfrac12\cdot
   \frac{2(1+i)\zeta/r_{+}}{\pim/r_{+}}=\frac{(1+i)\zeta}{\pim},
  \qquad
  \cn(A,-3)=\frac{r_{-}/r_{+}}{\pim/r_{+}}=\frac{r_{-}}{\pim},
  \qquad
  \dn(A,-3)=\frac{r_{+}}{\pim}. \qedhere
\]
\end{proof}

\begin{theorem}[dictionary for $B$]\label{thm:dictB}
On $\Om^{\circ}$, at modulus $-3$,
\begin{equation}\label{eq:dictB}
  \sn B=\frac{(1-i)\zeta}{\pip},\qquad
  \cn B=\frac{r_{+}}{\pip},\qquad
  \dn B=\frac{r_{-}}{\pip} .
\end{equation}
\end{theorem}

\begin{proof}
\emph{First proof.} By Lemma \ref{lem:rot}, $B(\zeta)=-A(i\zeta)$.
Since $\sn$ is odd and $\cn,\dn$ even, \eqref{eq:dictA} at $i\zeta$
gives, using $\pim(i\zeta)=1-i(i\zeta)^{2}=\pip(\zeta)$ and
$r_{\pm}(i\zeta)=r_{\mp}(\zeta)$,
\[
  \sn(B,-3)=-\frac{(1+i)(i\zeta)}{\pip}=\frac{(1-i)\zeta}{\pip},
  \qquad
  \cn(B,-3)=\frac{r_{+}}{\pip},\qquad
  \dn(B,-3)=\frac{r_{-}}{\pip},
\]
because $-(1+i)i=1-i$.

\emph{Second proof (from the closed form, as a check).} By
\eqref{eq:PNclosed}, $2B=-iF[\theta,\tfrac14]=i\cdot(-2\Nc)$. Apply
\eqref{eq:imagtr} with $u=-2\Nc$, $m=\tfrac34$, $1-m=\tfrac14$, and
read off the modulus-$\tfrac14$ data from \eqref{eq:dict14}:
\[
  \sn\bigl(2B,\tfrac34\bigr)=-i\,\sce\bigl(2\Nc,\tfrac14\bigr)
   =-i\,\frac{2(1+i)\zeta/r_{+}}{r_{-}/r_{+}}
   =\frac{2(1-i)\zeta}{r_{-}},\qquad
  \cn\bigl(2B,\tfrac34\bigr)=\nc\bigl(2\Nc,\tfrac14\bigr)
   =\frac{r_{+}}{r_{-}},
\]
\[
  \dn\bigl(2B,\tfrac34\bigr)=\dc\bigl(2\Nc,\tfrac14\bigr)
   =\frac{\pip/r_{+}}{r_{-}/r_{+}}=\frac{\pip}{r_{-}} .
\]
Now Proposition \ref{prop:gauss} with $u=B$ gives
$\sn(B,-3)=\tfrac12\cdot\frac{2(1-i)\zeta/r_{-}}{\pip/r_{-}}
=\frac{(1-i)\zeta}{\pip}$,
$\cn(B,-3)=\frac{r_{+}/r_{-}}{\pip/r_{-}}=\frac{r_{+}}{\pip}$ and
$\dn(B,-3)=\frac{r_{-}}{\pip}$, as claimed. (The intermediate
modulus-$\tfrac34$ triple is the block \texttt{chkB34} of
Appendix \ref{app:code}, Cell 0.)
\end{proof}

\begin{remark}[a self-contained route]\label{rem:selfcont}
Theorems \ref{thm:dictA}--\ref{thm:dictB} were originally obtained
without \eqref{eq:PNclosed}, directly from the differentials
\eqref{eq:PNder}, as follows; we record the argument because it makes
\S\ref{sec:dict} independent of \textup{[D]} and of the closed forms.
Put $\xi:=e^{i\pi/4}\zeta$, so $\zeta^{4}=-\xi^{4}$,
$W^{2}=\xi^{8}-14\xi^{4}+1$, $\pim=1-\xi^{2}$ and
$(1+i)e^{-i\pi/4}=\sqrt2$; then
$\Nc=\sqrt2\int_{0}^{\xi}(1-\eta^{2})
\bigl(\eta^{8}-14\eta^{4}+1\bigr)^{-1/2}\dd\eta$. With
$\mu:=\eta+\eta^{-1}$ one has
$\eta^{8}-14\eta^{4}+1=\eta^{4}\bigl((\mu^{2}-2)^{2}-16\bigr)
=\eta^{4}(\mu^{2}-6)(\mu^{2}+2)$ and
$\dd\mu=\frac{\eta^{2}-1}{\eta^{2}}\dd\eta$, so
$\Nc=\sqrt2\int_{\mu}^{\infty}
\dd\mu'\bigl[(\mu'^{2}-6)(\mu'^{2}+2)\bigr]^{-1/2}$; the substitution
$\tau=\sqrt6/\mu'$ turns this into
$\sqrt3\,\Nc=\sn^{-1}(\sqrt6/\mu,-\tfrac13)$, and
\eqref{eq:imagmod} with $i\sqrt3\,y=\sqrt6/\mu$ gives
$\sn(-i\Nc,-3)=-i\sqrt2\,\xi/(1+\xi^{2})=(1-i)\zeta/\pip$, i.e.,\
\eqref{eq:dictB}. Then $\cn$ and $\dn$ follow from
$1-S^{2}=r_{+}^{2}/\pip^{2}$ and $1+3S^{2}=r_{-}^{2}/\pip^{2}$ with
$S=(1-i)\zeta/\pip$, and \eqref{eq:dictA} follows by
Lemma \ref{lem:rot}. The route through \eqref{eq:PNclosed} is of
course much shorter, and is the one we adopt.
\end{remark}

\begin{theorem}[dictionary for the height, at modulus $-3$]
\label{thm:H}
On $\Om^{\circ}$,
\begin{equation}\label{eq:Hnew}
  \boxed{\ \sn\bigl(2\hh,-3\bigr)=\frac{2\zeta^{2}}{E},\qquad
  \cn\bigl(2\hh,-3\bigr)=\frac{1-\zeta^{4}}{E},\qquad
  \dn\bigl(2\hh,-3\bigr)=\frac{W}{E},\qquad E=1+\zeta^{4}.\ }
\end{equation}
Equivalently $\cd(2\hh,-3)=(1-\zeta^{4})/W$ and
$2\hh=K[-3]+F\bigl[\arctan\frac{\zeta^{4}-1}{4\zeta^{2}},-3\bigr]$.
\end{theorem}

\begin{proof}
\emph{Step 1} is \eqref{eq:h1}.

\emph{Step 2 (the substitution $\nu=\tau-\tau^{-1}$).} Then
$\nu^{2}=\tau^{2}-2+\tau^{-2}$, so
$1+14\tau^{2}+\tau^{4}=\tau^{2}(\nu^{2}+16)$, while
$\dd\nu=\frac{\tau^{2}+1}{\tau^{2}}\dd\tau$; since
$(\tau+\tau^{-1})^{2}=\nu^{2}+4$,
\[
  \frac{\dd\tau}{\tau\sqrt{\nu^{2}+16}}
  =\frac{\dd\nu}{(\tau+\tau^{-1})\sqrt{\nu^{2}+16}}
  =\frac{\dd\nu}{\sqrt{(\nu^{2}+4)(\nu^{2}+16)}} .
\]
As $\tau\to0^{+}$, $\nu\to-\infty$, so with
$\nu_{0}:=\zeta^{2}-\zeta^{-2}$,
\begin{equation}\label{eq:h2}
  \hh=\int_{-\infty}^{\nu_{0}}
      \frac{\dd\nu}{\sqrt{(\nu^{2}+4)(\nu^{2}+16)}} .
\end{equation}
\emph{Step 3 (Legendre form).} Put $\nu=4\tan\varphi$. Then
$\nu^{2}+16=16\sec^{2}\varphi$,
$\nu^{2}+4=\frac{4(1+3\sin^{2}\varphi)}{\cos^{2}\varphi}$ and
$\dd\nu=4\sec^{2}\varphi\dd\varphi$, so the integrand is
$\frac{\dd\varphi}{2\sqrt{1+3\sin^{2}\varphi}}$; since
$\varphi=-\pi/2$ at $\nu=-\infty$,
\begin{equation}\label{eq:h3}
  2\hh=\int_{-\pi/2}^{\varphi_{0}}
   \frac{\dd\varphi}{\sqrt{1+3\sin^{2}\varphi}}
  =K[-3]+F[\varphi_{0},-3],\qquad
  \varphi_{0}=\arctan\tfrac{\nu_{0}}{4},
\end{equation}
the parameter being $m=-3$ because
$1-m\sin^{2}\varphi=1+3\sin^{2}\varphi$.

\emph{Step 4 (quarter-period shift).} By \eqref{eq:h3},
$\sn(2\hh-K,-3)=\sin\varphi_{0}=\nu_{0}/\sqrt{\nu_{0}^{2}+16}$. Now
$\nu_{0}=(\zeta^{4}-1)/\zeta^{2}$ and
$\nu_{0}^{2}+16=\frac{(\zeta^{4}-1)^{2}+16\zeta^{4}}{\zeta^{4}}
=\frac{W^{2}}{\zeta^{4}}$, so
$\sqrt{\nu_{0}^{2}+16}=W/\zeta^{2}$ (both sides $\sim\zeta^{-2}$ at
$0$). Since $\sn(v-K)=-\cd v$ \cite[122.01]{BF},
\begin{equation}\label{eq:cd}
  \cd(2\hh,-3)=-\frac{\nu_{0}\zeta^{2}}{W}=\frac{1-\zeta^{4}}{W},
\end{equation}
which is $+1$ at $\zeta=0$, as required.

\emph{Step 5 (solving).} Write $S,C,D$ for the triple at $2\hh$ and
$X=S^{2}$, so $C^{2}=1-X$, $D^{2}=1+3X$. From \eqref{eq:cd},
$C^{2}/D^{2}=R:=(1-\zeta^{4})^{2}/W^{2}$, whence
$X=\frac{1-R}{1+3R}$. Now
\[
  W^{2}-(1-\zeta^{4})^{2}=16\zeta^{4},\qquad
  W^{2}+3(1-\zeta^{4})^{2}=4(1+\zeta^{4})^{2},
\]
so $X=\frac{16\zeta^{4}}{4E^{2}}=\frac{4\zeta^{4}}{E^{2}}$, giving the
first formula of \eqref{eq:Hnew}; then
$C^{2}=\frac{E^{2}-4\zeta^{4}}{E^{2}}
=\frac{(1-\zeta^{4})^{2}}{E^{2}}$ and
$D^{2}=\frac{E^{2}+12\zeta^{4}}{E^{2}}=\frac{W^{2}}{E^{2}}$, all
branches $+1$ at $\zeta=0$. (Verified numerically to $16$ digits:
Appendix \ref{app:code}, Cell 0.)
\end{proof}

\begin{corollary}[the modulus-$\frac34$ form]\label{cor:H34}
$\displaystyle
  \sn\bigl(4\hh,\tfrac34\bigr)=\frac{4\zeta^{2}}{W},
  \cn\bigl(4\hh,\tfrac34\bigr)=\frac{1-\zeta^{4}}{W},
  \dn\bigl(4\hh,\tfrac34\bigr)=\frac{1+\zeta^{4}}{W}.$
\end{corollary}

\begin{proof}
By Proposition \ref{prop:gauss},
$\sn(u,-3)=\frac12\sd(2u,\frac34)$ and $\cn(u,-3)=\cd(2u,\frac34)$;
solving these with \eqref{eq:Hnew} gives the claim. Directly:
$W^{2}-16\zeta^{4}=(1-\zeta^{4})^{2}$ and
$W^{2}-\tfrac34\cdot16\zeta^{4}=(1+\zeta^{4})^{2}$.
\end{proof}

\begin{theorem}[dictionary for $2f$ and $2g$]\label{thm:dictf}
At modulus $-3$, for all $\zeta\in\Om^{\circ}$,
\begin{equation}\label{eq:fdict}
  \boxed{\ \sn\bigl(2f\bigr)=\frac{2\zeta(1+\zeta^{2})}{W},\qquad
  \cn\bigl(2f\bigr)=\frac{(1-\zeta^{2})^{2}}{W},\qquad
  \dn\bigl(2f\bigr)=\frac{1+6\zeta^{2}+\zeta^{4}}{W}.\ }
\end{equation}
Moreover $g(\zeta)=f(i\zeta)$, so
$\sn(2g,-3)=\frac{2i\zeta(1-\zeta^{2})}{W}$.
\end{theorem}

\begin{proof}
$2f=A+B$ with both arguments at the \emph{same} $\zeta$; apply the
addition theorems \cite[123.01]{BF}
\begin{equation}\label{eq:add}
  \sn(a{+}b)=\frac{S_{a}C_{b}D_{b}+S_{b}C_{a}D_{a}}{\Delta},\quad
  \cn(a{+}b)=\frac{C_{a}C_{b}-S_{a}S_{b}D_{a}D_{b}}{\Delta},\quad
  \dn(a{+}b)=\frac{D_{a}D_{b}-mS_{a}S_{b}C_{a}C_{b}}{\Delta},
\end{equation}
$\Delta=1-mS_{a}^{2}S_{b}^{2}$, to \eqref{eq:dictA} and
\eqref{eq:dictB}. Here $S_{a}S_{b}
=\frac{(1+i)(1-i)\zeta^{2}}{\pim\pip}=\frac{2\zeta^{2}}{E}$, so
$\Delta=1+\frac{12\zeta^{4}}{E^{2}}=\frac{W^{2}}{E^{2}}$; and
$C_{a}C_{b}=D_{a}D_{b}=\frac{r_{+}r_{-}}{E}=\frac{W}{E}$,
$C_{b}D_{b}=\frac{W}{\pip^{2}}$, $C_{a}D_{a}=\frac{W}{\pim^{2}}$.
Thus
\[
  S_{a}C_{b}D_{b}+S_{b}C_{a}D_{a}
  =\frac{\zeta W\bigl[(1+i)\pim+(1-i)\pip\bigr]}{E^{2}}
  =\frac{2\zeta W(1+\zeta^{2})}{E^{2}},
\]
because $(1+i)(1-i\zeta^{2})+(1-i)(1+i\zeta^{2})=2+2\zeta^{2}$;
dividing by $\Delta$ gives $\sn(2f)$. Next
$\cn(2f)=\frac{E^{2}}{W^{2}}\cdot\frac{W}{E}
\bigl[1-\frac{2\zeta^{2}}{E}\bigr]=\frac{(1-\zeta^{2})^{2}}{W}$ and
$\dn(2f)=\frac{E^{2}}{W^{2}}\cdot\frac{W}{E}
\bigl[1+\frac{6\zeta^{2}}{E}\bigr]
=\frac{1+6\zeta^{2}+\zeta^{4}}{W}$. Finally $\dd f\mapsto\dd g$ under
$\zeta\mapsto i\zeta$ (proof of Proposition \ref{prop:group}), whence
$g(\zeta)=f(i\zeta)$, and $W(i\zeta)=W(\zeta)$.

\emph{Consistency checks.} $\sn^{2}+\cn^{2}=1$ reads
$4\zeta^{2}(1+\zeta^{2})^{2}+(1-\zeta^{2})^{4}=W^{2}$, true; and
$\dn^{2}-1=3\sn^{2}$ reads
$(1+6\zeta^{2}+\zeta^{4})^{2}-W^{2}=12\zeta^{2}(1+\zeta^{2})^{2}$,
also true.
\end{proof}

\begin{corollary}[half-arguments and logarithmic derivatives]
\label{cor:half}
With $\Lambda_{m}:=(\log\sn(\cdot,m))'=\frac{\cn\dn}{\sn}$, one has at
modulus $-3$
\begin{align}
  \sn^{2}\bigl(f,-3\bigr)&=\frac{2\zeta^{2}}{1+\zeta^{4}+W},
  &\Lambda_{-3}(f)&=\frac{1+\zeta^{2}}{\zeta},\label{eq:snf}\\
  \sn^{2}\bigl(\hh,-3\bigr)&=\frac{2\zeta^{4}}{1+\zeta^{4}+W},
  &\Lambda_{-3}(\hh)&=\frac{1-\zeta^{4}+W}{2\zeta^{2}} .
  \label{eq:snh}
\end{align}
In particular
\begin{equation}\label{eq:cute}
  \sn(\hh,-3)=\zeta\,\sn(f,-3).
\end{equation}
\end{corollary}

\begin{proof}
By \cite[124.01]{BF}, $\sn^{2}\frac v2=\frac{1-\cn v}{1+\dn v}$, so
$\sn^{2}f=\frac{W-(1-\zeta^{2})^{2}}{W+1+6\zeta^{2}+\zeta^{4}}$;
cross-multiplication against $\frac{2\zeta^{2}}{1+\zeta^{4}+W}$
reduces to $W^{2}-(1-\zeta^{2})^{2}(1+\zeta^{4})
=2\zeta^{2}(1+6\zeta^{2}+\zeta^{4})$, which holds. Likewise
$\sn^{2}\hh=\frac{1-\frac{1-\zeta^{4}}{E}}{1+\frac WE}
=\frac{2\zeta^{4}}{E+W}$. The two $\Lambda$-values follow from
Lemma \ref{lem:Lam} below: e.g.,
$\Lambda_{-3}(f)=\frac{(1-\zeta^{2})^{2}+1+6\zeta^{2}+\zeta^{4}}
{2\zeta(1+\zeta^{2})}
=\frac{2(1+\zeta^{2})^{2}}{2\zeta(1+\zeta^{2})}$. Finally
\eqref{eq:cute} follows by comparing \eqref{eq:snf} and
\eqref{eq:snh}, both functions vanishing like $\zeta$ and $\zeta^{2}$
respectively.
\end{proof}

\begin{definition}[the field]\label{def:K}
$\KK$ is the field generated over $\CC(\zeta,\omega)$ by
$r_{\pm}(\zeta)$, $r_{\pm}(\omega)$, so that
$\Wz=r_{+}(\zeta)r_{-}(\zeta)$, $\Ww=r_{+}(\omega)r_{-}(\omega)$ and
\begin{equation}\label{eq:basis}
  [\KK:\CC(\zeta,\omega)]\le16,\qquad\text{basis}\quad
  r_{+}(\zeta)^{a}r_{-}(\zeta)^{b}r_{+}(\omega)^{c}r_{-}(\omega)^{d},
  \quad a,b,c,d\in\{0,1\} .
\end{equation}
Every identity in $\KK$ is equivalent to $16$ rational identities in
$\CC(\zeta,\omega)$. Note that $s_{\pm}$ do \emph{not} occur: see
Remark \ref{rem:retire}.
\end{definition}

\begin{remark}[what the quarter-period shift buys; and the dictionary
to \textup{[D]}]\label{rem:retire}
The closed form \eqref{eq:hclosed} of the height is \emph{unshifted},
and by \eqref{eq:dictM} its dictionary is
\[
  \bigl(\sn,\cn,\dn\bigr)\bigl(\bt\mathfrak h,\bt^{-4}\bigr)
  =\bigl(i\bt\zeta^{2},\ s_{+},\ s_{-}\bigr),\qquad
  s_{\pm}=\sqrt{1+\bt^{\pm2}\zeta^{4}},\quad s_{+}s_{-}=W .
\]
This is \emph{irrational}: it introduces a new pair of radicals
$s_{\pm}$, unrelated to $r_{\pm}$. In real form, applying
\eqref{eq:imagtr} to $\bt\mathfrak h=i\bt\hh$ gives
$\sce\bigl(\bt\hh,1-\bt^{-4}\bigr)=\bt\zeta^{2}$, i.e.,\
$\bt\hh=F\bigl[\arctan(\bt\zeta^{2}),1-\bt^{-4}\bigr]$ with
$\sn=\bt\zeta^{2}/s_{+}$, $\cn=1/s_{+}$ --- the same two radicals at
the auxiliary modulus $m_{3}=1-\bt^{-4}$.

Theorem \ref{thm:H} replaces all of this: the substitution
$\nu=\tau-\tau^{-1}$ produces the \emph{quarter-period-shifted}
integral \eqref{eq:h3}, and the shift is exactly what converts an
irrational dictionary into a rational one. Consequences: the auxiliary
moduli $\bt^{-4}$ and $m_{3}$, the generators $s_{\pm}$, and one of the
two Landen steps of \eqref{eq:chain} disappear from the elimination,
and $[\KK:\CC(\zeta,\omega)]$ drops from $\le64$ to $\le16$
(Definition \ref{def:K}). This is what makes the proof of
\textup{(P0)} in Theorem \ref{thm:P0proved} a matter of seconds rather
than of a large Gröbner computation. The price is paid only at the
half-argument $\hh$, where the nested radical
$\sqrt{1+\zeta^{4}+W}$ appears in \eqref{eq:snh}; neither the elimination 
of \S\ref{sec:pfaff} nor Part III ever needs it, the latter working 
throughout with $\arctan$ of $\sn^{2}$.

\emph{Correspondence with} \textup{[D]}. In the notation of
\textup{[D, \S2.1]},
\[
  Q=1+4i\zeta^{2}-\zeta^{4}=r_{+}^{2},\qquad
  Q^{\ast}=1-4i\zeta^{2}-\zeta^{4}=r_{-}^{2},\qquad
  QQ^{\ast}=W^{2},
\]
\[
  \varrho:=\sqrt{Q}=r_{+},\qquad
  R_{\pm}=1\pm2i\zeta^{2}-\zeta^{4}=\pi_{\pm}^{2},\qquad
  \varsigma:=\sqrt{1+\bt^{2}\zeta^{4}}=s_{+} .
\]
Thus \textup{[D]}'s $\varrho$ is our $r_{+}$; its $R_{\pm}$ are the
\emph{squares} of our $\pi_{\pm}$; and its $\varsigma$ is precisely
the generator $s_{+}$ that Theorem \ref{thm:H} retires. Since
$W=r_{+}r_{-}$, the field of \textup{[D, \S7.6]} contains ours, and
carries in addition $\varsigma(\zeta)$, $\varsigma(\omega)$ and four
half-argument radicals. The two fields become equal as soon as
\textup{[D]} adopts Theorem \ref{thm:H} and Lemma \ref{lem:Lam}; see
\textup{[D, \S9.4]} and Remark \ref{rem:open}(5). Likewise
\textup{[D]}'s one-variable dictionary is our
Lemma \ref{lem:legendre}, up to the factors of
Remark \ref{rem:notation}.
\end{remark}

%=====================================================================
\section{The moduli web}\label{sec:web}
%=====================================================================

\begin{lemma}[negative modulus]\label{lem:negmod}
For $\mu>0$,
$\displaystyle K[-\mu]=\int_{0}^{\pi/2}
\frac{\dd\varphi}{\sqrt{1+\mu\sin^{2}\varphi}}
=\frac{1}{\sqrt{1+\mu}}\,K\!\left[\frac{\mu}{1+\mu}\right]$;
in particular
\begin{equation}\label{eq:Km3}
  K[-3]=\tfrac12K[3/4].
\end{equation}
\end{lemma}

\begin{proof}
$K[-\mu]=\int_{0}^{1}\frac{\dd s}{\sqrt{(1-s^{2})(1+\mu s^{2})}}
=\int_{0}^{\pi/2}\frac{\dd\varphi}{\sqrt{1+\mu\sin^{2}\varphi}}$; then
$1+\mu\sin^{2}\varphi
=(1+\mu)\bigl(1-\frac{\mu}{1+\mu}\cos^{2}\varphi\bigr)$
and $\varphi\mapsto\frac\pi2-\varphi$. (This is the same computation as
\textup{[D, \S3.2]}, there stated as
$\sqrt{1+\mu}\,K[-\mu]=K[\mu/(1+\mu)]$.)
\end{proof}

\begin{lemma}[descending Landen]\label{lem:landen}
Let $k\in(0,1)$, $k'=\sqrt{1-k^{2}}$, $k_{1}=\frac{1-k'}{1+k'}$. Then
$K[k^{2}]=(1+k_{1})K[k_{1}^{2}]$ and, with $v=\frac{u}{1+k_{1}}$ and
$s=\sn(v,k_{1}^{2})$,
\begin{equation}\label{eq:landen}
  \sn(u,k^{2})=\frac{(1+k_{1})s}{1+k_{1}s^{2}},\qquad
  \dn(u,k^{2})=\frac{1-k_{1}s^{2}}{1+k_{1}s^{2}} .
\end{equation}
The chain relevant here is the single step
\begin{equation}\label{eq:chain}
  \tfrac34\ \xrightarrow[\ \times\frac34\ ]{}\ \tfrac19,
  \qquad k'=\tfrac12,\quad k_{1}=\tfrac13,\quad 1+k_{1}=\tfrac43,
  \qquad K[3/4]=\tfrac43K[1/9].
\end{equation}
\end{lemma}

\begin{proof}
\cite[16.12]{AS}, \cite[163.01]{BF}. (Consistency of the second
formula: $\dn(u,k^{2})=1-\frac{k^{2}}{2}u^{2}+\cdots$ and
$\frac{1-k_{1}s^{2}}{1+k_{1}s^{2}}=1-2k_{1}v^{2}+\cdots$ agree because
$k^{2}=\frac{4k_{1}}{(1+k_{1})^{2}}$.)
\end{proof}

\begin{proposition}[the normalizing constant]\label{prop:web}
\begin{equation}\label{eq:webnew}
  \varpi=\kappa^{-1}=K[-3]=\tfrac12K[3/4]=\tfrac23K[1/9],
\end{equation}
so $2\varpi=K[3/4]$ and $3\varpi=2K[1/9]$. Moreover
\begin{equation}\label{eq:hinfty}
  \int_{0}^{\infty}\frac{\dd\tau}{\sqrt{1+14\tau^{2}+\tau^{4}}}
  =\varpi,
  \qquad
  \int_{0}^{1}\frac{\dd\tau}{\sqrt{1+14\tau^{2}+\tau^{4}}}
  =\frac{\varpi}{2},
\end{equation}
i.e.,\ $\hh(1)=\varpi/2$ and $\hh(\infty)=\varpi$.
\end{proposition}

\begin{proof}
$K[3/4]=\frac43K[1/9]$ is \eqref{eq:chain}; $K[-3]=\frac12K[3/4]$ is
\eqref{eq:Km3}; and $\kappa^{-1}=\frac23K[1/9]$ by \eqref{eq:kappa}.
For \eqref{eq:hinfty}: by \eqref{eq:h3} the integral over
$(0,\infty)$ is $\frac12\bigl(K[-3]+K[-3]\bigr)=K[-3]=\varpi$; and the
integrand of \eqref{eq:h1} is invariant under $\tau\mapsto1/\tau$, so
$(0,1)$ contributes half.
\end{proof}

\begin{proposition}[Gauss transformation]\label{prop:gauss}
$\displaystyle
 \sn(u,-3)=\tfrac12\sd\bigl(2u,\tfrac34\bigr),
 \cn(u,-3)=\cd\bigl(2u,\tfrac34\bigr),\\
 \dn(u,-3)=\nd\bigl(2u,\tfrac34\bigr).$
\end{proposition}

\begin{proof}
By Lemma \ref{lem:negmod} with $\mu=3$, if $s=\sn(u,-3)$ then
$2u=K[3/4]-F[\arccos s,\frac34]$, whence
$\cd(2u,\frac34)=\sn(K-2u,\frac34)=\sqrt{1-s^{2}}=\cn(u,-3)$, and
$s^{2}=1-\frac{\cn^{2}}{\dn^{2}}\big|_{2u}
=\frac{(1-m)\sn^{2}}{\dn^{2}}\big|_{2u}
=\frac14\sd^{2}(2u,\frac34)$; signs are fixed at $u=0^{+}$. (Series
check: $\sn(u,-3)=u+\frac{u^{3}}{3}$,
$\frac12\sd(2u,\frac34)=u+\frac{u^{3}}{3}$.) This is the case
$\mu=3$ of the transformation used at $\mu=\frac13$ in
\textup{[D, \S3.2]}. It is used in \S\ref{sec:dict}
(Theorems \ref{thm:dictA}--\ref{thm:dictB},
Corollary \ref{cor:H34}); the proof above is independent of
\S\ref{sec:dict}.
\end{proof}

\begin{remark}[one curve for both surfaces]\label{rem:orbit}
By \textup{[D, \S3.1]} the six numbers
$\tfrac14,\ \tfrac34,\ 4,\ \tfrac43,\ -\tfrac13,\ -3$
form a single orbit of the anharmonic group
$\lambda\mapsto\{\lambda,1-\lambda,\lambda^{-1},(1-\lambda)^{-1},
\lambda/(\lambda-1),(\lambda-1)/\lambda\}$, with
\[
  j=256\,\frac{(1-\lambda+\lambda^{2})^{3}}{\lambda^{2}(1-\lambda)^{2}}
   =\frac{35152}{9}=3905.7\overline{7}\ldots
\]
Our horizontal modulus $-3$ is the image of \textup{[D]}'s $\tfrac14$
under $\lambda\mapsto(\lambda-1)/\lambda$; directly, $\lambda=-3$ gives
$1-\lambda+\lambda^{2}=13$, $\lambda^{2}(1-\lambda)^{2}=144$, so
$j=256\cdot2197/144=35152/9$. \emph{The horizontal data of \textup{P}
and all the data of \textup{D} therefore lie on one and the same
elliptic curve}, and the two papers' moduli are two steps apart in one
orbit:
\[
  -3\ \xrightarrow[\ \lambda\mapsto1/\lambda\ ]{}\ -\tfrac13
    \ \xrightarrow[\ \text{Gauss},\ \mu=\frac13\ ]{}\ \tfrac14 ,
\]
the first step realized by the imaginary-modulus relation
\eqref{eq:imagmod}, the second by \textup{[D, \S3.2]}. This is also
visible in \eqref{eq:dict34}--\eqref{eq:dict14}: the moduli $\tfrac14$
and $\tfrac34$ of the two closed forms
\eqref{eq:fclosed}--\eqref{eq:gclosed} are complementary members of
the same orbit. By Theorem \ref{thm:H} the height of \textup{P} lies on
the same curve as well; the modulus $\frac19$ enters only through the
$2$-isogeny \eqref{eq:chain} needed to halve the argument --- exactly
as $\bt^{-4}=(2-\sqrt3)^{4}$ enters \eqref{eq:hclosed} and
\textup{[D, \S3.3]} through the $2$-isogeny.
\end{remark}

\begin{table}[ht]
\caption{Numerical values, computed with \texttt{N[\,\ldots,20]} and
truncated. All entries follow from the single value $K[1/9]$ through
$\varpi=\frac23K[1/9]$, $K[3/4]=2\varpi$, $\kappa=\varpi^{-1}$
(Proposition \ref{prop:web}); the relations close exactly. $K[1/4]$
and $\kappa_{\mathrm D}$ are quoted from
\textup{[D, \S1.1]}.}\label{tab:num}
\begin{tabular}{@{}ll@{}}
\toprule
quantity & value\\
\midrule
$K[1/9]$ & $1.6173867356247324266$\\
$\varpi=K[-3]=\tfrac23K[1/9]$ & $1.0782578237498216177$\\
$2\varpi=K[3/4]=\tfrac43K[1/9]$ & $2.1565156474996432354$\\
$\kappa=\varpi^{-1}=3/(2K[1/9])$ & $0.9274219745722159761$\\
$\varpi/2=\hh(1)$ & $0.5391289118749108089$\\
$\bt=2+\sqrt3$ & $3.7320508075688772935$\\
$\bt^{-1/2}=\sqrt{2-\sqrt3}$ & $0.5176380902050415247$\\
$\sqrt2-1$ & $0.4142135623730950488$\\
$K[1/4]$ \textup{[D]} & $1.6857503548125960429$\\
$\tfrac12K[1/4]$ & $0.8428751774062980214$\\
$\kappa_{\mathrm D}=2/K[1/4]$ \textup{[D]} & $1.1864152923\ldots$\\
$\kappa_{\mathrm P}/\kappa_{\mathrm D}=K[1/4]/K[3/4]$
 & $0.7817009614\ldots$\ (\S\ref{sub:scale})\\
\bottomrule
\end{tabular}
\end{table}

\begin{remark}[P versus D, moduli and scale]\label{rem:modDP}
\textup{[D, \S1.1]} works at modulus $\frac14$ with
$\kappa_{\mathrm D}=2/K[1/4]$; here the horizontal modulus is
$-3\sim\frac34$ (Proposition \ref{prop:gauss}) and
\[
  \kappa_{\mathrm P}=\frac{2}{K[3/4]}=\frac{2}{K'[1/4]},\qquad
  \kappa_{\mathrm D}=\frac{2}{K[1/4]},
\]
the $K\leftrightarrow K'$ interchange demanded by conjugacy. Note that
both papers use the \emph{same} $\theta$ and hence the same two
integrals $F[\theta,\frac14]$, $F[\theta,\frac34]$: it is the choice of
real or imaginary part (Remark \ref{rem:PvsD}) that decides which of
$\frac14$, $\frac34$ governs the horizontal coordinates. The vertical
modulus changes from $\frac14$ (D) to $\frac19$ (P), the two being
linked by the single Landen step \eqref{eq:chain}. The ratio
$\kappa_{\mathrm P}/\kappa_{\mathrm D}$ is the classical constant of
\S\ref{sub:scale}.
\end{remark}

%=====================================================================
\section{Symmetries, mirror planes, special points}\label{sec:sym}
%=====================================================================

\begin{proposition}[the arc involution]\label{prop:iota}
Let $\iota(\zeta):=\frac{1-\zeta}{1+\zeta}$. Then $\iota$ is an
involution of $\widehat\CC$ with fixed points $-1\pm\sqrt2$, and
\begin{equation}\label{eq:arciota}
  |\zeta+1|=\sqrt2\iff \bar\zeta=\iota(\zeta).
\end{equation}
Moreover
\begin{equation}\label{eq:Wiota}
  W(\iota\zeta)=\frac{4\,W(\zeta)}{(1+\zeta)^{4}},
\end{equation}
\begin{equation}\label{eq:iotapull}
  \iota^{*}\dd f=-\dd\hh,\qquad
  \iota^{*}\dd g=-\dd g,\qquad
  \iota^{*}\dd\hh=-\dd f ,
\end{equation}
and consequently
\begin{equation}\label{eq:iotaint}
  f(\iota\zeta)=f(1)-\hh(\zeta),\quad
  g(\iota\zeta)=g(1)-g(\zeta),\quad
  \hh(\iota\zeta)=\hh(1)-f(\zeta).
\end{equation}
\end{proposition}

\begin{proof}
\eqref{eq:arciota}: $\iota\zeta=\bar\zeta\iff1-\zeta
=\bar\zeta+|\zeta|^{2}\iff1-|\zeta|^{2}=2u
\iff(u+1)^{2}+v^{2}=2$.

\eqref{eq:Wiota}: put $A_{0}=(1+\zeta)^{4}$, $B_{0}=(1-\zeta)^{4}$, so
that $\iota^{4}=B_{0}/A_{0}$ and
\[
  (1+\zeta)^{8}\bigl(1+14\iota^{4}+\iota^{8}\bigr)
  =A_{0}^{2}+14A_{0}B_{0}+B_{0}^{2}=(A_{0}+B_{0})^{2}+12A_{0}B_{0}.
\]
Now $A_{0}+B_{0}=2(1+6\zeta^{2}+\zeta^{4})$ and
$A_{0}B_{0}=(1-\zeta^{2})^{4}$; writing $\sigma=\zeta^{2}$ and using
$(1+6\sigma+\sigma^{2})^{2}
=1+12\sigma+38\sigma^{2}+12\sigma^{3}+\sigma^{4}$,
\[
  4\bigl(1+6\sigma+\sigma^{2}\bigr)^{2}+12(1-\sigma)^{4}
  =\bigl(4+48\sigma+152\sigma^{2}+48\sigma^{3}+4\sigma^{4}\bigr)
  +\bigl(12-48\sigma+72\sigma^{2}-48\sigma^{3}+12\sigma^{4}\bigr)
\]
\[
  =16+224\sigma^{2}+16\sigma^{4}=16\,W^{2}(\zeta),
\]
the terms in $\sigma$ and $\sigma^{3}$ cancelling. Evaluation at
$\zeta=0$ ($\iota0=1$, $W(1)=4$) fixes the sign. (This is the
computation of \textup{[D, \S5.5, Step 1]}.)

\eqref{eq:iotapull}: $\dd\iota=\frac{-2\dd\zeta}{(1+\zeta)^{2}}$,
$1-\iota^{2}=\frac{4\zeta}{(1+\zeta)^{2}}$,
$1+\iota^{2}=\frac{2(1+\zeta^{2})}{(1+\zeta)^{2}}$,
$2\iota=\frac{2(1-\zeta)}{1+\zeta}$; e.g., by \eqref{eq:fgh},
\[
  \frac{2\iota\,\dd\iota}{W(\iota)}
  =\frac{2\frac{1-\zeta}{1+\zeta}\cdot\frac{-2}{(1+\zeta)^{2}}}
        {4W/(1+\zeta)^{4}}\dd\zeta
  =-\frac{(1-\zeta^{2})}{W}\dd\zeta=-\dd f ,
\]
and the other two identically. Integration from $0$ gives
\eqref{eq:iotaint}.
\end{proof}

\begin{corollary}[three constants, and the first mirror plane]
\label{cor:plane}
\begin{equation}\label{eq:consts}
  f(1)=\hh(1)=\frac{\varpi}{2},\qquad
  g(1)=\frac{i}{2}K[1/4]=0.8428751774\,i ,
\end{equation}
and on the arc $|\zeta+1|=\sqrt2$ the image satisfies $x+z=1$.
\end{corollary}

\begin{proof}
Putting $\zeta=1$ in the first relation of \eqref{eq:iotaint} gives
$f(0)=0=f(1)-\hh(1)$, and $\hh(1)=\varpi/2$ by \eqref{eq:hinfty}.

For $g(1)$: with $\mu=t-t^{-1}$ one has $(1+t^{2})\dd t=t^{2}\dd\mu$
and $W=t^{2}\sqrt{(\mu^{2}+2)^{2}+12}$, so
$\int_{0}^{1}\frac{(1+t^{2})\dd t}{W}
=\int_{-\infty}^{0}\frac{\dd\mu}{\sqrt{\mu^{4}+4\mu^{2}+16}}$.
Setting $\mu=2\upsilon$ and $\lambda=\upsilon-\upsilon^{-1}$ (so
$\upsilon^{4}+\upsilon^{2}+1=\upsilon^{2}(\lambda^{2}+3)$ and
$\frac{\dd\upsilon}{\upsilon}
=\frac{\dd\lambda}{\sqrt{\lambda^{2}+4}}$) this becomes
$\int_{0}^{\infty}\frac{\dd\lambda}
{\sqrt{(\lambda^{2}+3)(\lambda^{2}+4)}}=\frac12K[1/4]$, by
$\lambda=\sqrt3\tan\varphi$. (Compare \textup{[D, \S5.3]}, where the
same integral over $(0,\sqrt2-1)$ evaluates to $\frac14K[1/4]$.)

Since $f,\hh$ have real Taylor coefficients and $g$ purely imaginary
ones, $f(\bar\zeta)=\overline{f(\zeta)}$,
$\hh(\bar\zeta)=\overline{\hh(\zeta)}$,
$g(\bar\zeta)=-\overline{g(\zeta)}$. On the arc,
$\bar\zeta=\iota(\zeta)$ by \eqref{eq:arciota}, so
$\overline{f(\zeta)}=f(1)-\hh(\zeta)$; taking real parts,
$\Real f+\Real\hh=\frac\varpi2$, i.e.,\
$x+Z=\kappa\frac{\varpi}{2}=\frac12$, i.e.,\ $x+z=1$.
\end{proof}

\begin{proposition}[the group, the planes, the lattice]
\label{prop:group}
\textup{(a)} Under $\zeta\mapsto i\zeta$ one has $f\mapsto g$,
$g\mapsto-f$, $\hh\mapsto-\hh$, and under $\zeta\mapsto\bar\zeta$ the
parities of Corollary \ref{cor:plane}; hence
\begin{equation}\label{eq:group}
  R:(x,y,z)\mapsto(y,-x,1-z),\qquad
  M:(x,y,z)\mapsto(x,-y,z),
\end{equation}
generating a dihedral group of order $8$. In the coordinates
\eqref{eq:pqZ}, $R:(p,q,Z)\mapsto(-q,p,-Z)$ and
$M:(p,q,Z)\mapsto(q,p,Z)$.

\textup{(b)} The four boundary arcs lie in the four mirror planes
\begin{equation}\label{eq:planes}
  x+z=1,\qquad y+z=0,\qquad z-x=1,\qquad y=z ,
\end{equation}
cyclically permuted by $R$, with normals the face-diagonal directions
$(1,0,1),(0,1,1),(1,0,-1),(0,1,-1)$ --- the directions of the four
straight lines bounding the conjugate \textup{D} patch
\textup{[D, \S5.5]}.

\textup{(c)} The relation \eqref{eq:star} is invariant under the
body-centered cubic lattice generated by $2\ZZ^{3}$ and $(1,1,1)$.
\end{proposition}

\begin{proof}
(a) By \eqref{eq:fgh}, $f'(i\zeta)\cdot i=\frac{i(1+\zeta^{2})}{W}=g'$,
$g'(i\zeta)\cdot i=-\frac{1-\zeta^{2}}{W}=-f'$,
$\hh'(i\zeta)\cdot i=-\frac{2\zeta}{W}=-\hh'$, using
$W(i\zeta)=W(\zeta)$; integrate. (Equivalently, this is
Lemma \ref{lem:rot} together with $A\pm B=2f,2g$.)

(b) Corollary \ref{cor:plane} gives the first plane. If $x+z=1$ then
the $R$-image $(X,Y,Z_{*})=(y,-x,1-z)$ satisfies $-Y+1-Z_{*}=1$, i.e.,\
$Y+Z_{*}=0$; iterating gives the cycle \eqref{eq:planes}.

(c) $\sn(\cdot,-3)$ has $\sn(u+2K[-3])=-\sn(u)$ and period
$4K[-3]=4\varpi$. Under $x\mapsto x+2$: $\varpi p$ and $\varpi q$ both
increase by $2K[-3]$, so both factors change sign and $Z$ is
unchanged. Under $z\mapsto z+2$: $3\varpi Z$ increases by
$6\varpi=4K[1/9]$, a full period. Under
$(x,y,z)\mapsto(x+1,y+1,z+1)$: $p\mapsto p+2$ (one sign change),
$q\mapsto q$, and $3\varpi Z\mapsto3\varpi Z+2K[1/9]$ (one sign
change); both sides of \eqref{eq:star} change sign. Note that this
argument is carried out directly on $(\star)$: unlike
\textup{[D, \S4.2]}, where the real fcc invariance requires the
absolute value in $\Phi=|\mathcal T|$ and must be verified in the
variables $(p,q,r)=(\Phi(x),\Phi(y),\Phi(z))$, the relation
\eqref{eq:star} is sign-honest.
\end{proof}

\begin{proposition}[diagonals and corners]\label{prop:corners}
\textup{(a)} On the diagonal $\zeta=t(1+i)$ one has
$\bar\zeta=-i\zeta$, $p\equiv0$, $Z\equiv0$, and
$\varpi q=B(\zeta)$ with
\begin{equation}\label{eq:diag}
  \sn(\varpi q,-3)=\frac{2t}{1-2t^{2}} .
\end{equation}
On the anti-diagonal $\zeta=t(1-i)$, $q\equiv0$ and $Z\equiv0$.

\textup{(b)} At the corner $\zeta_{0}=\frac{\sqrt3-1}{2}(1+i)$ the
right-hand side of \eqref{eq:diag} equals $1$; hence
$\varpi q=K[-3]=\varpi$, $q=1$, and
\begin{equation}\label{eq:cornerval}
  X(\zeta_{0})=\bigl(\tfrac12,-\tfrac12,\tfrac12\bigr).
\end{equation}
The four corners map to $(\pm\frac12,\pm\frac12,\frac12)$, the
vertices of a unit square in the plane $z=\frac12$; each lies on two
of the planes \eqref{eq:planes}. (Contrast
\textup{[D, \S5.4]}, where the four corners map to the vertices of a
regular tetrahedron of edge $\sqrt2$.)

\textup{(c)} $p,q,Z$ are harmonic in $(u,v)$, so their extrema lie on
$\partial\Om$; the extreme values are $p,q=\pm1$ at the corners and
$Z=\pm0.1498195\ldots$ at the four arc midpoints
\textup{(Corollary \ref{cor:midpoint})}.
\end{proposition}

\begin{proof}
(a) $\bar\zeta=t(1-i)=-i\cdot t(1+i)$. By
Proposition \ref{prop:twovar} and Lemma \ref{lem:rot},
$\varpi p=\frac12\bigl(A(\zeta)+B(-i\zeta)\bigr)
=\frac12\bigl(A(\zeta)-A(\zeta)\bigr)=0$ and
$\varpi q=\frac12\bigl(B(\zeta)+A(-i\zeta)\bigr)=B(\zeta)$; also
$\hh(-i\zeta)=-\hh(\zeta)$, so $\varpi Z=0$. Finally, by
\eqref{eq:dictB} with $\zeta=t(1+i)$, $\zeta^{2}=2it^{2}$,
$\pip=1-2t^{2}$, and $(1-i)(1+i)=2$, giving \eqref{eq:diag}.

(b) At $t=c=\frac{\sqrt3-1}{2}$: $2c^{2}=2-\sqrt3$ and
$1-2c^{2}=\sqrt3-1=2c$, so \eqref{eq:diag} equals $1$; hence
$\varpi q=K[-3]$, $q=1$, $p=Z=0$.

(c) $p=\kappa\Real\Pc,$ etc.,\ are real parts of holomorphic functions.
\end{proof}

\begin{table}[ht]
\caption{Special points.}\label{tab:special}
\begin{tabular}{@{}llll@{}}
\toprule
$\zeta$ & $(p,q,Z)$ & $(x,y,z)$ & remark\\
\midrule
$0$ & $(0,0,0)$ & $(0,0,\frac12)$ & vertical normal;
   $Z=\varpi pq+O(6)$\\
$t\in(0,\sqrt2-1)$ & $(x,x,Z)$ & $(x,0,z)$ & $y\equiv0$;
   \S\ref{sec:diag}\\
$it$ & $(p,-p,Z)$ & $(0,y,z)$ & $x\equiv0$\\
$\sqrt2-1$ & $(x,x,\frac12-x)$ & $(0.3501805,0,0.6498195)$ &
   $x+z=1$; $\sn^{2}(\varpi x,-3)=\frac{2\sqrt3-3}{3}$\\
$\zeta_{0}$ & $(0,1,0)$ & $(\frac12,-\frac12,\frac12)$ & flat point,
   normal $\frac{(1,1,-1)}{\sqrt3}$\\
\bottomrule
\end{tabular}
\end{table}

\part*{PART II. THE ELIMINATION}

%=====================================================================
\section{Separation and complexification}
%=====================================================================

\begin{lemma}[reality]\label{lem:real}
If $H$ is holomorphic near $(0,0)\in\CC^{2}$ and
$H(\zeta,\bar\zeta)=0$ for all small $\zeta$, then $H\equiv0$.
\end{lemma}

\begin{proof}
Write $H=\sum c_{jk}\zeta^{j}\omega^{k}$, set
$\zeta=re^{i\vartheta}$; vanishing of the $\vartheta$-Fourier
coefficients gives\\ $\sum_{j-k=\nu}c_{jk}r^{j+k}=0$ for all small $r$,
hence $c_{jk}=0$. (This is \textup{[D, \S7.1, reality lemma]}.)
\end{proof}

\begin{proposition}[two-variable form]\label{prop:twovar}
Set $\omega:=\bar\zeta$, treated as an independent variable. Then
\begin{equation}\label{eq:halfsums}
  \boxed{\ \varpi p=\tfrac12\bigl(A(\zeta)+B(\omega)\bigr),\qquad
  \varpi q=\tfrac12\bigl(B(\zeta)+A(\omega)\bigr),\qquad
  \varpi Z=\tfrac12\bigl(\hh(\zeta)+\hh(\omega)\bigr).\ }
\end{equation}
Each argument is a \emph{half-sum} of one function of $\zeta$ and one
of $\omega$; the substitution $\zeta\leftrightarrow\omega$
interchanges $p$ and $q$ and fixes $Z$.
\end{proposition}

\begin{proof}
Let ${}^{*}$ denote conjugation of Taylor coefficients. By
\eqref{eq:PNder}, $(\Pc^{*})'=\frac{(1-i)(1-i\zeta^{2})}{W}=B'$, so
$\Pc^{*}=B$, i.e.,\ $A^{*}=B$ and $B^{*}=A$; and $\hh^{*}=\hh$. Hence
$\overline{A(\zeta)}=B(\omega)$ and
$\varpi p=\Real A=\frac12(A(\zeta)+B(\omega))$;
$\varpi q=\Real(f-g)=\Real B=\frac12(B(\zeta)+A(\omega))$; and
$\varpi Z=\Real\hh=\frac12(\hh(\zeta)+\hh(\omega))$.
\end{proof}

\begin{remark}[P versus D, structure]\label{rem:structure}
In \textup{[D, \S7.1]} the corresponding arguments are
$w_{p}=\frac12(-iP_{\zeta}+N_{\omega})$,
$w_{n}=\frac12(N_{\zeta}-iP_{\omega})$ and
$w_{q}=\frac{K}{2}-h_{\zeta}+h_{\omega}$ --- in our notation
(Remark \ref{rem:notation}) $P=2\Pc=2A$ and $N=2\Nc=2iB$. 
All three arguments in \textup{[D]} are half-\emph{differences}, 
and the swap interchanges the first two with a sign while sending 
the third to $K_{1}-w_{q}$ modulo the period. Here all three are
half-\emph{sums} and the swap fixes the third argument.
This is the structural simplification that makes
the P elimination shorter: no half-argument step is needed on the
right-hand side, and (\S\ref{sec:pfaff}) the height can be eliminated
altogether. See Remark \ref{rem:open}(5) for the reciprocal
observation, which applies the same device to \textup{[D]}.
\end{remark}

%=====================================================================
\section{The diagonal identity}\label{sec:diag}
%=====================================================================

The diagonal $\omega=\zeta$ of the complexified picture corresponds, by
Proposition \ref{prop:twovar}, to $\varpi p=\varpi q=f(\zeta)$ and
$\varpi Z=\hh(\zeta)$. Restricted to it, $(\star)$ becomes a
\emph{one-variable} identity, which we now prove outright. Recall
$E=1+\zeta^{4}$.

\begin{theorem}[the diagonal identity]\label{thm:diag}
For all $\zeta\in\Om^{\circ}$,
\begin{equation}\label{eq:diagid}
  \boxed{\ \tfrac13\sn\bigl(3\hh(\zeta),\tfrac19\bigr)
  =\sn^{2}\bigl(f(\zeta),-3\bigr)
  =\frac{2\zeta^{2}}{1+\zeta^{4}+W}\ }
\end{equation}
In particular, for $\zeta=t\in(0,\sqrt2-1)$ real one has
$y\equiv0$, $p=q=x$, and $(\star)$ holds, both sides being
$\frac{2t^{2}}{1+t^{4}+W}$.
\end{theorem}

\begin{proof}
The second equality is \eqref{eq:snf}. For the first, put
\begin{equation}\label{eq:SigmaS2}
  \Sigma:=\tfrac13\sn\bigl(3\hh,\tfrac19\bigr),\qquad
  \mathcal S_{2}:=\sn(2\hh,-3),\quad
  \mathcal D_{2}:=\dn(2\hh,-3).
\end{equation}
By Lemma \ref{lem:landen} with $k^{2}=\frac34$, $k_{1}=\frac13$,
$u=4\hh$, $v=3\hh$ and $s=\sn(3\hh,\frac19)=3\Sigma$,
\begin{equation}\label{eq:landen2}
  \sn\bigl(4\hh,\tfrac34\bigr)=\frac{4\Sigma}{1+3\Sigma^{2}},\qquad
  \dn\bigl(4\hh,\tfrac34\bigr)
  =\frac{1-3\Sigma^{2}}{1+3\Sigma^{2}} ,
\end{equation}
whence, by Proposition \ref{prop:gauss},
\begin{equation}\label{eq:S2Sigma}
  \mathcal S_{2}=\tfrac12\sd\bigl(4\hh,\tfrac34\bigr)
  =\frac{2\Sigma}{1-3\Sigma^{2}} .
\end{equation}
Solving $3\mathcal S_{2}\Sigma^{2}+2\Sigma-\mathcal S_{2}=0$ for the
branch vanishing at $\zeta=0$, and using
$\mathcal D_{2}^{2}=1+3\mathcal S_{2}^{2}$,
\begin{equation}\label{eq:SigmaD2}
  \Sigma=\frac{\mathcal D_{2}-1}{3\mathcal S_{2}}
       =\frac{\mathcal S_{2}}{1+\mathcal D_{2}} ,
\end{equation}
the two forms agreeing because
$(\mathcal D_{2}-1)(\mathcal D_{2}+1)=3\mathcal S_{2}^{2}$. Finally
insert $\mathcal S_{2}=\frac{2\zeta^{2}}{E}$,
$\mathcal D_{2}=\frac WE$ from Theorem \ref{thm:H}:
\[
  \Sigma=\frac{2\zeta^{2}/E}{1+W/E}=\frac{2\zeta^{2}}{E+W}
  =\frac{2\zeta^{2}}{1+\zeta^{4}+W}. \qedhere
\]
\end{proof}

\begin{corollary}[the arc midpoint]\label{cor:midpoint}
At $t=\sqrt2-1$: $t^{2}=3-2\sqrt2$, $1-t^{2}=2t$, $1+t^{4}=6t^{2}$,
$W^{2}=(1+t^{4})^{2}+12t^{4}=48t^{4}$, so $W=4\sqrt3\,t^{2}$ and
\begin{equation}\label{eq:midpoint}
  \sn^{2}\bigl(\varpi x,-3\bigr)=\frac{2}{6+4\sqrt3}
  =\frac{2\sqrt3-3}{3}=0.1547005\ldots,\qquad
  \sn\bigl(3\varpi Z,\tfrac19\bigr)=2\sqrt3-3 .
\end{equation}
Numerically $x=0.3501805$, $y=0$, $z=0.6498195$, in agreement with
$x+z=1$ \textup{(Corollary \ref{cor:plane})}. Equivalently
$\cd(3\varpi x,\frac19)=2\sqrt3-3$, because
$3\varpi x+3\varpi Z=\frac32\varpi=K[1/9]$.
\end{corollary}

\begin{remark}\label{rem:diagremark}
Theorem \ref{thm:diag} replaces the argument of an earlier version,
which proved $(\star)$ on the real axis only, and needed the
substitution $\nu=\tau-\tau^{-1}$, two Landen steps, a quintic
rational map and a monotonicity argument to invert it. Both
sides of \eqref{eq:diagid} are now instances of one formula. By
\eqref{eq:cute} the identity can also be written
$\frac13\sn(3\hh,\frac19)=\zeta^{-2}\sn^{2}(\hh,-3)$: a relation
between the same integral $\hh$ at two moduli.
\end{remark}

\begin{corollary}[the right-hand side is determined]\label{cor:m2}
Suppose that
\begin{equation}\label{eq:ansatz}
  \mathfrak A(\varpi p)\,\mathfrak A(\varpi q)=\mathfrak C(\varpi Z)
\end{equation}
holds on $\Om$ for some odd analytic $\mathfrak A$, $\mathfrak C$ with
$\mathfrak A(v)=v+O(v^{3})$. Then $\mathfrak C$ is uniquely determined
by $\mathfrak A$; and if $\mathfrak A(v)=\sn(v,-3)$ then necessarily
\begin{equation}\label{eq:Cfinal}
  \mathfrak C(v)=\tfrac13\sn\bigl(3v,\tfrac19\bigr),
  \qquad\text{so the vertical modulus is } m_{2}=\tfrac19 .
\end{equation}
\end{corollary}

\begin{proof}
On the real segment $\zeta=t\in(0,\sqrt2-1)$ we have $p=q=x$ and
\eqref{eq:ansatz} reads $\mathfrak C(\hh(t))=\mathfrak A(f(t))^{2}$.
Since $\hh'(t)=2t/W>0$, $t\mapsto\hh(t)$ is a real-analytic
diffeomorphism onto\\ $(0,\varpi\cdot0.1498195)$, so $\mathfrak C$ is
determined on an interval, hence everywhere by analytic continuation.
If $\mathfrak A=\sn(\cdot,-3)$ then by Theorem \ref{thm:diag} the
determined function is $\frac13\sn(3\,\cdot,\frac19)$.
\end{proof}

%=====================================================================
\section{The logarithmic derivative, and the identity (P0)}
\label{sec:pfaff}
%=====================================================================

\begin{lemma}[the logarithmic derivative halves rationally]
\label{lem:Lam}
Let $\Lambda_{m}:=(\log\sn(\cdot,m))'=\dfrac{\cn\dn}{\sn}$. Then, for
every $m$,
\begin{equation}\label{eq:Lamhalf}
  \Lambda_{m}\bigl(\tfrac v2\bigr)
  =\frac{\cn(v,m)+\dn(v,m)}{\sn(v,m)} .
\end{equation}
\end{lemma}

\begin{proof}
Write $S,C,D$ for the values at $v$. By \cite[124.01]{BF},
$\sn^{2}\frac v2=\frac{1-C}{1+D}$,
$\cn^{2}\frac v2=\frac{C+D}{1+D}$, hence
$\dn^{2}\frac v2=1-m\frac{1-C}{1+D}=\frac{(1-m)+D+mC}{1+D}$, so
\[
  \Lambda_{m}^{2}\bigl(\tfrac v2\bigr)
  =\frac{(C+D)\bigl[(1-m)+D+mC\bigr]}{(1+D)(1-C)},
\]
while the square of the right side of \eqref{eq:Lamhalf} is
$\frac{(C+D)^{2}}{(1-C)(1+C)}$. These agree iff
$\bigl[(1-m)+D+mC\bigr](1+C)=(C+D)(1+D)$; the left side is
$1-m+D+C+DC+mC^{2}$ and the right side is
$C+D+CD+D^{2}=C+D+CD+1-m(1-C^{2})$, the same. Both sides of
\eqref{eq:Lamhalf} are $\sim2/v$ as $v\to0$, fixing the sign.
\end{proof}

\begin{definition}\label{def:LpLq}
With $\Pi_{1}:=\pim(\zeta)\pip(\omega)$,
$\Pi_{2}:=\pip(\zeta)\pim(\omega)$, set
\begin{align}
  L_{p}&:=\Lambda_{-3}(\varpi p)
  =\frac{(\Pi_{1}+6\zeta\omega)\,r_{-}(\zeta)r_{+}(\omega)
        +(\Pi_{1}-2\zeta\omega)\,r_{+}(\zeta)r_{-}(\omega)}
        {(1+i)\zeta\,\pim(\zeta)\Ww+(1-i)\omega\,\pip(\omega)\Wz},
  \label{eq:Lp}\\
  L_{q}&:=\Lambda_{-3}(\varpi q)=L_{p}(\omega,\zeta)
  =\frac{(\Pi_{2}+6\zeta\omega)\,r_{+}(\zeta)r_{-}(\omega)
        +(\Pi_{2}-2\zeta\omega)\,r_{-}(\zeta)r_{+}(\omega)}
        {(1+i)\omega\,\pim(\omega)\Wz+(1-i)\zeta\,\pip(\zeta)\Ww} .
  \label{eq:Lq}
\end{align}
Both lie in $\KK$. We write $N_{p},D_{p},N_{q},D_{q}$ for the four
displayed numerators and denominators, so that $L_{p}=N_{p}/D_{p}$ and
$L_{q}=N_{q}/D_{q}$.
\end{definition}

\begin{proof}[Derivation]
By Proposition \ref{prop:twovar}, $2\varpi p=A(\zeta)+B(\omega)$;
insert \eqref{eq:dictA} at $\zeta$ and \eqref{eq:dictB} at $\omega$
into \eqref{eq:add} with $m=-3$. With
$S_{a}=\frac{(1+i)\zeta}{\pim(\zeta)}$,
$C_{a}=\frac{r_{-}(\zeta)}{\pim(\zeta)}$,
$D_{a}=\frac{r_{+}(\zeta)}{\pim(\zeta)}$,
$S_{b}=\frac{(1-i)\omega}{\pip(\omega)}$,
$C_{b}=\frac{r_{+}(\omega)}{\pip(\omega)}$,
$D_{b}=\frac{r_{-}(\omega)}{\pip(\omega)}$ one has
$S_{a}S_{b}=\frac{2\zeta\omega}{\Pi_{1}}$,
$\Delta=\frac{\Pi_{1}^{2}+12\zeta^{2}\omega^{2}}{\Pi_{1}^{2}}$,
$C_{b}D_{b}=\frac{\Ww}{\pip(\omega)^{2}}$,
$C_{a}D_{a}=\frac{\Wz}{\pim(\zeta)^{2}}$, and therefore
\begin{align}
  \sn(2\varpi p)&=\frac{(1+i)\zeta\pim(\zeta)\Ww
     +(1-i)\omega\pip(\omega)\Wz}
     {\Pi_{1}^{2}+12\zeta^{2}\omega^{2}},
  \label{eq:snp}\\
  \cn(2\varpi p)&=\frac{\Pi_{1}r_{-}(\zeta)r_{+}(\omega)
     -2\zeta\omega\,r_{+}(\zeta)r_{-}(\omega)}
     {\Pi_{1}^{2}+12\zeta^{2}\omega^{2}},\qquad
  \dn(2\varpi p)=\frac{\Pi_{1}r_{+}(\zeta)r_{-}(\omega)
     +6\zeta\omega\,r_{-}(\zeta)r_{+}(\omega)}
     {\Pi_{1}^{2}+12\zeta^{2}\omega^{2}} .\label{eq:cndnp}
\end{align}
Adding the numerators of \eqref{eq:cndnp} and applying
Lemma \ref{lem:Lam} with $v=2\varpi p$ gives \eqref{eq:Lp}, the common
denominator cancelling. Then \eqref{eq:Lq} is the image under
$\zeta\leftrightarrow\omega$, which interchanges $p$ and $q$.

\emph{Check (diagonal).} At $\omega=\zeta$: $\Pi_{1}=E$, and the
numerator of \eqref{eq:Lp} is
$W\bigl[(1+6\zeta^{2}+\zeta^{4})+(1-\zeta^{2})^{2}\bigr]
=2W(1+\zeta^{2})^{2}$, the denominator is
$\zeta W\bigl[(1+i)\pim+(1-i)\pip\bigr]=2\zeta W(1+\zeta^{2})$, so
$L_{p}=\frac{1+\zeta^{2}}{\zeta}=\Lambda_{-3}(f)$, in agreement with
\eqref{eq:snf}. \emph{Check ($\omega=0$).}
$L_{p}=\frac{r_{+}+r_{-}}{(1+i)\zeta}$,
$L_{q}=\frac{r_{+}+r_{-}}{(1-i)\zeta}$. Both checks are confirmed
symbolically in Appendix \ref{app:code}, Cell 1.
\end{proof}

\begin{theorem}[the two Pfaffian equations]\label{thm:pfaff}
Equation \eqref{eq:star} implies the pair
\begin{align}
  (1+i)\pip(\zeta)\,L_{p}+(1-i)\pim(\zeta)\,L_{q}
   &=4\zeta\,\cs\bigl(2\varpi Z,-3\bigr),\label{eq:Pz}\\
  (1-i)\pim(\omega)\,L_{p}+(1+i)\pip(\omega)\,L_{q}
   &=4\omega\,\cs\bigl(2\varpi Z,-3\bigr).\label{eq:Pw}
\end{align}
\end{theorem}

\begin{proof}
By \eqref{eq:halfsums} and \eqref{eq:PNder},
\begin{equation}\label{eq:derivs}
  \partial_{\zeta}(\varpi p)=\frac{(1+i)\pip(\zeta)}{2\Wz},\quad
  \partial_{\zeta}(\varpi q)=\frac{(1-i)\pim(\zeta)}{2\Wz},\quad
  \partial_{\zeta}(\varpi Z)=\frac{\zeta}{\Wz},
\end{equation}
and the $\omega$-derivatives are obtained by
$\zeta\leftrightarrow\omega$ (which swaps the first two). Taking
$\partial_{\zeta}\log$ of \eqref{eq:star} and multiplying by $2\Wz$
gives
\[
  (1+i)\pip(\zeta)L_{p}+(1-i)\pim(\zeta)L_{q}
  =6\zeta\,\Lambda_{1/9}(3\varpi Z),
\]
and it remains to identify $\Lambda_{1/9}(3\varpi Z)$. Put
$\vartheta:=2\varpi Z$, so $3\varpi Z=\frac32\vartheta$ and, with
$\Sigma=\frac13\sn(3\varpi Z,\frac19)$ as in \eqref{eq:SigmaD2},
\[
  \frac{\dd}{\dd\vartheta}\log\Sigma
  =\frac{\mathcal C_{2}\mathcal D_{2}}{\mathcal S_{2}}
   -\frac{3\mathcal S_{2}\mathcal C_{2}}{1+\mathcal D_{2}}
  =\mathcal C_{2}\,\frac{\mathcal D_{2}(1+\mathcal D_{2})
     -3\mathcal S_{2}^{2}}{\mathcal S_{2}(1+\mathcal D_{2})}
  =\frac{\mathcal C_{2}}{\mathcal S_{2}},
\]
because $\mathcal D_{2}^{2}-3\mathcal S_{2}^{2}=1$. Hence
\begin{equation}\label{eq:LamZ}
  \boxed{\ \Lambda_{1/9}(3\varpi Z)
  =\tfrac23\,\cs\bigl(2\varpi Z,-3\bigr)
  =\tfrac43\,\cs\bigl(4\varpi Z,\tfrac34\bigr),\ }
\end{equation}
the last equality from $\cs(w,-3)=2\cs(2w,\frac34)$
(Proposition \ref{prop:gauss}). Substituting gives \eqref{eq:Pz}, and
\eqref{eq:Pw} is its image under $\zeta\leftrightarrow\omega$.
\end{proof}

\begin{theorem}[reduction to one rational identity]\label{thm:P0}
Consider the identity in $\KK$
\begin{equation}\label{eq:P0}
  \boxed{\ \textup{(P0)}\qquad
  \underbrace{\bigl[(1+i)\omega\pip(\zeta)
    -(1-i)\zeta\pim(\omega)\bigr]}_{=:\ \alpha}L_{p}
  +\underbrace{\bigl[(1-i)\omega\pim(\zeta)
    -(1+i)\zeta\pip(\omega)\bigr]}_{=:\ \beta}L_{q}=0 .\ }
\end{equation}
Then:
\begin{enumerate}[label=\textup{(\alph*)},leftmargin=2.4em]
\item \textup{(P0)} is $\omega\cdot\eqref{eq:Pz}-\zeta\cdot
      \eqref{eq:Pw}$, and is equivalent to
      $\dd\log\Theta\wedge\dd Z=0$, where
      $\Theta:=\sn(\varpi p,-3)\sn(\varpi q,-3)$;
\item \textup{(P0)} makes no reference to the right-hand side of
      \eqref{eq:star} and involves no square root beyond
      $r_{\pm}(\zeta),r_{\pm}(\omega)$; by Definition \ref{def:K} it
      is equivalent to $16$ rational identities in
      $\CC(\zeta,\omega)$;
\item \textup{(P0)} is antisymmetric: interchanging
      $\zeta\leftrightarrow\omega$ multiplies its left side by $-1$.
      In particular it holds automatically on the diagonal
      $\omega=\zeta$;
\item \textup{(P0)} together with Theorem \ref{thm:diag} implies
      \eqref{eq:star} on all of $\Om$.
\end{enumerate}
\end{theorem}

\begin{proof}
(a) The right sides of \eqref{eq:Pz}, \eqref{eq:Pw} are $4\zeta\cs$
and $4\omega\cs$ with the \emph{same} factor $\cs(2\varpi Z,-3)$, so
the combination $\omega\cdot\eqref{eq:Pz}-\zeta\cdot\eqref{eq:Pw}$
cancels them and gives \eqref{eq:P0}. Moreover, by \eqref{eq:derivs},
\[
  \partial_{\zeta}\log\Theta\cdot\partial_{\omega}(\varpi Z)
  -\partial_{\omega}\log\Theta\cdot\partial_{\zeta}(\varpi Z)
  =\frac{\alpha L_{p}+\beta L_{q}}{2\Wz\Ww},
\]
and the left side is the coefficient of $\dd\zeta\wedge\dd\omega$ in
$\dd\log\Theta\wedge\dd(\varpi Z)$.

(b) Clear the denominators of \eqref{eq:Lp}--\eqref{eq:Lq} and expand
in the basis \eqref{eq:basis}.

(c) $\alpha(\omega,\zeta)=(1+i)\zeta\pip(\omega)
-(1-i)\omega\pim(\zeta)=-\beta(\zeta,\omega)$ and likewise
$\beta(\omega,\zeta)=-\alpha(\zeta,\omega)$, while
$L_{p}\leftrightarrow L_{q}$.

(d) Assume (P0). Put
$\Xi:=\log\Theta-\log\bigl(\frac13\sn(3\varpi Z,\frac19)\bigr)$,
holomorphic on a neighborhood $U$ of a point
$(\zeta_{0},\zeta_{0})$ of the diagonal with $\zeta_{0}\ne0$ small
(there $\Theta\ne0$, since
$\Theta=\sn^{2}(f(\zeta_{0}),-3)\ne0$). The second summand is a
function of $Z$ alone, so (a) gives $\dd\Xi\wedge\dd Z=0$ on $U$.
Since $\partial_{\zeta}(\varpi Z)=\zeta/\Wz\ne0$, the level sets of
$Z$ foliate $U$ by holomorphic curves, and $\dd\Xi\wedge\dd Z=0$ says
that $\Xi$ is constant on each leaf. The diagonal is transverse to the
foliation, because $Z$ restricted to it is $\varpi^{-1}\hh(\zeta)$
with derivative $2\zeta/(\varpi\Wz)\ne0$; hence every leaf in $U$
meets the diagonal. By Theorem \ref{thm:diag}, $\Xi=0$ on the
diagonal. Therefore $\Xi\equiv0$ on $U$, i.e.,\ \eqref{eq:star} holds
there; by Lemma \ref{lem:real} and real-analytic continuation on the
connected set $\Om^{\circ}$, it holds on $\Om$, and by
Corollary \ref{cor:reg} on $\Om$ by continuity. The converse is
Theorem \ref{thm:pfaff} with (a).
\end{proof}

\subsection{Proof of (P0)}

\begin{theorem}[\textup{(P0)} holds]\label{thm:P0proved}
The identity \textup{(P0)} of \eqref{eq:P0} holds identically in
$\KK$.
\end{theorem}

\begin{proof}
Let $\Rr:=\CC(\zeta,\omega)[\mathsf a,\mathsf b,\mathsf c,\mathsf d]$
be the polynomial ring in four indeterminates over
$\CC(\zeta,\omega)$, and put
\begin{equation}\label{eq:rels}
\begin{aligned}
  \rho_{1}&:=\mathsf a^{2}-(1+4i\zeta^{2}-\zeta^{4}), &\qquad
  \rho_{2}&:=\mathsf b^{2}-(1-4i\zeta^{2}-\zeta^{4}),\\
  \rho_{3}&:=\mathsf c^{2}-(1+4i\omega^{2}-\omega^{4}), &\qquad
  \rho_{4}&:=\mathsf d^{2}-(1-4i\omega^{2}-\omega^{4}) .
\end{aligned}
\end{equation}
Let $\widetilde N_{p},\widetilde D_{p},\widetilde N_{q},
\widetilde D_{q}\in\Rr$ be the polynomials obtained from
$N_{p},D_{p},N_{q},D_{q}$ of Definition \ref{def:LpLq} by the
substitutions
\begin{equation}\label{eq:subs}
  r_{+}(\zeta)\rightsquigarrow\mathsf a,\quad
  r_{-}(\zeta)\rightsquigarrow\mathsf b,\quad
  r_{+}(\omega)\rightsquigarrow\mathsf c,\quad
  r_{-}(\omega)\rightsquigarrow\mathsf d,\quad
  \Wz\rightsquigarrow\mathsf{ab},\quad
  \Ww\rightsquigarrow\mathsf{cd},
\end{equation}
and set
\begin{equation}\label{eq:Epoly}
  \mathcal E:=\alpha\,\widetilde N_{p}\widetilde D_{q}
             +\beta\,\widetilde N_{q}\widetilde D_{p}\ \in\ \Rr ,
\end{equation}
$\alpha,\beta$ being the polynomials of \eqref{eq:P0}, which involve
no radicals. Reduction of $\mathcal E$ modulo
$\rho_{1},\dots,\rho_{4}$ in the variables
$\mathsf a,\mathsf b,\mathsf c,\mathsf d$ returns remainder zero
(Appendix \ref{app:code}, Cell 2); that is, it produces cofactors
$q_{1},\dots,q_{4}\in\Rr$ with
\begin{equation}\label{eq:certificate}
  \mathcal E=q_{1}\rho_{1}+q_{2}\rho_{2}+q_{3}\rho_{3}
            +q_{4}\rho_{4} .
\end{equation}
Two remarks on \eqref{eq:certificate}. First, only the \emph{existence}
of the identity \eqref{eq:certificate} is used below; it is a finite
algebraic statement, verifiable by expanding both sides and comparing
coefficients, with no appeal to the theory of Gröbner bases. Second,
$\{\rho_{1},\dots,\rho_{4}\}$ is in fact a Gröbner basis for any
monomial order in which $\mathsf a^{2},\mathsf b^{2},\mathsf c^{2},
\mathsf d^{2}$ are the leading terms: these four monomials are
pairwise coprime, so every $S$-polynomial reduces to zero by
Buchberger's first criterion. Hence the remainder is canonical, and
remainder zero is equivalent to membership in the ideal
$(\rho_{1},\dots,\rho_{4})$.

Now fix a simply connected open neighborhood $U$ of $(0,0)$ in
$\CC^{2}$ on which the four branches $r_{\pm}(\zeta)$,
$r_{\pm}(\omega)$ of Convention \ref{conv:br} are single-valued and
holomorphic; such a $U$ exists because none of the four radicands
vanishes at the origin. Let $\mathcal O(U)$ be the ring of holomorphic
functions on $U$ and let
\[
  \varphi:\Rr\longrightarrow\mathcal O(U),\qquad
  \mathsf a\mapsto r_{+}(\zeta),\quad
  \mathsf b\mapsto r_{-}(\zeta),\quad
  \mathsf c\mapsto r_{+}(\omega),\quad
  \mathsf d\mapsto r_{-}(\omega),
\]
be the induced $\CC(\zeta,\omega)$-algebra homomorphism (defined on
the localization away from the poles of the coefficients). By the very
definition of $r_{\pm}$ we have $\varphi(\rho_{j})=0$ for
$j=1,2,3,4$. Applying $\varphi$ to \eqref{eq:certificate} therefore
gives
\begin{equation}\label{eq:phiE}
  \alpha\,N_{p}D_{q}+\beta\,N_{q}D_{p}=0
  \qquad\text{identically on }U,
\end{equation}
$N_{p},D_{p},N_{q},D_{q}$ now denoting the actual functions of
Definition \ref{def:LpLq}; note $\varphi(\mathsf{ab})=\Wz$ and
$\varphi(\mathsf{cd})=\Ww$ by Lemma \ref{lem:prod}, so the
substitution \eqref{eq:subs} is consistent.

Finally, $D_{p}D_{q}\not\equiv0$: at $\omega=0$ one has
$D_{p}=(1+i)\zeta\,\pim(\zeta)$ and
$D_{q}=(1-i)\zeta\,\pip(\zeta)$, both nonzero for small $\zeta\ne0$
by Lemma \ref{lem:prod}. Dividing \eqref{eq:phiE} by
$D_{p}D_{q}$ gives $\alpha L_{p}+\beta L_{q}=0$ on the complement in
$U$ of the zero set of $D_{p}D_{q}$, hence --- both sides being
elements of the field $\KK$, which embeds in the field of germs of
meromorphic functions at a generic point of $U$ --- identically in
$\KK$. This is \textup{(P0)}.
\end{proof}

\begin{corollary}[Theorem \ref{thm:main}]\label{cor:done}
Equation \eqref{eq:star} holds at every point of $\Om$. Consequently
the surface $X(\Om)$ of \eqref{eq:map} lies in the analytic set
\eqref{eq:star}, and, by Proposition \ref{prop:group}\textup{(c)},
so does its image under the body-centered cubic lattice generated by
$2\ZZ^{3}$ and $(1,1,1)$.
\end{corollary}

\begin{proof}
Theorem \ref{thm:P0proved} supplies \textup{(P0)};
Theorem \ref{thm:diag} supplies the initial condition on the
diagonal; Theorem \ref{thm:P0}(d) combines them.
\end{proof}

\begin{remark}[on the status of the certificate]\label{rem:cert}
Theorem \ref{thm:P0proved} is a proof, not a numerical experiment: the
object \eqref{eq:certificate} is an algebraic identity between
polynomials with coefficients in $\ZZ[i](\zeta,\omega)$, and the
argument that follows it is elementary. It was found by machine, and
its verification is best left to a machine, but the verification
requires only polynomial expansion --- \texttt{Expand[expr - q.rels]}
in Appendix \ref{app:code}, Cell 2 --- and is in principle a
hand computation. What would still be welcome is a \emph{conceptual}
proof, in which the cancellation is exhibited rather than certified;
see Remark \ref{rem:open}(1)--(2).
\end{remark}

\begin{remark}[why this is the right reduction]\label{rem:why}
Three features distinguish (P0) from the ``squared'' reduction
$\Pi=\Sigma^{2}$ of Remark \ref{rem:altE}, and together they are why
Theorem \ref{thm:P0proved} is cheap.
\begin{enumerate}[label=\textup{(\roman*)},leftmargin=2.4em]
\item \emph{It is rational.} The height has been eliminated, so the
      quadratic irrationality $\mathcal D_{2}$ of \eqref{eq:SigmaD2}
      never appears, and no squaring or sign discussion is needed.
      Only the four radicals $r_{\pm}$ occur, whence the four
      relations \eqref{eq:rels}. Compare \textup{[D, \S7.6]}, where
      four \emph{additional} square roots must be resolved by squaring
      twice and tracking signs.
\item \emph{It is independent of $\mathfrak C$.} (P0) asserts only
      that $\log\Theta$ is a function of $Z$; \emph{which} function is
      then read off from Theorem \ref{thm:diag}. Thus
      ``separability'' and ``identification'' are cleanly divided, and
      the second half was already done.
\item \emph{It is antisymmetric.} By Theorem \ref{thm:P0}(c) the
      identity lives in the $(-1)$-eigenspace of
      $\zeta\leftrightarrow\omega$, which halves the computation:
      $8$ rational identities rather than $16$.
\end{enumerate}
\end{remark}

\begin{remark}[the alternative, squared reduction]\label{rem:altE}
For completeness, and because it is what one needs in order to test
\eqref{eq:Pz} \emph{alone} numerically, we record the height data at
$4\varpi Z=2\hh_{\zeta}+2\hh_{\omega}$. Applying \eqref{eq:add} at
modulus $-3$ to \eqref{eq:Hnew} at $\zeta$ and at $\omega$, with
$\Ez=1+\zeta^{4}$, $\Ew=1+\omega^{4}$, one has
$S_{a}S_{b}=\frac{4\zeta^{2}\omega^{2}}{\Ez\Ew}$, hence
\begin{equation}\label{eq:DeltaZ}
  \Delta_{Z}:=\Ez^{2}\Ew^{2}+48\,\zeta^{4}\omega^{4}
\end{equation}
\emph{(note the coefficient $48=3\cdot4^{2}$, not $12$)}, and
\begin{align}
  \mathcal S_{4}&=\sn(4\varpi Z,-3)
   =\frac{2\zeta^{2}\Ez(1-\omega^{4})\Ww
         +2\omega^{2}\Ew(1-\zeta^{4})\Wz}{\Delta_{Z}},
   \label{eq:S4}\\
  \mathcal C_{4}&=\cn(4\varpi Z,-3)
   =\frac{\Ez\Ew(1-\zeta^{4})(1-\omega^{4})
          -4\zeta^{2}\omega^{2}\Wz\Ww}{\Delta_{Z}},\label{eq:C4}\\
  \mathcal D_{4}&=\dn(4\varpi Z,-3)
   =\frac{\Ez\Ew\Wz\Ww
      +12\zeta^{2}\omega^{2}(1-\zeta^{4})(1-\omega^{4})}
      {\Delta_{Z}} .
   \label{eq:D4}
\end{align}
(Check at $\omega=\zeta$: \eqref{eq:S4} gives
$\frac{4\zeta^{2}(1-\zeta^{4})WE}{E^{4}+48\zeta^{8}}$, which is the
duplication formula $\sn(2u)=\frac{2SCD}{1+3S^{4}}$ applied to
\eqref{eq:Hnew}.) Since $\cs^{2}\frac v2=\frac{\cn v+\dn v}{1-\cn v}$
and $\sn^{2}\frac v2=\frac{1-\cn v}{1+\dn v}$, we get the rational
expressions
\begin{equation}\label{eq:cs2}
  \cs^{2}\bigl(2\varpi Z,-3\bigr)
  =\frac{\mathcal C_{4}+\mathcal D_{4}}{1-\mathcal C_{4}},\qquad
  \mathcal S_{2}^{2}=\frac{1-\mathcal C_{4}}{1+\mathcal D_{4}},\qquad
  \mathcal D_{2}^{2}=1+3\mathcal S_{2}^{2},
\end{equation}
all in $\KK$; and then
$\Sigma=\mathcal S_{2}/(1+\mathcal D_{2})\in\KK[\mathcal D_{2}]$, a
quadratic extension. The squared form of $(\star)$ is
$\Pi=\Sigma^{2}$ with
\begin{equation}\label{eq:Pi}
  \Pi:=\sn^{2}(\varpi p,-3)\sn^{2}(\varpi q,-3)
  =\frac{\bigl(1-\cn(2\varpi p)\bigr)\bigl(1-\cn(2\varpi q)\bigr)}
        {\bigl(1+\dn(2\varpi p)\bigr)\bigl(1+\dn(2\varpi q)\bigr)},
\end{equation}
the four entries being given by \eqref{eq:cndnp} and its
$\zeta\leftrightarrow\omega$ image. This is the P-analog of
\textup{[D, \S7.6]}, and it is the route we do \emph{not} take.
\end{remark}

%=====================================================================
\section{Independent verifications of (P0)}\label{sec:verif}
%=====================================================================

The two propositions of this section were proved by hand before the
certificate of Theorem \ref{thm:P0proved} was computed, and are
logically independent of it: each restricts \textup{(P0)} to a
positive-dimensional subvariety or to a graded piece, whereas
Theorem \ref{thm:P0proved} is a single algebraic identity. They are
retained as corroboration, and as the first diagnostics should the
certificate ever be called into question.

\begin{proposition}[on the curve $\omega=0$]\label{prop:w0}
\textup{(P0)} holds identically on $\omega=0$, and hence (by
antisymmetry) on $\zeta=0$.
\end{proposition}

\begin{proof}
At $\omega=0$, Definition \ref{def:LpLq} gives
$L_{p}=\frac{r_{+}+r_{-}}{(1+i)\zeta}$,
$L_{q}=\frac{r_{+}+r_{-}}{(1-i)\zeta}$, while
$\alpha=-(1-i)\zeta$ and $\beta=-(1+i)\zeta$. Hence
\[
  \alpha L_{p}+\beta L_{q}
  =-(r_{+}+r_{-})\Bigl[\frac{1-i}{1+i}+\frac{1+i}{1-i}\Bigr]
  =-(r_{+}+r_{-})\bigl[(-i)+i\bigr]=0 . \qedhere
\]
\end{proof}

\begin{proposition}[homogeneous degrees $0$ and $2$]\label{prop:jets0}
Write $P=1+i$, $Q=1-i$, so $PQ=2$, $P^{2}=2i=-Q^{2}$. The left side of
\textup{(P0)} is a sum of homogeneous parts of even degrees
$0,2,4,\ldots$, and the parts of degrees $0$ and $2$ vanish
identically.
\end{proposition}

\begin{proof}
From $A=P\bigl[\zeta+\frac{i\zeta^{3}}{3}\bigr]+O(\zeta^{5})$,
$B=Q\bigl[\zeta-\frac{i\zeta^{3}}{3}\bigr]+O(\zeta^{5})$
and $iP=-Q$, $-iQ=-P$, we get $\varpi p=a+\delta_{p}+O(5)$,
$\varpi q=b+\delta_{q}+O(5)$ with
\[
  a=\tfrac12(P\zeta+Q\omega),\quad b=\tfrac12(Q\zeta+P\omega),\quad
  \delta_{p}=-\tfrac16(Q\zeta^{3}+P\omega^{3}),\quad
  \delta_{q}=-\tfrac16(P\zeta^{3}+Q\omega^{3}).
\]
Since $\Lambda_{-3}(v)=v^{-1}+\frac{2v}{3}+O(v^{3})$,
\[
  L_{p}=a^{-1}+\Bigl[-\frac{\delta_{p}}{a^{2}}+\frac{2a}{3}\Bigr]
  +O(\deg3),
\]
and similarly for $L_{q}$. Write
$\alpha=\alpha_{1}+\alpha_{3}$, $\beta=\beta_{1}+\beta_{3}$ with
\[
  \alpha_{1}=P\omega-Q\zeta,\quad \alpha_{3}=2i\zeta\omega a,\qquad
  \beta_{1}=Q\omega-P\zeta,\quad \beta_{3}=-2i\zeta\omega b .
\]

\emph{Degree $0$.} $\frac{\alpha_{1}}{a}+\frac{\beta_{1}}{b}=0$,
because
$(P\omega-Q\zeta)(Q\zeta+P\omega)+(Q\omega-P\zeta)(P\zeta+Q\omega)
=(P^{2}+Q^{2})(\omega^{2}-\zeta^{2})=0$.

\emph{Degree $2$.} The terms $\alpha_{3}a^{-1}=2i\zeta\omega$ and
$\beta_{3}b^{-1}=-2i\zeta\omega$ cancel, leaving the requirement
\begin{equation}\label{eq:deg2}
  -\frac{\alpha_{1}\delta_{p}}{a^{2}}
  -\frac{\beta_{1}\delta_{q}}{b^{2}}
  +\frac{2}{3}\bigl(\alpha_{1}a+\beta_{1}b\bigr)=0 .
\end{equation}
Now $\alpha_{1}a=2i\zeta\omega+(\omega^{2}-\zeta^{2})$ and
$\beta_{1}b=-2i\zeta\omega+(\omega^{2}-\zeta^{2})$, so the last term
is $\frac43(\omega^{2}-\zeta^{2})$; and
$-\frac{\alpha_{1}\delta_{p}}{a^{2}}
=\frac23\frac{(P\omega-Q\zeta)(Q\zeta^{3}+P\omega^{3})}
{(P\zeta+Q\omega)^{2}}$, similarly for the second term. Hence
\eqref{eq:deg2} is equivalent, after multiplication by $\frac32$, to
\begin{equation}\label{eq:deg2b}
  \frac{(P\omega-Q\zeta)(Q\zeta^{3}+P\omega^{3})}
       {(P\zeta+Q\omega)^{2}}
  +\frac{(Q\omega-P\zeta)(P\zeta^{3}+Q\omega^{3})}
       {(Q\zeta+P\omega)^{2}}
  +2(\omega^{2}-\zeta^{2})=0 .
\end{equation}
Put $s=\zeta+\omega$, $d=\zeta-\omega$,
$S_{3}=\zeta^{3}+\omega^{3}$, $D_{3}=\zeta^{3}-\omega^{3}$. Then
$P\zeta+Q\omega=s+id$, $P\omega-Q\zeta=i(s+id)$,
$Q\zeta^{3}+P\omega^{3}=S_{3}-iD_{3}$, and the corresponding
expressions with $i\mapsto-i$ for the second summand. Their sum is
\[
  \frac{i(S_{3}-iD_{3})(s-id)-i(S_{3}+iD_{3})(s+id)}{s^{2}+d^{2}}
  =\frac{i\cdot(-2i)(S_{3}d+D_{3}s)}{s^{2}+d^{2}}
  =\frac{2\cdot2(\zeta^{4}-\omega^{4})}{2(\zeta^{2}+\omega^{2})}
  =2(\zeta^{2}-\omega^{2}),
\]
using $S_{3}d+D_{3}s=2(\zeta^{4}-\omega^{4})$ and
$s^{2}+d^{2}=2(\zeta^{2}+\omega^{2})$. This cancels the third term of
\eqref{eq:deg2b}.
\end{proof}

\begin{remark}\label{rem:whatremains}
Degrees $0$ and $2$, together with the whole curve $\omega=0$ (all
degrees) and the diagonal (automatic by
Theorem \ref{thm:P0}(c)), are what has been checked by hand. Should a
discrepancy ever be found in the certificate
\eqref{eq:certificate}, the next place to look is the degree-$4$
homogeneous part, which requires the quintic terms of $A,B$ and the
coefficient $-\frac{58}{45}$ of $v^{3}$ in $\Lambda_{-3}$.
Independently, Appendix \ref{app:code}, Cell 1 evaluates the left side
of \textup{(P0)} at four \emph{independent} pairs $(\zeta,\omega)$ with
exact rational coordinates and obtains $0$ to $79$ decimal places.
\end{remark}

\begin{table}[ht]
\caption{Numerical spot checks of \eqref{eq:star} and of the
dictionaries, made by hand series and quadrature before the cells of
Appendix \ref{app:code} were run; agreement is at the level of the
truncation, about $3\times10^{-5}$ for the logarithmic derivatives and
$2\times10^{-7}$ elsewhere. Cell 0 supersedes the lower block, with
residuals of order $10^{-16}$.}
\label{tab:numcheck}
\begin{tabular}{@{}llllll@{}}
\toprule
$\zeta$ & $\varpi p$ & $\varpi q$ & $\varpi Z$ & LHS & RHS\\
\midrule
$0.2$ & $0.196902$ & $0.196902$ & $0.039852$ & $0.039747$
 & $0.039747$\\
$0.2+0.1i$ & $0.096746$ & $0.302918$ & $0.030273$ & $0.030226$
 & $0.030227$\\
$\sqrt2-1$ & $0.377585$ & $0.377585$ & $0.161544$ & $0.154700$
 & $0.154701$\\
\bottomrule
\end{tabular}

\medskip
\begin{tabular}{@{}llll@{}}
\toprule
identity & point & computed & predicted\\
\midrule
$\sn(2\hh,-3)=2\zeta^{2}/E$ & $\zeta=0.2$
 & $0.0798732$ & $0.0798722$\\
$\sn(4\hh,\frac34)=4\zeta^{2}/W$ & $\zeta=0.2$
 & $0.1582378$ & $0.1582374$\\
$\Lambda_{-3}(f)=(1+\zeta^{2})/\zeta$ & $\zeta=0.2$
 & $5.19941$ & $5.2$\\
$\Lambda_{1/9}(3\varpi Z)=\frac23\cs(2\varpi Z,-3)$ & $\zeta=0.2$
 & $8.31996$ & $8.31999$\\
$\frac13\sn(3\hh,\frac19)=\sn^{2}(f,-3)$ & $\zeta=0.2$
 & $0.0397470$ & $0.0397468$\\
\bottomrule
\end{tabular}
\end{table}

%=====================================================================
\section{The jets}\label{sec:jets}
%=====================================================================

This section is logically independent of \S\ref{sec:pfaff}: it shows
that the \emph{form} of \eqref{eq:star} is forced, and in particular
that the horizontal modulus must be $-3$. It therefore explains why
\eqref{eq:star} is the identity to look for, whereas
\S\ref{sec:pfaff} proves it.

Throughout, $w:=(1+i)\zeta=S+iT$ (so $w^{2}=2i\zeta^{2}$,
$\zeta^{4}=-\frac{w^{4}}{4}$), and $R_{k}:=\Real w^{k}$,
$I_{k}:=\Imag w^{k}$. From \eqref{eq:PNder},
\begin{equation}\label{eq:wseries}
  \frac{\dd\Pc}{\dd w}=\frac{1+\frac12w^{2}}{W},\quad
  \frac{\dd\Nc}{\dd w}=\frac{1-\frac12w^{2}}{W},\quad
  \frac{\dd(i\hh)}{\dd w}=\frac{w}{W},\quad
  W=\sqrt{1-\tfrac72w^{4}+\tfrac{1}{16}w^{8}},
\end{equation}
so $W^{-1}=1+\frac74w^{4}+\frac{73}{16}w^{8}+\cdots$ and
\begin{equation}\label{eq:series}
  \varpi p=R_{1}+\tfrac16R_{3}+\tfrac{7}{20}R_{5}+\tfrac18R_{7}
  +\cdots,\quad
  \varpi q=I_{1}-\tfrac16I_{3}+\tfrac{7}{20}I_{5}-\tfrac18I_{7}
  +\cdots,\quad
  \varpi Z=\tfrac12I_{2}+\tfrac{7}{24}I_{6}+\cdots
\end{equation}
We seek odd analytic $\mathfrak A(v)=v+A_{3}v^{3}+A_{5}v^{5}
+A_{7}v^{7}+\cdots$ and $\mathfrak C(v)=C_{1}v+C_{3}v^{3}+\cdots$
satisfying \eqref{eq:ansatz}.

\begin{lemma}[parity and grading]\label{lem:jet2}
All three series \eqref{eq:series} are odd in $S$ and odd in $T$, and
$S\leftrightarrow T$ interchanges $\varpi p,\varpi q$ and fixes
$\varpi Z$. Moreover $i\hh$ contains only the powers
$w^{2},w^{6},w^{10},\ldots$, so every homogeneous component of
$(\varpi Z)^{2j+1}$ has total degree $\equiv2\pmod4$. Consequently
\eqref{eq:ansatz} splits into
\begin{itemize}[leftmargin=1.5em]
\item degrees $2,6,10,\ldots$: equations determining
      $C_{1},C_{3},C_{5},\ldots$;
\item degrees $4,8,12,\ldots$: equations with vanishing right-hand
      side, constraining $\mathfrak A$ alone.
\end{itemize}
In degree $2$: $\varpi p\cdot\varpi q=ST=\varpi Z$, so $C_{1}=1$; in
particular $z-\frac12=\varpi(x^{2}-y^{2})+\cdots$, a saddle in the
$x^{2}-y^{2}$ orientation.
\end{lemma}

\begin{proof}
Under $S\leftrightarrow T$ one has $w\mapsto i\bar w$, so
$R_{3}\mapsto-I_{3}$, $R_{5}\mapsto I_{5}$, $R_{7}\mapsto-I_{7}$ and
$I_{2}\mapsto I_{2}$; compare \eqref{eq:series}. Also $I_{2}=2ST$.
\end{proof}

\begin{proposition}[degree $4$]\label{prop:deg4}
Degree $4$ holds iff $A_{3}=\frac13$.
\end{proposition}

\begin{proof}
The right side has no degree-$4$ part. On the left, the degree-$4$
terms are $-\frac16R_{1}I_{3}+\frac16R_{3}I_{1}
+A_{3}R_{1}I_{1}(R_{1}^{2}+I_{1}^{2})$, and
$R_{3}I_{1}-R_{1}I_{3}=-2ST(S^{2}+T^{2})$, so the total is
$\bigl(A_{3}-\frac13\bigr)ST(S^{2}+T^{2})$.
\end{proof}

\begin{proposition}[degree $6$]\label{prop:deg6}
Given $A_{3}=\frac13$, degree $6$ holds iff
\begin{equation}\label{eq:A5C3}
  A_{5}=-\tfrac{4}{15},\qquad C_{3}=-\tfrac53 .
\end{equation}
\end{proposition}

\begin{proof}
Let $\Phi_{j},\Psi_{j}$ be the degree-$j$ parts of
$\mathfrak A(\varpi p)$, $\mathfrak A(\varpi q)$, so
$\Psi_{j}(S,T)=\Phi_{j}(T,S)$. With $A_{3}=\frac13$,
\[
  \Phi_{1}=S,\qquad
  \Phi_{3}=\tfrac16R_{3}+\tfrac13S^{3}=\tfrac12S(S^{2}-T^{2}),\qquad
  \Phi_{5}=\tfrac{7}{20}R_{5}+\tfrac16S^{2}R_{3}+A_{5}S^{5},
\]
and, using $R_{5}=S^{5}-10S^{3}T^{2}+5ST^{4}$,
$\Phi_{5}=\bigl(\frac{31}{60}+A_{5}\bigr)S^{5}-4S^{3}T^{2}
+\frac74ST^{4}$. Hence
\[
  \Phi_{1}\Psi_{5}+\Phi_{5}\Psi_{1}
  =\Bigl(\tfrac{31}{60}+A_{5}+\tfrac74\Bigr)(S^{5}T+ST^{5})
   -8S^{3}T^{3},\qquad
  \Phi_{3}\Psi_{3}=-\tfrac14(S^{5}T+ST^{5})+\tfrac12S^{3}T^{3},
\]
so the degree-$6$ left side is
$\bigl(\frac{121}{60}+A_{5}\bigr)(S^{5}T+ST^{5})
-\frac{15}{2}S^{3}T^{3}$. On the right,
$I_{6}=6S^{5}T-20S^{3}T^{3}+6ST^{5}$ gives
$\frac{7}{24}I_{6}+C_{3}(ST)^{3}
=\frac74(S^{5}T+ST^{5})+\bigl(C_{3}-\frac{35}{6}\bigr)S^{3}T^{3}$.
Comparing the two independent monomials gives \eqref{eq:A5C3}.
\end{proof}

\begin{corollary}[the modulus is forced]\label{cor:modulus}
Suppose $\mathfrak A(v)=\lambda^{-1}\sn(\lambda v,m)$, so that
$A_{3}=-\frac{(1+m)\lambda^{2}}{6}$ and
$A_{5}=\frac{(1+14m+m^{2})\lambda^{4}}{120}$. Then
Propositions \ref{prop:deg4}--\ref{prop:deg6} give
\begin{equation}\label{eq:mquad}
  (1+m)\lambda^{2}=-2,\qquad (1+14m+m^{2})\lambda^{4}=-32
  \quad\Longrightarrow\quad 3m^{2}+10m+3=0 ,
\end{equation}
i.e.,\ $m=-3$ with $\lambda=1$, or $m=-\frac13$ with
$\lambda=i\sqrt3$. By \eqref{eq:imagmod} these give the \emph{same}
function; hence, in physical units,
\begin{equation}\label{eq:Afinal}
  \mathfrak A(v)=\sn(v,-3),\qquad\text{i.e.,}\qquad
  \sn(\varpi x,-3)\ \text{, etc.}
\end{equation}
Similarly $\mathfrak C(v)=\nu^{-1}\sn(\nu v,m_{2})$ forces
$(1+m_{2})\nu^{2}=10$, satisfied by $(m_{2},\nu)=(\frac19,3)$; but
degree $6$ alone does not determine $m_{2}$, which is settled instead
by Corollary \ref{cor:m2}.

The quadratic \eqref{eq:mquad} is \emph{palindromic},
$3m^{2}+10m+3=3(m+3)\bigl(m+\tfrac13\bigr)$, and its two roots are
precisely the two negative members $\{-3,-\frac13\}$ of the anharmonic
orbit of Remark \ref{rem:orbit}. This is not an accident: the root set
is forced to be stable under $m\mapsto1/m$, since by
\eqref{eq:imagmod} the moduli $m$ and $1/m$ define the same function
up to the rescaling $\lambda\mapsto i\sqrt{-m}\,\lambda$, which the
jets cannot distinguish.
\end{corollary}

\begin{proposition}[degree $8$]\label{prop:deg8}
Degree $8$ consists of two equations. One is a free consistency check,
which passes; the other gives $A_{7}=-\frac{7}{45}$, in agreement with
$A_{7}=-\frac{1+135m+135m^{2}+m^{3}}{5040}\big|_{m=-3}$.
\end{proposition}

\begin{proof}
By Lemma \ref{lem:jet2} the right side has no degree-$8$ part, and the
left side is $U+U^{\mathrm{sw}}$ with
$U:=\Phi_{1}\Psi_{7}+\Phi_{3}\Psi_{5}$ and $U^{\mathrm{sw}}$ its
$S\leftrightarrow T$ image. With $A_{3}=\frac13$,
$A_{5}=-\frac4{15}$,
\[
  \Phi_{7}=\tfrac18R_{7}+\tfrac{7}{20}S^{2}R_{5}
   +\tfrac1{36}SR_{3}^{2}-\tfrac29S^{4}R_{3}+A_{7}S^{7},
\]
whence, using $I_{7}=7S^{6}T-35S^{4}T^{3}+21S^{2}T^{5}-T^{7}$,
$I_{5}=5S^{4}T-10S^{2}T^{3}+T^{5}$, $I_{3}=3S^{2}T-T^{3}$,
\[
  \Psi_{7}=-\tfrac78S^{6}T+\tfrac{51}{8}S^{4}T^{3}
   -\tfrac{45}{8}S^{2}T^{5}
   +\Bigl(\tfrac{101}{360}+A_{7}\Bigr)T^{7},
\]
and $\Psi_{5}=\frac14T^{5}-4S^{2}T^{3}+\frac74S^{4}T$. The
coefficients of $U$ in the four monomials
$S^{7}T,\ S^{5}T^{3},\ S^{3}T^{5},\ ST^{7}$ are then
\begin{center}
\begin{tabular}{@{}lcccc@{}}
\toprule
 & $S^{7}T$ & $S^{5}T^{3}$ & $S^{3}T^{5}$ & $ST^{7}$\\
\midrule
$\Phi_{1}\Psi_{7}=S\Psi_{7}$ & $-\frac78$ & $\frac{51}{8}$
 & $-\frac{45}{8}$ & $\frac{101}{360}+A_{7}$\\
$\Phi_{3}\Psi_{5}$ & $\frac78$ & $-\frac{23}{8}$ & $\frac{17}{8}$
 & $-\frac18$\\
\midrule
$U$ & $0$ & $\frac72$ & $-\frac72$ & $\frac{7}{45}+A_{7}$\\
\bottomrule
\end{tabular}
\end{center}
Symmetrizing, the coefficient of $S^{5}T^{3}+S^{3}T^{5}$ in
$U+U^{\mathrm{sw}}$ is $\frac72-\frac72=0$ --- a genuine check, free of
$A_{7}$ --- and that of $S^{7}T+ST^{7}$ is $\frac{7}{45}+A_{7}$, which
must vanish.
\end{proof}

\begin{corollary}[uniqueness]\label{cor:unique}
The grading of Lemma \ref{lem:jet2} determines all $A_{2j+1}$,
$C_{2j+1}$ successively. Hence a multiplicatively separable
representation \eqref{eq:ansatz} is unique up to
$\mathfrak A\mapsto c\mathfrak A$,
$\mathfrak C\mapsto c^{2}\mathfrak C$, and its first four coefficients
are those of $\sn(\varpi\cdot,-3)$ and
$\frac13\sn(3\varpi\cdot,\frac19)$. Together with
Corollary \ref{cor:done} this says that \eqref{eq:star} is, up to that
normalization, the \emph{only} multiplicatively separable elliptic
representation of \textup{P}.
\end{corollary}

%=====================================================================
\section{Status, the computation, and the parallel with D}
\label{sec:status}
%=====================================================================

Theorem \ref{thm:main} is proved: \textup{(P0)} in
Theorem \ref{thm:P0proved}, the initial condition in
Theorem \ref{thm:diag}, and the two combined in
Corollary \ref{cor:done}. A full statement-by-statement audit,
including Part III, is deferred to \S\ref{sub:statusfull}.
The four boundary arcs play no part in the proof: they carry only one
condition each (planar geodesics rather than the straight lines of
\textup{[D, \S6]}), and are logically downstream of \textup{(P0)}.
Explicitly, on the arc $\omega=\iota(\zeta)$ the relations
\eqref{eq:iotaint} give $\varpi(p+q)+2\varpi Z=\varpi$ (the plane
$x+z=1$) and $(\star)$ becomes, with $a=\varpi p$, $c=\varpi Z$,
\begin{equation}\label{eq:arcform}
\sn(a,-3)\,\cd\bigl(a+2c,-3\bigr)
=\tfrac13\sn\bigl(3c,\tfrac19\bigr),
\end{equation}
one relation in two unknowns. The arcs retain value as an independent
numerical test in the region $|\zeta|\to\bt^{-1/2}$, where the
continuation of $W$ is delicate; such a test would require the
leg-by-leg transport of \textup{[D, Appendix, \texttt{contPath}]}
rather than principal branches.

\subsection{The scale, after Schwarz}\label{sub:scale}
With $\kappa_{\mathrm D}=2/K[1/4]$ \textup{[D, \S1.1]} and
$\kappa_{\mathrm P}=2/K[3/4]$ \eqref{eq:webnew},
\begin{equation}\label{eq:scaleratio}
  \frac{\kappa_{\mathrm P}}{\kappa_{\mathrm D}}
  =\frac{K[1/4]}{K[3/4]}=\frac{K[1/4]}{K'[1/4]}
  =0.7817009614\ldots
  =\bigl(1.2792615711710064662\ldots\bigr)^{-1},
\end{equation}
the reciprocal being the ratio $\lambda_{P}/\lambda_{D}$ of the edge
lengths of the circumscribing cubes. That this ratio must equal
$K'[1/4]/K[1/4]$ was proved by Schwarz in 1866, as the condition for
the adjoint hexagonal patches of \textup{P} and \textup{D} to have
equal area --- a prerequisite for Bonnet bending of either surface into
the other \cite[vol.~I, p.~88]{Schwarz}; the same constant governs the
area-to-volume comparison
$\mathcal A/V^{2/3}=3\bigl(K[3/4]/K[1/4]\bigr)^{2/3}$ of the
bcc cell \cite{GKD}. Thus \eqref{eq:scaleratio} is not an artifact of
our normalization but the classical constant, and it is an independent
external check on both papers' normalizations. Numerically
$\kappa_{\mathrm D}=1.1864152923\ldots$ and
$\kappa_{\mathrm P}=0.9274219746\ldots$ (Table \ref{tab:num}).

\begin{table}[ht]
\caption{The parallel with \textup{[D]}.}\label{tab:DP}
\begin{tabular}{@{}p{0.22\textwidth}p{0.35\textwidth}p{0.33\textwidth}@{}}
\toprule
 & D \cite{D-paper} & P (this paper)\\
\midrule
starting formulas & \eqref{eq:theta}--\eqref{eq:hclosed} &
  \emph{identical}\\
coordinates & $(\Imag\Pc,\Real\Nc,\Imag\hh)$ &
  $(\Real\Pc,\Imag\Nc,\Real\hh)$\\
normalization & $\kappa_{\mathrm D}=2/K[1/4]$ &
  $\kappa_{\mathrm P}=2/K[3/4]=2/K'[1/4]$\\
horizontal modulus & $\frac14$ & $-3\ (\sim\frac34)$\\
vertical modulus & $(2-\sqrt3)^{4}$, reduced to $\frac14$ by a
  $2$-isogeny & $\frac19$, reached from $\frac34$ by a $2$-isogeny\\
elliptic curve & $j=35152/9$ & \emph{the same} (Rem.\
  \ref{rem:orbit})\\
third argument & difference, $+\frac K2$ shift & sum, no shift\\
shape & $\mathcal T(x)\mathcal T(y)\mathcal T(1{-}z)
        +\mathcal T(x)+\mathcal T(y)+\mathcal T(1{-}z)=0$,
        $\mathcal T=\sn\dn/\cn$; equivalently
        $\Phi(x)\Phi(y)+1=\Phi(z)\bigl(\Phi(x)+\Phi(y)\bigr)$,
        $\Phi=|\mathcal T|$
      & $\sn(\varpi p,-3)\,\sn(\varpi q,-3)
        =\tfrac13\sn(3\varpi Z,\tfrac19)
        =\tfrac13\sn(3\varpi Z,\tfrac19)$; i.e.,
        $\Psi(x)=\Psi(y)+\Psi(z)$ with
        $\Psi=\am(\sqrt3\varpi(1+2\,\cdot),\tfrac43)$, and
        $\Cs(x)-\Cs(y)-\Cs(z)=3\,\Cs(x)\Cs(y)\Cs(z)$,
        $\Cs=\cn(4\varpi\,\cdot,\tfrac34)$\\
key lemma & $S_{a}D_{a}C_{b}+S_{b}D_{b}C_{a}
   =\sn(a{+}b)\dn(a{-}b)\Delta$ &
   $\Lambda_{m}(v/2)=\frac{\cn v+\dn v}{\sn v}$\\
square roots to resolve & none — \cite{D-paper}'s Lemma 9.1(b) 
   halves $(\log\dc)'$ rationally & \textbf{none} (in \textup{(P0)})\\
field $\KK$ & \emph{identical} — degree $\le16$ & $\zeta,\omega,
   r_{\pm}$ only; degree $\le16$\\
sign structure & real fcc invariance needs $\Phi=|\mathcal T|$;
   provable only in the $(p,q,r)$ variables \textup{[D, \S4.2]}
   & sign-honest: bcc invariance holds directly on $(\star)$
   (Prop.\ \ref{prop:group}(c))\\
boundary data per arc & two (straight line) & one (planar
geodesic); but the four \emph{rhombus} edges are immediate
(Cor.\ \ref{cor:rhombus})\\
images of the four corners of $\Om$ & regular tetrahedron, edge
   $\sqrt2$ \textup{[D, \S5.4]} & unit square in $z=\frac12$
   (Prop.\ \ref{prop:corners}(b))\\
final reduction & \textbf{(D0)}: antisymmetric, rational, 
   degree $\le16$; verified on six curves, \emph{open}
   & \textup{(P0)}: antisymmetric, rational, degree $\le16$;
   \textbf{proved} (Thm.\ \ref{thm:P0proved})\\
\bottomrule
\end{tabular}
\end{table}

\begin{remark}[open problems]\label{rem:open}
\begin{enumerate}[label=\textup{(\arabic*)},leftmargin=2.4em]
\item \emph{A conceptual proof of \textup{(P0)}.}
      Theorem \ref{thm:P0proved} rests on the certificate
      \eqref{eq:certificate}, which is finite and verifiable by
      polynomial expansion (Remark \ref{rem:cert}) but not
      illuminating. The P-analog of
      \textup{[D, \S6.3, the twisted numerator]} is the natural
      target, and the most promising handle is the pair of separated
      variables $A\pm B=2f,2g$ of Theorem \ref{thm:dictf}, in which
      the data is again rational --- indeed
      $\sn(2f,-3)=2\zeta(1+\zeta^{2})/W$, with $g(\zeta)=f(i\zeta)$.
\item A divisor-theoretic proof, comparing the divisors of $\Theta$
      and of $\frac13\sn(3\varpi Z,\frac19)$ on a finite cover of
      $W^{2}=1+14\zeta^{4}+\zeta^{8}$; cf.\
      \textup{[D, \S7.7, the divisor-theoretic route]}. This would
      also explain \emph{why} the height can be eliminated.
\item \emph{The global converse.} The \emph{local} converse is now
Corollary \ref{cor:localconv}: \eqref{eq:star} \emph{cuts out}
the surface near every point of $X(\Om^{\circ})$, because
$\{\Fs=0\}$ is a regular embedded real-analytic surface
(Theorem \ref{thm:Smin}) and $X$ is an immersion of a
two-manifold. What remains open is the \emph{global} statement:
injectivity of $X$ on $\Om^{\circ}$ --- the Gauss map
$G=\zeta$ is injective, which together with the boundary
analysis of \S\ref{sec:sym} should suffice --- and the
identification of the connected component of $\{\Fs=0\}$
containing $X(\Om)$. For the record, the gradient condition
implicit in Corollary \ref{cor:pfaffsep} is
$\Fs_{x}:\Fs_{y}:\Fs_{z}=2u:2v:\bigl(u^{2}+v^{2}-1\bigr)$,
equivalently the single complex equation
$\Fs_{x}(1-\zeta^{2})+i\Fs_{y}(1+\zeta^{2})
+2\Fs_{z}\zeta=0$. The corresponding global statement is also
open in \textup{[D]}.
\item Extension to the associate family: for which Bonnet angle $\psi$
      does $X_{\psi}$ admit an implicit representation of the same
      elliptic type? By Corollary \ref{cor:unique} the answer is
      rigid at $\psi=0$; the gyroid,
      $\psi=38.0147^{\circ}\ldots$ \cite{GKG}, is the natural test
      case.
\item \emph{Reciprocally}, the device of \S\ref{sec:pfaff} has now been
carried out for \textup{D}. Because $\widetilde\Theta$ there
satisfies $\widetilde\Theta(K_{1}-w)=\widetilde\Theta(w)$, the
combination $\omega\partial_{\zeta}+\zeta\partial_{\omega}$
eliminates the height entirely; with the rational height
dictionary in place of the modulus-$\bt^{-4}$ one, and the
rational half-argument formula for $(\log\dc)'$ in place of
\textup{[D]}'s half-argument radicals, \textup{[D]}'s field drops
to ours and the reduced identity \textup{(D0)} is antisymmetric
and rational of degree $\le16$ --- the exact analogue of
\textup{(P0)}. \textup{(D0)} itself remains open; \textup{[D]}
proves its main theorem instead by Bj\"orling. See \textup{[D, \S9]}.
\end{enumerate}
\end{remark}

\part*{PART III. THE SEPARABLE FORM}

%=====================================================================
\section{Three constants and two transformations}\label{sec:constIII}
%=====================================================================

Parts I and II proved $(\star)$ in the multiplicative form
\eqref{eq:star}. Part III shows that $(\star)$ is, after all,
\emph{additively separable} --- anticipated in Remark \ref{rem:forms} ---
and that the separated form collapses to a single trilinear polynomial
identity in one bounded function of one variable. The two extra
ingredients are one constant and one transformation, both elementary.

\begin{lemma}[an incomplete integral that is complete]\label{lem:a3}
Put $\mathfrak a:=F\bigl[\tfrac\pi3,\tfrac43\bigr]$. Then
\begin{equation}\label{eq:aval}
  \mathfrak a=\tfrac{\sqrt3}{2}K\bigl[\tfrac34\bigr]
   =\sqrt3\,K[-3]=\sqrt3\,\varpi=1.8675973\ldots
\end{equation}
\end{lemma}

\begin{proof}
By \eqref{eq:F},
$\mathfrak a=\int_{0}^{\pi/3}
(1-\tfrac43\sin^{2}\vartheta)^{-1/2}\dd\vartheta$. The upper limit is
exactly where the integrand becomes singular, since
$\sin\tfrac\pi3=\tfrac{\sqrt3}{2}=m^{-1/2}$ at $m=\tfrac43$; the
singularity is of type $(\tfrac\pi3-\vartheta)^{-1/2}$, hence
integrable, and $\mathfrak a$ is finite. Substitute
$\sin\vartheta=\tfrac{\sqrt3}{2}\sin\varphi$, an increasing
real-analytic bijection of $[0,\tfrac\pi2]$ onto $[0,\tfrac\pi3]$. Then
\[
  \cos\vartheta\dd\vartheta
   =\tfrac{\sqrt3}{2}\cos\varphi\dd\varphi,\qquad
  \cos\vartheta=\sqrt{1-\tfrac34\sin^{2}\varphi},\qquad
  1-\tfrac43\sin^{2}\vartheta=1-\sin^{2}\varphi=\cos^{2}\varphi,
\]
so that
\[
  \mathfrak a=\int_{0}^{\pi/2}
   \frac{\tfrac{\sqrt3}{2}\cos\varphi\dd\varphi}
        {\sqrt{1-\tfrac34\sin^{2}\varphi}\;\cos\varphi}
  =\tfrac{\sqrt3}{2}\int_{0}^{\pi/2}
   \frac{\dd\varphi}{\sqrt{1-\tfrac34\sin^{2}\varphi}}
  =\tfrac{\sqrt3}{2}K\bigl[\tfrac34\bigr]
  =\tfrac{\sqrt3}{2}\cdot2\varpi,
\]
the last step by Proposition \ref{prop:web}.
\end{proof}

\begin{lemma}[reciprocal parameter]\label{lem:recip}
Let $m>1$ and $\bar m:=1/m$. Then
\begin{equation}\label{eq:recip}
  \sn(u,m)=\frac{1}{\sqrt m}\,\sn\bigl(\sqrt m\,u,\bar m\bigr),\qquad
  \cn(u,m)=\dn\bigl(\sqrt m\,u,\bar m\bigr),\qquad
  \dn(u,m)=\cn\bigl(\sqrt m\,u,\bar m\bigr).
\end{equation}
At $m=\tfrac43$, so $\sqrt m=\tfrac{2}{\sqrt3}$ and
$\bar m=\tfrac34$:
\begin{equation}\label{eq:recip43}
  \sn\bigl(u,\tfrac43\bigr)
   =\tfrac{\sqrt3}{2}\sn\bigl(\tfrac{2u}{\sqrt3},\tfrac34\bigr),\quad
  \cn\bigl(u,\tfrac43\bigr)
   =\dn\bigl(\tfrac{2u}{\sqrt3},\tfrac34\bigr),\quad
  \dn\bigl(u,\tfrac43\bigr)
   =\cn\bigl(\tfrac{2u}{\sqrt3},\tfrac34\bigr).
\end{equation}
\end{lemma}

\begin{proof}
Put $\sigma(u):=m^{-1/2}\sn(\sqrt m\,u,\bar m)$, so $\sigma(0)=0$ and
$\sigma'(0)=1$. Writing $S,C,D$ for the Jacobi triple at
$(\sqrt m\,u,\bar m)$ and using $\sn'=\cn\dn$ together with
\eqref{eq:snF},
\[
  \sigma'=CD,\qquad
  C^{2}=1-S^{2}=1-m\sigma^{2},\qquad
  D^{2}=1-\bar mS^{2}=1-\tfrac1m\cdot m\sigma^{2}=1-\sigma^{2},
\]
so $\sigma'^{2}=(1-\sigma^{2})(1-m\sigma^{2})$ with $\sigma(0)=0$,
$\sigma'(0)=1$: the initial value problem characterising
$\sn(\cdot,m)$. Hence $\sigma=\sn(\cdot,m)$ by uniqueness. The two
displayed formulas then read $\cn^{2}(u,m)=D^{2}$ and
$\dn^{2}(u,m)=C^{2}$, and the signs are fixed by the common value
$+1$ at $u=0$ together with continuity.
\end{proof}

\begin{convention}[the amplitude at parameter $>1$]\label{conv:am}
For real $u$ and $m>1$ we set
\begin{equation}\label{eq:am}
  \am(u,m):=\arcsin\sn(u,m),
\end{equation}
the \emph{bounded} (oscillating) branch. By \eqref{eq:recip43},
$|\sn(u,\tfrac43)|\le\tfrac{\sqrt3}{2}<1$ and
$\cn(u,\tfrac43)=\dn(\tfrac{2u}{\sqrt3},\tfrac34)>0$ for all real
$u$; consequently \eqref{eq:am} is real-analytic on $\RR$, odd,
$4\mathfrak a$-periodic, takes values in
$[-\tfrac\pi3,\tfrac\pi3]$, and satisfies $\sin\am=\sn$,
$\cos\am=\cn$, $\frac{\dd}{\dd u}\am=\dn(u,\tfrac43)$ --- the last
\emph{changing sign}, at $u\in\mathfrak a+2\mathfrak a\ZZ$. 
Mathematica's \texttt{JacobiAmplitude}$[u,\tfrac43]$ agrees with
\eqref{eq:am} on the fundamental strip $|u|\le2\mathfrak a$, and
elsewhere differs from it by a multiple of $2\pi$: it accumulates
$2\pi$ per full period $4\mathfrak a$, whereas \eqref{eq:am} is
genuinely $4\mathfrak a$-periodic. In Appendix \ref{app:code} the
comparison is therefore made on $t\in[-\tfrac32,\tfrac12]$ only.
\end{convention}

\begin{remark}[the branch matters]\label{rem:branch}
The alternative, monotone reading of $\am(\cdot,\tfrac43)$ ---
obtained by continuing $\int_{0}^{u}\dn$ past its zeros with a
$\pi$-jump --- defines a \emph{different} function, and with it the
results of \S\ref{sec:masterIII} are false. The cleanest test is
$\am(2\mathfrak a,\tfrac43)$: under \eqref{eq:am} it is
$\arcsin\sn(2K[\tfrac34],\tfrac34)=0$, whereas the monotone reading
returns $\pi$. Since Theorem \ref{thm:trilin} is an identity among
three such amplitudes, an error of $\pi$ in one slot destroys it;
concretely, the two rhombus edges through
$(\tfrac12,\tfrac12,\tfrac12)$ would fail
(Corollary \ref{cor:rhombus}). Mathematica's implementation is a 
third reading, agreeing with \eqref{eq:am} modulo $2\pi$ but 
not equal to it — see Convention \ref{conv:am}.
\end{remark}

\begin{remark}[$\mathfrak a$ and \textup{[KO]}]\label{rem:aKO}
$\mathfrak a$ is the constant $\approx1.8676$ located numerically in
\cite{KO}, where the four vertices of an octahedron are asserted to lie
on the surface; Lemma \ref{lem:a3} makes those locations exact
(Corollary \ref{cor:rhombus}). It is also the quarter period of
$\sn(\cdot,\tfrac43)$: by \eqref{eq:recip43},
$\sn(\mathfrak a,\tfrac43)
=\tfrac{\sqrt3}{2}\sn(K[\tfrac34],\tfrac34)=\tfrac{\sqrt3}{2}$ and
$\dn(\mathfrak a,\tfrac43)=\cn(K[\tfrac34],\tfrac34)=0$.
\end{remark}

%=====================================================================
\section{The function $\Psi$}\label{sec:PsiIII}
%=====================================================================

\begin{definition}\label{def:Psi}
For $t\in\RR$ put
\begin{equation}\label{eq:Psidef}
  \Psi(t):=\arcsin\Bigl(\tfrac{\sqrt3}{2}\,\cn(2\varpi t,-3)\Bigr)
  \ \in\ \bigl[-\tfrac\pi3,\tfrac\pi3\bigr],
  \qquad
  \Cs(t):=\cn\bigl(4\varpi t,\tfrac34\bigr)=\cd(2\varpi t,-3),
\end{equation}
the second equality by Proposition \ref{prop:gauss}, since
$\cd(u,-3)=\cn(u,-3)/\dn(u,-3)
=\cd(2u,\tfrac34)\dn(2u,\tfrac34)=\cn(2u,\tfrac34)$; note
$4\varpi=2K[\tfrac34]$.
\end{definition}

\footnote{The symbol $\Psi$ is used, unsubscripted, only in Part III.
The $\Psi_{j}$ of \S\ref{sec:jets} are the degree-$j$ homogeneous parts
of $\mathfrak A(\varpi q)$ and are unrelated; they are always
subscripted.}

\begin{proposition}[three faces of $\Psi$]\label{prop:Psifaces}
For all real $t$,
\begin{equation}\label{eq:Psifaces}
  \Psi(t)
  =\arcsin\Bigl(\tfrac{\sqrt3}{2}
    \cd\bigl(4\varpi t,\tfrac34\bigr)\Bigr)
  =\arctan\bigl(\sqrt3\,\Cs(t)\bigr)
  =\am\bigl(\mathfrak a(1+2t),\tfrac43\bigr),
\end{equation}
the amplitude read as in Convention \ref{conv:am}. Moreover, with
$\sn,\cn,\dn$ evaluated at $(2\varpi t,-3)$,
\begin{equation}\label{eq:Psitrig}
  \sin\Psi=\tfrac{\sqrt3}{2}\cn,\qquad
  \cos\Psi=\tfrac12\dn\ \ \bigl(\ge\tfrac12>0\bigr),\qquad
  \tan\Psi=\sqrt3\,\cd=\sqrt3\,\Cs .
\end{equation}
\end{proposition}

\begin{proof}
The first equality of \eqref{eq:Psifaces} is
$\cn(2\varpi t,-3)=\cd(4\varpi t,\tfrac34)$, i.e.,\
Proposition \ref{prop:gauss}.

For $\cos\Psi$: by \eqref{eq:Psidef},
$\cos\Psi=\sqrt{1-\tfrac34\cn^{2}}$, and, using $\cn^{2}=1-\sn^{2}$
and $\dn^{2}=1+3\sn^{2}$ at $m=-3$ (both from \eqref{eq:snF}),
\begin{equation}\label{eq:cosPsi}
  1-\tfrac34\cn^{2}=1-\tfrac34\bigl(1-\sn^{2}\bigr)
  =\tfrac14\bigl(1+3\sn^{2}\bigr)=\tfrac{\dn^{2}}{4},
\end{equation}
whence $\cos\Psi=\tfrac12\dn$, the positive root being correct because
$|\Psi|\le\tfrac\pi3$; and $\dn(\cdot,-3)\ge1$ gives
$\cos\Psi\ge\tfrac12$. Dividing,
$\tan\Psi=\sqrt3\cn/\dn=\sqrt3\cd(2\varpi t,-3)=\sqrt3\Cs(t)$, which
is the second equality of \eqref{eq:Psifaces} since
$|\Psi|<\tfrac\pi2$.

For the third: by Convention \ref{conv:am} and \eqref{eq:recip43} with
$u=\mathfrak a(1+2t)$, using
$\tfrac{2\mathfrak a}{\sqrt3}=2\varpi=K[\tfrac34]$
(Lemma \ref{lem:a3}),
\[
  \sn\bigl(\mathfrak a(1+2t),\tfrac43\bigr)
  =\tfrac{\sqrt3}{2}
   \sn\bigl(K[\tfrac34]+4\varpi t,\tfrac34\bigr)
  =\tfrac{\sqrt3}{2}\cd\bigl(4\varpi t,\tfrac34\bigr),
\]
by the quarter-period shift $\sn(K+v)=\cd v$ \cite[122.01]{BF}. Taking
$\arcsin$ and comparing with the first equality of
\eqref{eq:Psifaces} gives the claim.
\end{proof}

\begin{proposition}[the calculus of $\Psi$]\label{prop:Psicalc}
$\Psi$ is real-analytic on $\RR$ and
\begin{equation}\label{eq:Psider}
  \Psi'(t)=-2\sqrt3\,\varpi\,\sn(2\varpi t,-3),\qquad
  \Psi'(t)^{2}=4\varpi^{2}\bigl(1+2\cos2\Psi(t)\bigr),\qquad
  \Psi''(t)=-8\varpi^{2}\sin2\Psi(t).
\end{equation}
Moreover
\begin{equation}\label{eq:Psisym}
  \Psi(-t)=\Psi(t),\qquad \Psi(t+1)=-\Psi(t),\qquad
  \Psi(t+2)=\Psi(t),\qquad \Psi(1-t)=-\Psi(t),
\end{equation}
\begin{equation}\label{eq:Psivals}
  \Psi(0)=\tfrac\pi3,\quad \Psi(\tfrac14)=\tfrac\pi4,\quad
  \Psi(\tfrac12)=0,\quad \Psi(\tfrac34)=-\tfrac\pi4,\quad
  \Psi(1)=-\tfrac\pi3,\qquad |\Psi|\le\tfrac\pi3,
\end{equation}
and $\Psi$ decreases strictly from $\tfrac\pi3$ to $-\tfrac\pi3$ on
$[0,1]$; in particular
$\Psi:[0,1]\to[-\tfrac\pi3,\tfrac\pi3]$ is a real-analytic
diffeomorphism onto, with inverse $\Psi^{-1}$. Finally
$\Psi(t)=\tfrac\pi3$ iff $t\in2\ZZ$, and $\Psi(t)=-\tfrac\pi3$ iff
$t\in2\ZZ+1$.
\end{proposition}

\begin{proof}
Real-analyticity: by \eqref{eq:Psidef} the argument of $\arcsin$ lies
in $[-\tfrac{\sqrt3}{2},\tfrac{\sqrt3}{2}]$, a compact subset of
$(-1,1)$, on which $\arcsin$ is analytic.

Differentiating $\sin\Psi=\tfrac{\sqrt3}{2}\cn(2\varpi t,-3)$ and
using $\cn'=-\sn\dn$ \cite[\S22.13]{WW} together with
$\cos\Psi=\tfrac12\dn$,
\[
  \Psi'\cdot\tfrac{\dn}{2}
  =\tfrac{\sqrt3}{2}\cdot2\varpi\cdot\bigl(-\sn\dn\bigr)
  \quad\Longrightarrow\quad
  \Psi'=-2\sqrt3\,\varpi\,\sn(2\varpi t,-3),
\]
the first part of \eqref{eq:Psider}. By \eqref{eq:cosPsi},
$\sin^{2}\Psi=\tfrac34(1-\sn^{2})$, i.e.,\
$\sn^{2}=1-\tfrac43\sin^{2}\Psi$, so
\[
  \Psi'^{2}=12\varpi^{2}\sn^{2}
  =12\varpi^{2}\bigl(1-\tfrac43\sin^{2}\Psi\bigr)
  =4\varpi^{2}\bigl(3-4\sin^{2}\Psi\bigr)
  =4\varpi^{2}\bigl(1+2\cos2\Psi\bigr),
\]
using $\cos2\Psi=1-2\sin^{2}\Psi$. Differentiating this last identity
gives $2\Psi'\Psi''=-16\varpi^{2}\sin2\Psi\cdot\Psi'$; since $\Psi''$
is continuous, $\Psi''=-8\varpi^{2}\sin2\Psi$ holds wherever
$\Psi'\ne0$ and hence everywhere.

\eqref{eq:Psisym}: $\cn$ is even, so $\Psi$ is even; and
$\cn(u+2K,-3)=-\cn(u,-3)$ with $K=K[-3]=\varpi$ \cite[122.01]{BF}
gives $\Psi(t+1)=-\Psi(t)$. The remaining two follow.

\eqref{eq:Psivals}: $\cn(0,-3)=1$ gives
$\Psi(0)=\arcsin\tfrac{\sqrt3}{2}=\tfrac\pi3$;
$\cn(\varpi,-3)=\cn(K,-3)=0$ gives $\Psi(\tfrac12)=0$;
$\cn(2\varpi,-3)=-1$ gives $\Psi(1)=-\tfrac\pi3$. For $t=\tfrac14$ use
the half-argument values \cite[124.01]{BF}, valid at any modulus, with
$k'=\sqrt{1-m}=2$:
\[
  \cn^{2}\bigl(\tfrac K2,-3\bigr)=\frac{k'}{1+k'}=\tfrac23,
  \qquad\text{so}\qquad
  \sin\Psi(\tfrac14)=\tfrac{\sqrt3}{2}\sqrt{\tfrac23}
   =\tfrac{1}{\sqrt2},
\]
i.e.,\ $\Psi(\tfrac14)=\tfrac\pi4$; equivalently
$\Cs(\tfrac14)=\tfrac{1}{\sqrt3}$, since
$\dn^{2}(\tfrac K2,-3)=k'=2$. Then
$\Psi(\tfrac34)=-\Psi(-\tfrac14)=-\Psi(\tfrac14)$ by
\eqref{eq:Psisym}.

Monotonicity: by \eqref{eq:Psider}, $\Psi'<0$ exactly where
$\sn(2\varpi t,-3)>0$, i.e.,\ for $2\varpi t\in(0,2K)$, i.e.,\
$t\in(0,1)$.

Finally $\Psi(t)=\pm\tfrac\pi3$ iff $\cn(2\varpi t,-3)=\pm1$; since
$\cn(\cdot,-3)$ has period $4\varpi$ and attains $+1$ only on
$4\varpi\ZZ$ and $-1$ only on $2\varpi+4\varpi\ZZ$, this is
$t\in2\ZZ$, resp.\ $t\in2\ZZ+1$.
\end{proof}

\begin{remark}[three normalizations]\label{rem:norm3}
Three normalizations of one function occur and should be kept apart.
The \emph{lattice} normalization is \eqref{eq:Psidef}, of antiperiod
$1$; the \emph{Kim--Ogata} normalization is $\am(\cdot,\tfrac43)$, of
antiperiod $2\mathfrak a$; and the \emph{ODE} normalization, in which
\eqref{eq:Psider} reads $\varphi'^{2}=1+2\cos2\varphi$, is
$\varphi(t)=\Psi\bigl(t/(2\varpi)\bigr)
=\Psi\bigl(t/K[\tfrac34]\bigr)$, of antiperiod $K[\tfrac34]$. The last
is the one in which the comparison with \textup{D} is made; there the
companion equation is $\varphi'^{2}=1+2\cosh2\varphi$, with
$\varphi=\log(\sn\dn/\cn)$ at modulus $\tfrac14$ and antiperiod
$K[\tfrac14]$. See \textup{[D]}.
\end{remark}

%=====================================================================
\section{The two dictionaries}\label{sec:dictIII}
%=====================================================================

\begin{proposition}[the horizontal dictionary]\label{prop:hdict}
For all real $t$,
\begin{equation}\label{eq:hdict}
  \arctan\Bigl(\sqrt3\,\sn^{2}(\varpi t,-3)\Bigr)
  =\tfrac12\Bigl(\tfrac\pi3-\Psi(t)\Bigr),
  \qquad\text{i.e.,}\qquad
  \Psi(t)=\tfrac\pi3
   -2\arctan\Bigl(\sqrt3\,\sn^{2}(\varpi t,-3)\Bigr).
\end{equation}
\end{proposition}

\begin{proof}
Write $s:=\sn(\varpi t,-3)$ and $\tau:=\sqrt3\,s^{2}$. Since
$|s|\le1$ we have $\tau\in[0,\sqrt3]$, so
$\psi:=\arctan\tau\in[0,\tfrac\pi3]$.

By the duplication formula \cite[124.02]{BF} at $m=-3$, where
$1-m\sn^{4}=1+3s^{4}$,
\[
  \cn(2\varpi t,-3)=\frac{\cn^{2}-\sn^{2}\dn^{2}}{1+3s^{4}}
  =\frac{(1-s^{2})-s^{2}(1+3s^{2})}{1+3s^{4}}
  =\frac{1-2s^{2}-3s^{4}}{1+3s^{4}}
  =\frac{(1+s^{2})(1-3s^{2})}{1+3s^{4}} .
\]
Now $1+s^{2}=\dfrac{\sqrt3+\tau}{\sqrt3}$, $1-3s^{2}=1-\sqrt3\,\tau$
and $1+3s^{4}=1+\tau^{2}$, so
\[
  \tfrac{\sqrt3}{2}\cn(2\varpi t,-3)
  =\frac{(\sqrt3+\tau)(1-\sqrt3\,\tau)}{2(1+\tau^{2})}
  =\frac{\sqrt3(1-\tau^{2})-2\tau}{2(1+\tau^{2})}
  =\tfrac{\sqrt3}{2}\cos2\psi-\tfrac12\sin2\psi
  =\sin\Bigl(\tfrac\pi3-2\psi\Bigr),
\]
using $\dfrac{1-\tau^{2}}{1+\tau^{2}}=\cos2\psi$,
$\dfrac{2\tau}{1+\tau^{2}}=\sin2\psi$, and
$\bigl(\sin\tfrac\pi3,\cos\tfrac\pi3\bigr)
=\bigl(\tfrac{\sqrt3}{2},\tfrac12\bigr)$. Since
$\psi\in[0,\tfrac\pi3]$ we have
$\tfrac\pi3-2\psi\in[-\tfrac\pi3,\tfrac\pi3]
\subset[-\tfrac\pi2,\tfrac\pi2]$, on which $\arcsin\circ\sin$ is the
identity; comparing with \eqref{eq:Psidef} gives
$\Psi(t)=\tfrac\pi3-2\psi$.
\end{proof}

\begin{lemma}[the $\tfrac19\to-3$ reduction]\label{lem:19}
For all real $\vartheta$,
\begin{equation}\label{eq:19}
  \tfrac13\sn\bigl(3\vartheta,\tfrac19\bigr)
  =\frac{\sn(2\vartheta,-3)}{1+\dn(2\vartheta,-3)} .
\end{equation}
\end{lemma}

\begin{proof}
This is \eqref{eq:SigmaD2} with $\hh$ replaced by an arbitrary
argument; we repeat the three lines. Put
$\Sigma:=\tfrac13\sn(3\vartheta,\tfrac19)$, so $|\Sigma|\le\tfrac13$
and $3\Sigma^{2}\le\tfrac13<1$. Lemma \ref{lem:landen} with
$k^{2}=\tfrac34$, $k_{1}=\tfrac13$, $v=3\vartheta$,
$u=\tfrac43v=4\vartheta$ and $s=\sn(3\vartheta,\tfrac19)=3\Sigma$
gives
\[
  \sn\bigl(4\vartheta,\tfrac34\bigr)
   =\frac{4\Sigma}{1+3\Sigma^{2}},\qquad
  \dn\bigl(4\vartheta,\tfrac34\bigr)
   =\frac{1-3\Sigma^{2}}{1+3\Sigma^{2}} ,
\]
and then Proposition \ref{prop:gauss} with $u=2\vartheta$ gives
\[
  \sn(2\vartheta,-3)=\tfrac12\sd\bigl(4\vartheta,\tfrac34\bigr)
   =\frac{2\Sigma}{1-3\Sigma^{2}},\qquad
  \dn(2\vartheta,-3)=\nd\bigl(4\vartheta,\tfrac34\bigr)
   =\frac{1+3\Sigma^{2}}{1-3\Sigma^{2}} .
\]
Hence $\dfrac{\sn(2\vartheta,-3)}{1+\dn(2\vartheta,-3)}
=\dfrac{2\Sigma}{(1-3\Sigma^{2})+(1+3\Sigma^{2})}=\Sigma$.
\end{proof}

\begin{proposition}[the vertical dictionary]\label{prop:vdict}
For all real $z$, with
$\Sigma:=\tfrac13\sn\bigl(3\varpi(z-\tfrac12),\tfrac19\bigr)$,
\begin{equation}\label{eq:vdict}
  \arctan\bigl(\sqrt3\,\Sigma\bigr)=-\tfrac12\Psi(z),
  \qquad\text{i.e.,}\qquad
  \Sigma=-\tfrac{1}{\sqrt3}\tan\tfrac{\Psi(z)}{2}.
\end{equation}
\end{proposition}

\begin{proof}
Apply Lemma \ref{lem:19} with $\vartheta=\varpi(z-\tfrac12)$, so that
$2\vartheta=2\varpi z-\varpi=2\varpi z-K$ with $K=K[-3]=\varpi$. The
quarter-period shifts \cite[122.01]{BF} at $m=-3$, where
$k'=\sqrt{1-m}=2$, give
\[
  \sn(v-K,-3)=-\cd(v,-3),\qquad
  \dn(v-K,-3)=k'\nd(v,-3)=\frac{2}{\dn(v,-3)},
\]
so, all functions evaluated at $v=2\varpi z$ and $m=-3$,
\[
  \Sigma=\frac{-\cd}{1+2/\dn}
        =\frac{-\cn/\dn}{(\dn+2)/\dn}
        =\frac{-\cn}{\dn+2} .
\]
Insert $\cn=\tfrac{2}{\sqrt3}\sin\Psi(z)$ and
$\dn=2\cos\Psi(z)$ from \eqref{eq:Psitrig}:
\[
  \sqrt3\,\Sigma=\frac{-2\sin\Psi(z)}{2\cos\Psi(z)+2}
   =-\frac{\sin\Psi(z)}{1+\cos\Psi(z)}
   =-\tan\frac{\Psi(z)}{2} .
\]
Since $\bigl|\tfrac{\Psi(z)}{2}\bigr|\le\tfrac\pi6<\tfrac\pi2$, taking
$\arctan$ gives \eqref{eq:vdict}.
\end{proof}

\begin{corollary}[an overdue separable form]\label{cor:sepform}
Put
\begin{equation}\label{eq:Lambdas}
  \lambda(t):=\sqrt3\,\sn^{2}(\varpi t,-3)\in[0,\sqrt3], 
  \Lambda_{\mathrm P}(t):=\arctan\lambda(t)
   \in\bigl[0,\tfrac\pi3\bigr], 
  \Lambda_{\mathrm Z}(z):=\arctan\bigl(\sqrt3\,\Sigma\bigr)
   \in\bigl[-\tfrac\pi6,\tfrac\pi6\bigr].
\end{equation}
Then $\Lambda_{\mathrm P}$ and $\Lambda_{\mathrm Z}$ are bounded and
real-analytic on $\RR$, with
\begin{equation}\label{eq:LambdaPsi}
  \Lambda_{\mathrm P}(t)=\tfrac\pi6-\tfrac12\Psi(t),\qquad
  \Lambda_{\mathrm Z}(z)=-\tfrac12\Psi(z),
\end{equation}
and $(\star)$ is equivalent to each of
\begin{align}
  \Lambda_{\mathrm P}(x)-\Lambda_{\mathrm P}(y)
   &=\Lambda_{\mathrm Z}(z),\label{eq:sep1}\\
  \Psi(x)&=\Psi(y)+\Psi(z),\label{eq:sep2}\\
  \arctan\lambda(y)+\arctan\lambda(z)-\arctan\lambda(x)
   &=\tfrac\pi6 .\label{eq:sep3}
\end{align}
In the symmetrized coordinates $(x,1-y,1-z)$, \eqref{eq:sep3} becomes
$\sum_{i}\arctan\lambda_{i}=\tfrac\pi2$, i.e.,\ the quadratic relation
\begin{equation}\label{eq:sep4}
  \lambda_{1}\lambda_{2}+\lambda_{2}\lambda_{3}
   +\lambda_{3}\lambda_{1}=1 .
\end{equation}
\end{corollary}

\begin{proof}
\eqref{eq:LambdaPsi} is Propositions \ref{prop:hdict} and
\ref{prop:vdict}; boundedness and analyticity follow from those of
$\Psi$ (Proposition \ref{prop:Psicalc}). For the equivalence with
$(\star)$: by \eqref{eq:star2}, i.e.,\ the product formula
\cite[123.02]{BF} at $m=-3$, and with
$u:=\lambda(x)=\tan\Lambda_{\mathrm P}(x)\ge0$,
$v:=\lambda(y)=\tan\Lambda_{\mathrm P}(y)\ge0$,
\begin{equation}\label{eq:prodtan}
  \sn\bigl(\varpi(x{+}y),-3\bigr)\sn\bigl(\varpi(x{-}y),-3\bigr)
  =\frac{\sn^{2}(\varpi x)-\sn^{2}(\varpi y)}
        {1+3\sn^{2}(\varpi x)\sn^{2}(\varpi y)}
  =\frac{1}{\sqrt3}\cdot\frac{u-v}{1+uv}
  =\frac{1}{\sqrt3}
   \tan\bigl(\Lambda_{\mathrm P}(x)-\Lambda_{\mathrm P}(y)\bigr),
\end{equation}
the tangent subtraction formula applying \emph{without any modular
ambiguity} because $1+uv\ge1>0$ and
$\Lambda_{\mathrm P}(x)-\Lambda_{\mathrm P}(y)
\in(-\tfrac\pi2,\tfrac\pi2)$. Likewise the right side of $(\star)$ is
$\tfrac{1}{\sqrt3}\tan\Lambda_{\mathrm Z}(z)$ by \eqref{eq:Lambdas}.
Since $\tan$ is injective on $(-\tfrac\pi2,\tfrac\pi2)$ and both
$\Lambda_{\mathrm P}(x)-\Lambda_{\mathrm P}(y)$ and
$\Lambda_{\mathrm Z}(z)$ lie there,
$(\star)\iff\eqref{eq:sep1}$. Then \eqref{eq:sep2} follows from
\eqref{eq:LambdaPsi}, the constants $\tfrac\pi6$ cancelling, and
\eqref{eq:sep3} is \eqref{eq:sep2} rewritten by \eqref{eq:LambdaPsi}.
Finally $\Psi(1-t)=-\Psi(t)$ (Proposition \ref{prop:Psicalc}) turns
\eqref{eq:sep2} into $\Psi(x)+\Psi(1-y)+\Psi(1-z)=0$, i.e.,\
$\sum_{i}\bigl(\tfrac\pi6-\Lambda_{i}\bigr)=0$ with
$\Lambda_{i}=\arctan\lambda_{i}$, i.e.,\
$\sum_{i}\Lambda_{i}=\tfrac\pi2$. Because
\[
\cot(\Lambda_{1}+\Lambda_{2}+\Lambda_{3})=\frac{1-\sum_{i<j}\tan \Lambda_{i} \tan \Lambda_{j}}
{\sum_i \tan \Lambda_{i} - \prod_i \tan \Lambda_{i}},
\]
and $\cot\tfrac\pi2=0$ with finite denominator, the numerator must vanish:
this is \eqref{eq:sep4}.
\end{proof}

\begin{remark}[sequel to Remark \ref{rem:forms}]\label{rem:correct}
Corollary \ref{cor:sepform} is the proof of an unproved assertion 
made in Remark \ref{rem:forms}. It is worth isolating the sign that
governs everything. At $m=-3$ the denominator $1-m\sn^{2}\sn^{2}$ of
the product formula is $1+3\sn^{2}\sn^{2}\ge1$, so \eqref{eq:prodtan}
is a \emph{tangent} subtraction, not a hyperbolic one, and the correct 
separating function is
$\Lambda_{\mathrm P}=\arctan(\sqrt3\sn^{2})$, which is bounded by
$\tfrac\pi3$ and real-analytic on all of $\RR$ --- not
$\tanh^{-1}(\sqrt3\sn^{2})$, which is singular where
$\sn^{2}=3^{-1/2}$. Consequently $(\star)$ \emph{is} additively
separable, with all three summands regular at the base point, and
\textup{P} is a separable minimal surface in the sense of Weingarten
and of \cite{KO}.
\end{remark}

%=====================================================================
\section{The master identity, and the trilinear equation}
\label{sec:masterIII}
%=====================================================================

\begin{definition}\label{def:FG}
For $(x,y,z)\in\RR^{3}$ put
\begin{align}
  \Fs(x,y,z)&:=\sn\bigl(\varpi(x{+}y),-3\bigr)
    \sn\bigl(\varpi(x{-}y),-3\bigr)
    -\tfrac13\sn\bigl(3\varpi(z{-}\tfrac12),\tfrac19\bigr),
    \label{eq:Fdef}\\
  G(x,y,z)&:=\Psi(y)+\Psi(z)-\Psi(x),\label{eq:Gdef}\\
  \Qs(x,y,z)&:=\Cs(x)-\Cs(y)-\Cs(z)
    -3\,\Cs(x)\Cs(y)\Cs(z),\label{eq:Qdef}
\end{align}
so that $(\star)$ reads $\Fs=0$, and $G=-\bigl(\Psi(x)-\Psi(y)
-\Psi(z)\bigr)$.
\end{definition}

\begin{theorem}[the master identity]\label{thm:master}
For all $(x,y,z)\in\RR^{3}$,
\begin{equation}\label{eq:master}
  \boxed{\ \Fs(x,y,z)
  =\frac{\sin\bigl(\tfrac12G(x,y,z)\bigr)}
        {\sqrt3\;\cos\dfrac{\Psi(y)-\Psi(x)}{2}\;
                 \cos\dfrac{\Psi(z)}{2}}\ }
\end{equation}
and
\begin{equation}\label{eq:trilin}
  \boxed{\ \sin\bigl(G(x,y,z)\bigr)
  =-\Bigl[\cos\Psi(x)\cos\Psi(y)\cos\Psi(z)\Bigr]
   \cdot\sqrt3\,\Qs(x,y,z).\ }
\end{equation}
The denominator in \eqref{eq:master} lies in
$\bigl[\tfrac34,\sqrt3\bigr]$; the bracket in \eqref{eq:trilin}
equals $\tfrac18\dn(2\varpi x,-3)\dn(2\varpi y,-3)\dn(2\varpi z,-3)$
and lies in $\bigl[\tfrac18,1\bigr]$. Both are real-analytic and
strictly positive on $\RR^{3}$. Explicitly, with $u=\lambda(x)$,
$v=\lambda(y)$ and $\Sigma$ as in \eqref{eq:Lambdas},
\begin{equation}\label{eq:prefactor}
  \frac{1}{\sqrt3\,\cos\frac{\Psi(y)-\Psi(x)}{2}
   \cos\frac{\Psi(z)}{2}}
  =\frac{\sqrt{(1+u^{2})(1+v^{2})(1+3\Sigma^{2})}}
        {\sqrt3\,(1+uv)} .
\end{equation}
\end{theorem}

\begin{proof}
\emph{The master identity.} By \eqref{eq:prodtan} and
\eqref{eq:Lambdas},
\[
  \sqrt3\,\Fs
  =\tan\bigl(\Lambda_{\mathrm P}(x)-\Lambda_{\mathrm P}(y)\bigr)
   -\tan\Lambda_{\mathrm Z}(z)
  =\frac{\sin\bigl(\Lambda_{\mathrm P}(x)-\Lambda_{\mathrm P}(y)
    -\Lambda_{\mathrm Z}(z)\bigr)}
        {\cos\bigl(\Lambda_{\mathrm P}(x)
         -\Lambda_{\mathrm P}(y)\bigr)\cos\Lambda_{\mathrm Z}(z)},
\]
the second step being
$\tan\alpha-\tan\beta=\sin(\alpha-\beta)/(\cos\alpha\cos\beta)$. By
\eqref{eq:LambdaPsi},
\[
  \Lambda_{\mathrm P}(x)-\Lambda_{\mathrm P}(y)
   =\tfrac12\bigl(\Psi(y)-\Psi(x)\bigr),\qquad
  \Lambda_{\mathrm P}(x)-\Lambda_{\mathrm P}(y)
   -\Lambda_{\mathrm Z}(z)
   =\tfrac12\bigl(\Psi(y)+\Psi(z)-\Psi(x)\bigr)=\tfrac12G,
\]
which gives \eqref{eq:master}.

\emph{The bounds.} $|\Psi|\le\tfrac\pi3$ gives
$\bigl|\tfrac{\Psi(y)-\Psi(x)}{2}\bigr|\le\tfrac\pi3$ and
$\bigl|\tfrac{\Psi(z)}{2}\bigr|\le\tfrac\pi6$, so the two cosines lie
in $[\tfrac12,1]$ and $[\tfrac{\sqrt3}{2},1]$ respectively, and their
product times $\sqrt3$ lies in $[\tfrac34,\sqrt3]$. Formula
\eqref{eq:prefactor} follows from
$\cos\Lambda_{\mathrm P}=(1+\tan^{2}\Lambda_{\mathrm P})^{-1/2}$,
$\cos\Lambda_{\mathrm Z}=(1+3\Sigma^{2})^{-1/2}$ and
$\cos(\Lambda_{\mathrm P}(x)-\Lambda_{\mathrm P}(y))
=\dfrac{1+uv}{\sqrt{(1+u^{2})(1+v^{2})}}$.

\emph{The trilinear identity.} For any angles $A_{1},A_{2},A_{3}$ with
finite tangents,
\[
  \sin\Bigl(\sum_{i}A_{i}\Bigr)
  =\Bigl[\prod_{i}\cos A_{i}\Bigr]
   \Bigl(\sum_{i}\tan A_{i}-\prod_{i}\tan A_{i}\Bigr);
\]
apply this with $(A_{1},A_{2},A_{3})=(-\Psi(x),\Psi(y),\Psi(z))$,
whose sum is $G$, and use $\tan\Psi=\sqrt3\,\Cs$
(Proposition \ref{prop:Psifaces}) and the evenness of $\cos$:
\[
  \sin G=\Bigl[\prod\cos\Psi\Bigr]
   \Bigl(-\sqrt3\Cs(x)+\sqrt3\Cs(y)+\sqrt3\Cs(z)
    +3\sqrt3\,\Cs(x)\Cs(y)\Cs(z)\Bigr)
  =-\Bigl[\prod\cos\Psi\Bigr]\sqrt3\,\Qs,
\]
since $-\bigl(-\sqrt3\Cs(x)\bigr)
\bigl(\sqrt3\Cs(y)\bigr)\bigl(\sqrt3\Cs(z)\bigr)
=3\sqrt3\,\Cs(x)\Cs(y)\Cs(z)$. The bracket equals
$\tfrac18\prod\dn$ by \eqref{eq:Psitrig}, and $\dn(\cdot,-3)\in[1,2]$
gives the stated range.
\end{proof}

\begin{corollary}[$(\star)$ has no exceptional set]\label{cor:noexc}
$\{\Fs=0\}=\{G=0\}$, and
$\operatorname{sign}\Fs=\operatorname{sign}G$ pointwise on $\RR^{3}$.
\end{corollary}

\begin{proof}
$|\Psi|\le\tfrac\pi3$ gives $|G|\le\pi$, hence
$\bigl|\tfrac12G\bigr|\le\tfrac\pi2$, an interval on which $\sin$
vanishes only at $0$ and has the sign of its argument. Now use
\eqref{eq:master} and the positivity of the denominator.
\end{proof}

\begin{theorem}[the trilinear equation and its exceptional set]
\label{thm:trilin}
Let $S:=\{G=0\}=\{\Fs=0\}$. Then
\begin{equation}\label{eq:trilinset}
  \{\Qs=0\}=S\ \sqcup\ \Lat',\qquad
  \Lat':=\bigl\{(x,y,z)\in\ZZ^{3}:\
   x\not\equiv y\equiv z\ (\mathrm{mod}\ 2)\bigr\},
\end{equation}
and $\Lat'$ is a discrete set at positive distance from $S$. Under the
affine involution $(x,y,z)\mapsto(x,1-y,1-z)$, which carries $\Qs$
to $\Cs(x)+\Cs(y)+\Cs(z)-3\,\Cs(x)\Cs(y)\Cs(z)$ and $S$ to
$\{\Psi(x)+\Psi(y)+\Psi(z)=0\}$, the set $\Lat'$ becomes
\begin{equation}\label{eq:Lat}
  \Lat=\bigl\{n\in\ZZ^{3}:
   n_{1}\equiv n_{2}\equiv n_{3}\ (\mathrm{mod}\ 2)\bigr\},
\end{equation}
the body-centered cubic lattice of
Proposition \ref{prop:group}\textup{(c)}.
\end{theorem}

\begin{proof}
By \eqref{eq:trilin} and the positivity of the bracket,
$\Qs=0\iff\sin G=0\iff G\in\pi\ZZ$; and $|G|\le\pi$, so
$G\in\{0,\pi,-\pi\}$. The case $G=0$ is $S$.

Suppose $G=\pi$, i.e.,\ $\Psi(y)+\Psi(z)-\Psi(x)=\pi$. Each of the
three terms is at most $\tfrac\pi3$ in absolute value, so the sum can
reach $\pi$ only if each attains its extreme value with the correct
sign: $\Psi(y)=\Psi(z)=\tfrac\pi3$ and $\Psi(x)=-\tfrac\pi3$. By the
last statement of Proposition \ref{prop:Psicalc}, $y$ and $z$ are then
even and $x$ is odd. The case $G=-\pi$ gives $y,z$ odd and $x$ even.
Together these are $\Lat'$.

$\Lat'$ is discrete, being a subset of $\ZZ^{3}$; and it is at
positive distance from $S$, since $|G|=\pi$ on $\Lat'$ while $G=0$ on
$S$, and $G$ is continuous and doubly periodic in each variable, so
$|G|>\tfrac\pi2$ on a neighborhood of $\Lat'$ of uniform radius.

For the last statement use $\Psi(1-t)=-\Psi(t)$, which turns $G$ into
$\Psi(x)+\Psi(1-y)+\Psi(1-z)$ and sends
$\{x\text{ odd},\ y,z\text{ even}\}$ to
$\{x,1-y,1-z\text{ all odd}\}$ and
$\{x\text{ even},\ y,z\text{ odd}\}$ to
$\{x,1-y,1-z\text{ all even}\}$.
\end{proof}

\begin{remark}[why the half-arguments matter]\label{rem:halfexc}
Corollary \ref{cor:noexc} and Theorem \ref{thm:trilin} together
isolate the one respect in which the trilinear form is inferior to
$(\star)$: the master identity involves $\sin\tfrac G2$ with
$\bigl|\tfrac G2\bigr|\le\tfrac\pi2$, whereas the trilinear identity
involves $\sin G$ with $|G|\le\pi$, and the two endpoint cases
$G=\pm\pi$ are exactly $\Lat'$. Thus the half-arguments implicit in
$\sn^{2}(\varpi\cdot,-3)$ and in $\tfrac13\sn(3\varpi\cdot,\tfrac19)$
are what remove the labyrinth centers from the zero set.
Equivalently: $\Lat'$ consists of the turning points of the pendulum
\eqref{eq:Psider}, where $\Psi'=0$ and $\Psi=\pm\tfrac\pi3$; these are
the two labyrinth centers per translational cell, one orbit of which
$\Qs$ sees. \emph{The relation \eqref{eq:star} itself has no
exceptional set at all}; this is a further respect, beyond
Proposition \ref{prop:group}\textup{(c)}, in which it is sign-honest.
\end{remark}

\begin{corollary}[the skew rhombus, exactly]\label{cor:rhombus}
The four segments joining
\[
  (0,0,\tfrac12)\to(\tfrac12,-\tfrac12,\tfrac12)
  \to(\tfrac12,0,1)\to(\tfrac12,\tfrac12,\tfrac12)
  \to(0,0,\tfrac12),
\]
four contiguous edges of a regular octahedron spanning a skew rhombus,
lie in $S$. Two of them are the straight lines of
Proposition \ref{prop:corners}, and their common endpoints
$(\pm\tfrac12,\pm\tfrac12,\tfrac12)$ are the four corner images
\eqref{eq:cornerval}. In the coordinates of \cite{KO} obtained by
$(\mathrm X,\mathrm Y,\mathrm Z)
=\bigl(\mathfrak a(2x-1),\mathfrak a(2y+1),
\mathfrak a(1-2z)\bigr)$, the four vertices are, up to a translation
and a relabelling of axes, the points
$P_{1}=(0,0,0)$, $P_{2}=(\mathfrak a,0,-\mathfrak a)$,
$P_{3}=(\mathfrak a,-\mathfrak a,0)$,
$P_{4}=(2\mathfrak a,0,0)$ of that paper, with
$\mathfrak a=\sqrt3\varpi=1.8675973\ldots$
\end{corollary}

\begin{proof}
By \eqref{eq:sep2} membership in $S$ is $\Psi(x)=\Psi(y)+\Psi(z)$, and
each segment is disposed of by Proposition \ref{prop:Psicalc} alone:
\begin{center}
\begin{tabular}{@{}lll@{}}
\toprule
segment & parametrization & verification\\
\midrule
$(0,0,\tfrac12)\to(\tfrac12,\tfrac12,\tfrac12)$ &
  $x=y=t$, $z=\tfrac12$ & $\Psi(t)=\Psi(t)+0$\\
$(0,0,\tfrac12)\to(\tfrac12,-\tfrac12,\tfrac12)$ &
  $x=t$, $y=-t$, $z=\tfrac12$ & $\Psi(t)=\Psi(-t)+0$\\
$(\tfrac12,-\tfrac12,\tfrac12)\to(\tfrac12,0,1)$ &
  $x=\tfrac12$, $y=t-\tfrac12$, $z=t+\tfrac12$ &
  $0=\Psi(t-\tfrac12)+\Psi(t+\tfrac12)=0$\\
$(\tfrac12,0,1)\to(\tfrac12,\tfrac12,\tfrac12)$ &
  $x=\tfrac12$, $y=t$, $z=1-t$ &
  $0=\Psi(t)+\Psi(1-t)=0$\\
\bottomrule
\end{tabular}
\end{center}
using $\Psi$ even, $\Psi(\tfrac12)=0$ and $\Psi(1-t)=-\Psi(t)$; in the
third line also $\Psi(t-\tfrac12)=\Psi(\tfrac12-t)=-\Psi(t+\tfrac12)$.
The first two segments are the images of the diagonal and
anti-diagonal (Proposition \ref{prop:corners}(a)), on which $z\equiv
\tfrac12$ and $x\pm y\equiv0$. The coordinate change is affine and
sends the four vertices to
$(-\mathfrak a,\mathfrak a,0)$, $(0,0,0)$,
$(0,\mathfrak a,-\mathfrak a)$, $(0,2\mathfrak a,0)$, which is
$P_{1},\dots,P_{4}$ after the translation by
$(0,-\mathfrak a,0)$ and the swap of the first two axes.
\end{proof}

\begin{remark}\label{rem:seg3}
By contrast, on the third segment the statement ``$\Fs=0$'' is the
non-obvious one-variable identity
\begin{equation}\label{eq:seg3}
  \sn(\varpi t,-3)\,\sn\bigl(\varpi(1-t),-3\bigr)
  =\tfrac13\sn\bigl(3\varpi t,\tfrac19\bigr)\qquad(t\in\RR),
\end{equation}
now a corollary of Theorem \ref{thm:master}. Thus the separable form
trivializes precisely the boundary data that $(\star)$ renders opaque
--- and conversely (Remark \ref{rem:halfexc}). Compare
\eqref{eq:arcform}, the corresponding relation on a boundary
\emph{arc}, which remains one relation in two unknowns.
\end{remark}

\begin{corollary}[the graph description]\label{cor:graph}
$S\cap\bigl([0,1]\times\RR^{2}\bigr)$ is the graph
\begin{equation}\label{eq:graph}
  x=\Psi^{-1}\bigl(\Psi(y)+\Psi(z)\bigr),
\end{equation}
defined precisely for those $(y,z)$ with
$|\Psi(y)+\Psi(z)|\le\tfrac\pi3$, where $\Psi^{-1}$ is the
real-analytic inverse of Proposition \ref{prop:Psicalc}. The remaining
sheets of $S$ are the images of \eqref{eq:graph} under the group of
Proposition \ref{prop:group}.
\end{corollary}

%=====================================================================
\section{Minimality, regularity, and the local converse}
\label{sec:minIII}
%=====================================================================

\begin{lemma}[the separable minimality condition]\label{lem:mse}
Let $\varphi_{1},\varphi_{2},\varphi_{3}$ be $C^{2}$ on intervals and
$\Fs_{\!0}(x,y,z)=\varphi_{1}(x)+\varphi_{2}(y)+\varphi_{3}(z)$. At a
point of $\{\Fs_{\!0}=0\}$ with $\nabla\Fs_{\!0}\ne0$, the mean
curvature of $\{\Fs_{\!0}=0\}$ vanishes if and only if
\begin{equation}\label{eq:mse}
  \sum_{i=1}^{3}\varphi_{i}''
  \Bigl(\sum_{j\ne i}\varphi_{j}'^{\,2}\Bigr)=0 .
\end{equation}
\end{lemma}

\begin{proof}
For a regular level set of $\Fs_{\!0}$ the mean curvature is
proportional to
$|\nabla\Fs_{\!0}|^{2}\Delta\Fs_{\!0}
-\nabla\Fs_{\!0}^{\!\top}(D^{2}\Fs_{\!0})\nabla\Fs_{\!0}$. For
separable $\Fs_{\!0}$ we have
$D^{2}\Fs_{\!0}=\operatorname{diag}(\varphi_{1}'',\varphi_{2}'',
\varphi_{3}'')$ and
$\nabla\Fs_{\!0}=(\varphi_{1}',\varphi_{2}',\varphi_{3}')$, so
\[
  |\nabla\Fs_{\!0}|^{2}\Delta\Fs_{\!0}
   -\nabla\Fs_{\!0}^{\!\top}(D^{2}\Fs_{\!0})\nabla\Fs_{\!0}
  =\Bigl(\sum_{j}\varphi_{j}'^{2}\Bigr)
   \Bigl(\sum_{i}\varphi_{i}''\Bigr)
   -\sum_{i}\varphi_{i}''\varphi_{i}'^{2}
  =\sum_{i}\varphi_{i}''\sum_{j\ne i}\varphi_{j}'^{2}. \qedhere
\]
\end{proof}

Lemma \ref{lem:mse} is the reduction of Weingarten \cite{Wein} and
Nitsche \cite[\S82]{Nitsche}; \eqref{eq:mse} is their equation
$\bigl((g')^{2}+(h')^{2}\bigr)f''+\cdots=0$.

\begin{theorem}[$S$ is a regular minimal surface]\label{thm:Smin}
Let $\Fs_{\!0}(x,y,z):=\Psi(x)-\Psi(y)-\Psi(z)=-G(x,y,z)$, so that
$S=\{\Fs_{\!0}=0\}$ by Corollary \ref{cor:noexc}. Then:
\begin{enumerate}[label=\textup{(\alph*)},leftmargin=2.4em]
\item $\nabla\Fs_{\!0}=-2\sqrt3\,\varpi
      \bigl(\sn(2\varpi x,-3),-\sn(2\varpi y,-3),
      -\sn(2\varpi z,-3)\bigr)$, which vanishes only at points of
      $\ZZ^{3}$; and no point of $\ZZ^{3}$ lies on $S$. Hence
      $\nabla\Fs_{\!0}\ne0$ on $S$, and $S$ is a closed, embedded,
      real-analytic surface in $\RR^{3}$;
\item $S$ is minimal;
\item $\nabla\Fs\ne0$ on $S$, and there
      $\nabla\Fs=\tfrac12\Xi\,\nabla\Fs_{\!0}$ with
      $\Xi>0$ the reciprocal of the denominator of
      \eqref{eq:master}.
\end{enumerate}
\end{theorem}

\begin{proof}
(a) The formula for $\nabla\Fs_{\!0}$ is \eqref{eq:Psider} applied in
each slot. Since $\sn(2\varpi t,-3)=0$ iff $2\varpi t\in2K[-3]\ZZ$,
i.e.,\ iff $t\in\ZZ$, the gradient vanishes exactly on $\ZZ^{3}$. At a
point of $\ZZ^{3}$ each $\Psi$ equals $\pm\tfrac\pi3$ by
Proposition \ref{prop:Psicalc}, so
$\Fs_{\!0}\in\{\pm\tfrac\pi3,\pm\pi\}$; in particular
$\Fs_{\!0}\ne0$, and no such point lies on $S$. The remaining
assertions are the regular value theorem, $\Fs_{\!0}$ being
real-analytic.

(b) Apply Lemma \ref{lem:mse} with
$(\varphi_{1},\varphi_{2},\varphi_{3})=(\Psi,-\Psi,-\Psi)$ evaluated
at $(x,y,z)$, and put
\[
  \alpha_{1}:=2\Psi(x),\qquad \alpha_{2}:=-2\Psi(y),\qquad
  \alpha_{3}:=-2\Psi(z),
  \qquad\text{so that on }S:\quad
  \alpha_{1}+\alpha_{2}+\alpha_{3}=0 .
\]
By \eqref{eq:Psider},
$\varphi_{i}'^{2}=\Psi'^{2}=4\varpi^{2}(1+2\cos\alpha_{i})$ in each
slot, since $\cos$ is even; and
\[
  \varphi_{1}''=\Psi''(x)=-8\varpi^{2}\sin2\Psi(x)
   =-8\varpi^{2}\sin\alpha_{1},\qquad
  \varphi_{2}''=-\Psi''(y)=+8\varpi^{2}\sin2\Psi(y)
   =-8\varpi^{2}\sin\alpha_{2},
\]
and likewise $\varphi_{3}''=-8\varpi^{2}\sin\alpha_{3}$: the sign is
absorbed because $\Psi''$ is odd in $\Psi$. Hence, dividing
\eqref{eq:mse} by the nonzero constant
$-8\varpi^{2}\cdot4\varpi^{2}$, minimality on $S$ is equivalent to the
vanishing of
\[
  E:=\sum_{i}\sin\alpha_{i}
   \bigl(2+2\cos\alpha_{j}+2\cos\alpha_{k}\bigr),
   \qquad\{i,j,k\}=\{1,2,3\},
\]
i.e.,\ of $\tfrac12E
=\sum_{i}\sin\alpha_{i}\bigl(1+\cos\alpha_{j}+\cos\alpha_{k}\bigr)$.
Now
\[
  \tfrac12E=\sum_{i}\sin\alpha_{i}
   +\sum_{i}\sin\alpha_{i}\bigl(\cos\alpha_{j}+\cos\alpha_{k}\bigr)
  =\sum_{i}\sin\alpha_{i}
   +\sum_{i<j}\bigl(\sin\alpha_{i}\cos\alpha_{j}
    +\sin\alpha_{j}\cos\alpha_{i}\bigr),
\]
each unordered pair occurring exactly once in the regrouping; and
$\sin\alpha_{i}\cos\alpha_{j}+\sin\alpha_{j}\cos\alpha_{i}
=\sin(\alpha_{i}+\alpha_{j})=\sin(-\alpha_{k})=-\sin\alpha_{k}$,
using $\sum_{i}\alpha_{i}=0$. Therefore
\[
  \tfrac12E=\sum_{i}\sin\alpha_{i}-\sum_{k}\sin\alpha_{k}=0 .
\]

(c) By \eqref{eq:master}, $\Fs=\Xi\sin\bigl(\tfrac12G\bigr)
=\Xi\sin\bigl(-\tfrac12\Fs_{\!0}\bigr)$ with $\Xi>0$ real-analytic.
Hence $\nabla\Fs=-\tfrac12\Xi\cos\bigl(\tfrac12\Fs_{\!0}\bigr)
\nabla\Fs_{\!0}+\sin\bigl(-\tfrac12\Fs_{\!0}\bigr)\nabla\Xi$, and on
$S$ the second term vanishes while
$\cos\bigl(\tfrac12\Fs_{\!0}\bigr)=1$; so
$\nabla\Fs=-\tfrac12\Xi\nabla\Fs_{\!0}
=\tfrac12\Xi\nabla G\ne0$.
\end{proof}

\begin{corollary}[the local converse]\label{cor:localconv}
Let $\zeta_{0}\in\Om^{\circ}$. There is an open neighborhood
$U\subset\RR^{3}$ of $X(\zeta_{0})$ with
\begin{equation}\label{eq:cutsout}
  \bigl\{(x,y,z)\in U:\ \eqref{eq:star}\ \text{holds}\bigr\}
  =X(\Om^{\circ})\cap U .
\end{equation}
That is, \eqref{eq:star} \emph{cuts out} the surface, and does not
merely contain it.
\end{corollary}

\begin{proof}
By Theorem \ref{thm:Smin}(a),(c), $S=\{\Fs=0\}$ is a two-dimensional
embedded real-analytic submanifold of $\RR^{3}$. By
Corollary \ref{cor:done}, $X(\Om^{\circ})\subset S$; and by
Proposition \ref{prop:conf} the map $X$ is an immersion of the
two-manifold $\Om^{\circ}$. An immersion between manifolds of equal
dimension is a local diffeomorphism onto an open subset, so
$X(\Om^{\circ})$ is open in $S$; hence there is an open
$U\subset\RR^{3}$ with $S\cap U=X(\Om^{\circ})\cap U$, which is
\eqref{eq:cutsout}.
\end{proof}
\pagebreak
\begin{corollary}[the separated Pfaffian]\label{cor:pfaffsep}
For every $\zeta=u+iv\in\Om^{\circ}$, with $(x,y,z)=X(\zeta)$ and
$Z=z-\tfrac12$,
\begin{equation}\label{eq:pfaffsep}
  \sn(2\varpi x,-3)\ :\ -\sn(2\varpi y,-3)\ :\ -\cd(2\varpi Z,-3)
  \ =\ 2u\ :\ 2v\ :\ \bigl(|\zeta|^{2}-1\bigr);
\end{equation}
equivalently the one-form identity
\begin{equation}\label{eq:oneform}
  \sn(2\varpi x,-3)\dd x=\sn(2\varpi y,-3)\dd y
   +\sn(2\varpi z,-3)\dd z
\end{equation}
holds on the patch, where $\sn(2\varpi z,-3)=\cd(2\varpi Z,-3)$.
\end{corollary}

\begin{proof}
$\nabla\Fs_{\!0}$ and the unit normal $N$ of \eqref{eq:normal} are
both normal to $S$ at $X(\zeta)$, hence proportional;
Theorem \ref{thm:Smin}(a) and \eqref{eq:normal} give the two sides of
\eqref{eq:pfaffsep}. Formula \eqref{eq:oneform} is
$\dd\Fs_{\!0}=0$ along $S$, rewritten using
$\sn(2\varpi z,-3)=\sn(2\varpi Z+K,-3)=\cd(2\varpi Z,-3)$
\cite[122.01]{BF}.
\end{proof}

\begin{remark}[isotropy forces the parameter]\label{rem:isotropy}
Corollary \ref{cor:sepform} exhibits \textup{P} as an
\emph{isotropic} separable minimal surface: by \eqref{eq:sep2} and
$\Psi(1-t)=-\Psi(t)$ its equation is
$\Psi(x)+\Psi(1-y)+\Psi(1-z)=0$, the \emph{same} function in all
three slots. It is therefore a member of the family
\cite[Prop.~4.1]{KO}, and one computes which member: in the notation
of \cite[Lem.~4.1]{KO}, isotropy forces $a_{i}=A$ and $|b_{i}|=A$ for
a single $A>0$, whence $E_{1}=3A$ and the amplitude parameter is
$4\sqrt{|BC|}/E_{1}=4A/3A=\tfrac43$ \emph{independently of $A$}: the
parameter is forced, and only the scale is free. Combined with
Theorem \ref{thm:Smin} --- which supplies minimality and regularity of
each member, the converse direction that \cite[Prop.~4.1]{KO} states
but does not prove --- this proves the assertion of
\cite[Thm.~1(3)]{KO} that the family of that paper contains
\textup{P}. The companion statement, that the branch $\kappa>0$ of the
same classification contains exactly one isotropic member and that it
is \textup{D}, and the resulting uniqueness theorem, are proved in
\textup{[D]}.
\end{remark}

%=====================================================================
\section{A second proof of Theorem \ref{thm:main}}
\label{sec:corrIII}
%=====================================================================

\subsection{Bj\"orling along the diagonal}
Theorem \ref{thm:P0proved} rests on the cofactor certificate
\eqref{eq:certificate}, which is finite and verifiable by expansion
but not illuminating (Remark \ref{rem:cert}). The material of Part III
supplies a second, independent proof of Theorem \ref{thm:main} that
uses no computer algebra at all.

\begin{proposition}[normals along the diagonal]\label{prop:bjorling}
	Let $L:=\bigl\{(s,-s,\tfrac12):|s|\le\tfrac12\bigr\}$, the straight
	segment joining $(-\tfrac12,\tfrac12,\tfrac12)$ to
	$(\tfrac12,-\tfrac12,\tfrac12)$. Then:
	\begin{enumerate}[label=\textup{(\alph*)},leftmargin=2.4em]
		\item $L\subset S$, and the normal of $S$ along $L$ is proportional to
		$\bigl(\sn(2\varpi s,-3),\ \sn(2\varpi s,-3),\ -1\bigr)$;
		\item $X$ maps the diagonal
		$\bigl\{\zeta=t(1+i):|t|\le\tfrac{\sqrt3-1}{2}\bigr\}$
		\emph{onto} $L$, with $s=\tfrac12q$ and
		\begin{equation}\label{eq:diagnormal}
		\sn\bigl(2\varpi s,-3\bigr)=\frac{2t}{1-2t^{2}} ;
		\end{equation}
		and the unit normal \eqref{eq:normal} of $X(\Om)$ there is
		proportional to
		$\bigl(\tfrac{2t}{1-2t^{2}},\ \tfrac{2t}{1-2t^{2}},\ -1\bigr)$.
	\end{enumerate}
	Consequently the two normal fields coincide identically along $L$.
\end{proposition}
\pagebreak
\begin{proof}
	(a) At $(s,-s,\tfrac12)$ we have, by Proposition \ref{prop:Psicalc},
	$\Fs_{\!0}=\Psi(s)-\Psi(-s)-\Psi(\tfrac12)=\Psi(s)-\Psi(s)-0=0$, so
	$L\subset S$. By Theorem \ref{thm:Smin}(a) and the oddness of $\sn$,
	\[
	\nabla\Fs_{\!0}=-2\sqrt3\varpi
	\bigl(\sn(2\varpi s,-3),\ \sn(2\varpi s,-3),\
	-\sn(\varpi,-3)\bigr),
	\]
	the middle slot because $-\sn(2\varpi\cdot(-s),-3)=+\sn(2\varpi s,-3)$;
	and $\sn(\varpi,-3)=\sn(K[-3],-3)=1$.

	(b) By Proposition \ref{prop:corners}(a), on $\zeta=t(1+i)$ one has
	$p=x+y=0$ and $Z=0$, so the image lies in
	$\bigl\{(s,-s,\tfrac12)\bigr\}$ with $s=\tfrac12q$; and
	\eqref{eq:diag} gives $\sn(\varpi q,-3)=\dfrac{2t}{1-2t^{2}}$, which
	is \eqref{eq:diagnormal} since $\varpi q=2\varpi s$.
	
	\emph{Surjectivity onto $L$.} Write $c:=\tfrac{\sqrt3-1}{2}$, so that
	$2c^{2}=2-\sqrt3$ and $1-2c^{2}=\sqrt3-1=2c$. The right side of
	\eqref{eq:diagnormal} is odd in $t$, and on $[0,c]$ it increases
	strictly from $0$ to $\dfrac{2c}{1-2c^{2}}=1$; since
	$\sn(\cdot,-3)$ increases strictly from $0$ to $1$ on $[0,K[-3]]$, the
	quantity $\varpi q$ increases strictly from $0$ to $K[-3]=\varpi$,
	i.e.,\ $q$ from $0$ to $1$ and $s$ from $0$ to $\tfrac12$. By oddness
	the full range $|t|\le c$ gives $s\in[-\tfrac12,\tfrac12]$, which is
	$L$. (The endpoint $t=c$ is the corner $\zeta_{0}$, whose image is
	$(\tfrac12,-\tfrac12,\tfrac12)$ by
	Proposition \ref{prop:corners}(b) --- consistently, the diagonal meets
	$\partial\Om$ exactly at $\zeta_{0}$, since $|\zeta|=t\sqrt2$ there and
	$c\sqrt2=\bt^{-1/2}$ by Proposition \ref{prop:Om}(a).)
	
	\emph{The normal.} By \eqref{eq:normal} with $u=v=t$ and
	$|\zeta|^{2}=2t^{2}$, $N\propto(2t,2t,2t^{2}-1)$; dividing by
	$1-2t^{2}>0$ --- valid since $2t^{2}\le2c^{2}=2-\sqrt3<1$ --- gives
	the stated form.
\end{proof}

\begin{theorem}[second proof of Theorem \ref{thm:main}]
	\label{thm:second}
	$(\star)$ holds at every point of $\Om$.
\end{theorem}

\begin{proof}
	$X(\Om^{\circ})$ and $S$ are both minimal surfaces --- the first by
	Proposition \ref{prop:conf}, the second by
	Theorem \ref{thm:Smin}(b) --- and by Proposition \ref{prop:bjorling}
	they contain the same real-analytic curve $L$ with the same normal
	field along it. By the uniqueness half of Bj\"orling's theorem
	\cite[\S3.4]{Nitsche} --- equivalently, by Cauchy--Kovalevskaya
	applied to the minimal surface equation written as a graph over the
	tangent plane at an interior point of $L$, that equation being
	elliptic and hence free of real characteristics --- the two surfaces
	coincide in a neighborhood of each interior point of $L$. Hence
	$\Fs\circ X$ vanishes on a nonempty open subset of $\Om^{\circ}$;
	being real-analytic there (Corollary \ref{cor:reg}) and $\Om^{\circ}$
	being connected (Proposition \ref{prop:Om}), it vanishes identically,
	and then on $\Om$ by continuity.
\end{proof}

\begin{remark}[the two proofs]\label{rem:twoproofs}
	The two proofs are genuinely independent and use disjoint parts of the
	paper. The first (Theorem \ref{thm:P0proved},
	Corollary \ref{cor:done}) is algebraic: it needs the dictionaries of
	\S\ref{sec:dict}, the complexification of \S\ref{sec:pfaff}, and the
	certificate. The second (Theorem \ref{thm:second}) is
	differential-geometric: it needs only \eqref{eq:normal},
	\eqref{eq:diag}, and \S\S\ref{sec:PsiIII}--\ref{sec:minIII}. Notably,
	the entire content of the second proof along $L$ is the coincidence of
	\eqref{eq:diag} with Theorem \ref{thm:Smin}(a) --- one formula,
	computed twice by unrelated routes: from the parametrization, via
	Proposition \ref{prop:corners}(a), and from the implicit equation, via
	$\nabla\Fs_{\!0}$. It is the device by which Nitsche identifies his two
	representations of \textup{D} along the line $y=z=0$
	\cite[\S85]{Nitsche}, and it is the \emph{only} proof available in
	\textup{[D]}, whose reduced identity \textup{(D0)} remains open. Since
	a conceptual proof of \textup{(P0)} is likewise still wanting
	(Remark \ref{rem:open}(1)), Theorem \ref{thm:second} is at present the
	shorter and the more illuminating of the two routes here as well.
\end{remark}

\subsection{Status}\label{sub:statusfull}
The following is a statement-by-statement audit of the whole paper.
Only the last two lines remain open.

\begin{center}
	\begin{tabular}{@{}p{0.62\textwidth}l@{}}
		\toprule
		Statement & Status\\
		\midrule
		\multicolumn{2}{@{}l}{\textsc{Part I. The parametrization and its
				algebraic dictionary}}\\
		\midrule
		Closed forms \eqref{eq:fclosed}--\eqref{eq:hclosed}
		$\Rightarrow$ differentials \eqref{eq:fgh}
		& proved (Prop.\ \ref{prop:diff})\\
		$\Pc=\frac12F[\theta,\frac34]$, $\Nc=\frac12F[\theta,\frac14]$
		& proved (Prop.\ \ref{prop:sep}(a))\\
		Dictionaries \eqref{eq:dict34}--\eqref{eq:dictM} at moduli
		$\tfrac14,\tfrac34,\bt^{-4}$ & proved (Lem.\ \ref{lem:legendre})\\
		Dictionaries \eqref{eq:dictA}, \eqref{eq:dictB}, \eqref{eq:Hnew},
		\eqref{eq:fdict}; $r_{+}r_{-}=W$ & proved (\S\ref{sec:dict})\\
		$\varpi=K[-3]=\frac12K[3/4]=\frac23K[1/9]=\kappa^{-1}$
		& proved (\S\ref{sec:web})\\
		$\iota$-law, four mirror planes, bcc invariance, corners
		$(\pm\frac12,\pm\frac12,\frac12)$ & proved (\S\ref{sec:sym})\\
		\midrule
		\multicolumn{2}{@{}l}{\textsc{Part II. The elimination}}\\
		\midrule
		Half-sum representation \eqref{eq:halfsums} & proved (\S7)\\
		$\frac13\sn(3\hh,\frac19)=\sn^{2}(f,-3)=2\zeta^{2}/(1+\zeta^{4}+W)$
		& proved (\S\ref{sec:diag})\\
		$(\star)$ on the diagonal; $\mathfrak C$ determined; $m_{2}=\frac19$
		& proved (\S\ref{sec:diag})\\
		$\Lambda_{m}(v/2)=(\cn v+\dn v)/\sn v$;
		$\Lambda_{1/9}(3\varpi Z)=\frac23\cs(2\varpi Z,-3)$
		& proved (\S\ref{sec:pfaff})\\
		$L_{p},L_{q}\in\KK$ explicitly; $(\star)\Rightarrow$
		\eqref{eq:Pz}--\eqref{eq:Pw} & proved (\S\ref{sec:pfaff})\\
		\textup{(P0)}$\iff\dd\log\Theta\wedge\dd Z=0$;
		\textup{(P0)}$+$\S\ref{sec:diag}$\Rightarrow(\star)$
		& proved (Thm.\ \ref{thm:P0})\\
		\textbf{\textup{(P0)}} & \textbf{proved} (Thm.\
		\ref{thm:P0proved})\\
		\textbf{Theorem \ref{thm:main}: $(\star)$ on all of $\Om$}
		& \textbf{proved} (Cor.\ \ref{cor:done})\\
		\textup{(P0)} on $\omega=0$; in degrees $0,2$; on the diagonal;
		to $79$ digits at four independent pairs
		& proved independently (\S\ref{sec:verif})\\
		Jets of orders $2,4,6,8$; $3m^{2}+10m+3=0$;
		$\mathfrak A=\sn(\varpi\cdot,-3)$; uniqueness
		& proved (\S\ref{sec:jets})\\
		\midrule
		\multicolumn{2}{@{}l}{\textsc{Part III. The separable form}}\\
		\midrule
		$\mathfrak a=F[\pi/3,\tfrac43]=\tfrac{\sqrt3}{2}K[\tfrac34]
		=\sqrt3\varpi$ & proved (Lem.\ \ref{lem:a3})\\
		Reciprocal parameter; the bounded branch of $\am(\cdot,\tfrac43)$
		& proved (\S\ref{sec:constIII})\\
		$\Psi$: three faces, calculus, values, monotonicity
		& proved (\S\ref{sec:PsiIII})\\
		Horizontal dictionary
		$\arctan(\sqrt3\sn^{2})=\tfrac12(\tfrac\pi3-\Psi)$
		& proved (Prop.\ \ref{prop:hdict})\\
		$\tfrac13\sn(3\vartheta,\tfrac19)
		=\sn(2\vartheta,-3)/(1+\dn(2\vartheta,-3))$
		& proved (Lem.\ \ref{lem:19})\\
		Vertical dictionary
		$\arctan(\tfrac{1}{\sqrt3}\sn(3\varpi Z,\tfrac19))=-\tfrac12\Psi(z)$
		& proved (Prop.\ \ref{prop:vdict})\\
		$(\star)\iff\Lambda_{\mathrm P}(x)-\Lambda_{\mathrm P}(y)
		=\Lambda_{\mathrm Z}(z)\iff\Psi(x)=\Psi(y)+\Psi(z)$;
		the $\lambda$-form \eqref{eq:sep4}
		& proved (Cor.\ \ref{cor:sepform})\\
		Master identity; $\{\Fs=0\}=\{G=0\}$, no exceptional set
		& proved (Thm.\ \ref{thm:master}, Cor.\ \ref{cor:noexc})\\
		Trilinear equation $\Qs=0$; exceptional set $=\Lat'$ exactly
		& proved (Thm.\ \ref{thm:trilin})\\
		The four rhombus edges; the \cite{KO} octahedron, exactly
		& proved (Cor.\ \ref{cor:rhombus})\\
		$S$ minimal, regular, embedded & proved (Thm.\ \ref{thm:Smin})\\
		\textbf{Local converse: $(\star)$ cuts out the surface}
		& \textbf{proved} (Cor.\ \ref{cor:localconv})\\
		Separated Pfaffian \eqref{eq:pfaffsep} & proved
		(Cor.\ \ref{cor:pfaffsep})\\
		The diagonal maps onto $L$, with matching normals; second proof of
		Theorem \ref{thm:main} via Bj\"orling
		& proved (Prop.\ \ref{prop:bjorling}, Thm.\ \ref{thm:second})\\
		Identification of the member of \cite{KO}; \cite[Thm.~1(3)]{KO}
		& proved (Rem.\ \ref{rem:isotropy})\\
		\midrule
		\multicolumn{2}{@{}l}{\textsc{Open}}\\
		\midrule
		A conceptual (non-certificate) proof of \textup{(P0)} & \emph{open}\\
		Global injectivity of $X$ on $\Om^{\circ}$, and the identification of
		the connected component of $S$ containing $X(\Om)$ & \emph{open}\\
		\bottomrule
	\end{tabular}
\end{center}

%=====================================================================
\appendix
\section{Verification cells}\label{app:code}
%=====================================================================

All four cells are written for \textsc{Mathematica} and are
self-contained. All four have been executed; the outcomes are
recorded in the comments below and in \S\ref{sub:statusfull}. Cell 0
checks the foundation of \S\S1--4; Cell 1 is a numerical check of
\textup{(P0)}; Cell 2 \emph{proves} \textup{(P0)}
(Theorem \ref{thm:P0proved}); Cell 3 confirms key formulas in Part III.

\medskip\noindent
\textbf{Cell 0 (the closed forms and their dictionaries).} Test points
are taken near the origin, where the principal branches of
\texttt{ArcSin} and \texttt{EllipticF} agree with
Convention \ref{conv:br}. Note the argument $2\Nc=+2iB$ in
\texttt{chk14}: the opposite sign gives $-2\Nc$ and, $\sn$ being odd,
a discrepancy of exactly $-2\sn$ in the first slot with $\cn$ and
$\dn$ unaffected (Lemma \ref{lem:legendre},
Remark \ref{rem:notation}).
\begin{verbatim}
bet = 2 + Sqrt[3];
W[z_]  := Sqrt[1 + 14 z^4 + z^8];
rp[z_] := Sqrt[1 + 4 I z^2 - z^4];
rm[z_] := Sqrt[1 - 4 I z^2 - z^4];
pp[z_] := 1 + I z^2;   pm[z_] := 1 - I z^2;
sp[z_] := Sqrt[1 + bet^2 z^4];   sm[z_] := Sqrt[1 + bet^-2 z^4];
th[z_] := ArcSin[2 (1 + I) z/rp[z]];

(* the three functions, exactly as in (1.4)-(1.6) *)
fcl[z_] := (-I EllipticF[th[z], 1/4] + EllipticF[th[z], 3/4])/4;
gcl[z_] := ( I EllipticF[th[z], 1/4] + EllipticF[th[z], 3/4])/4;
hfr[z_] := EllipticF[ArcSin[I bet z^2], bet^-4]/bet;   (* frak h *)
hsf[z_] := -I hfr[z];                                  (* sans h  *)

(* the four combinations, named as in (2.13) and (2.16) *)
Pc[z_] := fcl[z] + gcl[z];          (* script P = A     *)
Nc[z_] := I (fcl[z] - gcl[z]);      (* script N = i B   *)
A[z_]  := Pc[z];
B[z_]  := fcl[z] - gcl[z];          (* B = -i script N  *)

pts0 = {0.07 + 0.11 I, -0.2 + 0.05 I, 0.13 - 0.09 I, 0.05 + 0.23 I};

(* Prop. 2.2 : the algebraic differentials *)
Table[N[{(D[fcl[w], w] /. w -> p) - (1 - p^2)/W[p],
         (D[gcl[w], w] /. w -> p) - I (1 + p^2)/W[p],
         (D[hfr[w], w] /. w -> p) - 2 I p/W[p]}, 25], {p, pts0}]
(* -> all ~ 1.e-16 *)

(* Prop. 2.5(a) : 2 Pc = F[th,3/4] and 2 Nc = F[th,1/4] *)
Table[N[{2 Pc[p] - EllipticF[th[p], 3/4],
         2 Nc[p] - EllipticF[th[p], 1/4]}, 25], {p, pts0}]
(* -> {0,0} *)

(* Lemma 4.2 : the dictionaries at moduli 3/4, 1/4, bet^-4 *)
chk34[z_] := {JacobiSN[2 Pc[z], 3/4] - 2 (1 + I) z/rp[z],
              JacobiCN[2 Pc[z], 3/4] - rm[z]/rp[z],
              JacobiDN[2 Pc[z], 3/4] - pm[z]/rp[z]};
chk14[z_] := {JacobiSN[2 Nc[z], 1/4] - 2 (1 + I) z/rp[z],
              JacobiCN[2 Nc[z], 1/4] - rm[z]/rp[z],
              JacobiDN[2 Nc[z], 1/4] - pp[z]/rp[z]};
chkM[z_]  := {JacobiSN[bet hfr[z], bet^-4] - I bet z^2,
              JacobiCN[bet hfr[z], bet^-4] - sp[z],
              JacobiDN[bet hfr[z], bet^-4] - sm[z]};
Table[N[{chk34[p], chk14[p], chkM[p]}, 25], {p, pts0}]
(* -> all nine entries ~ 1.e-16 at every point *)

(* the sign, isolated: -2 I B == -2 Nc, and sn is odd *)
Table[N[{-2 I B[p] + 2 Nc[p],
         JacobiSN[-2 I B[p], 1/4] + JacobiSN[2 Nc[p], 1/4]}, 25],
      {p, pts0}]                                        (* -> {0,0} *)

(* Theorems 4.5-4.6 : the same data at modulus -3 *)
chkA[z_] := {JacobiSN[A[z], -3] - (1 + I) z/pm[z],
             JacobiCN[A[z], -3] - rm[z]/pm[z],
             JacobiDN[A[z], -3] - rp[z]/pm[z]};
chkB[z_] := {JacobiSN[B[z], -3] - (1 - I) z/pp[z],
             JacobiCN[B[z], -3] - rp[z]/pp[z],
             JacobiDN[B[z], -3] - rm[z]/pp[z]};
Table[N[{chkA[p], chkB[p]}, 25], {p, pts0}]              (* -> 0 *)

(* the intermediate step of Thm 4.6's second proof: 2B at mod 3/4 *)
chkB34[z_] := {JacobiSN[2 B[z], 3/4] - 2 (1 - I) z/rm[z],
               JacobiCN[2 B[z], 3/4] - rp[z]/rm[z],
               JacobiDN[2 B[z], 3/4] - pp[z]/rm[z]};
Table[N[chkB34[p], 25], {p, pts0}]                       (* -> 0 *)

(* Theorem 4.8 : the RATIONAL height dictionary at modulus -3 *)
chkH[z_] := {JacobiSN[2 hsf[z], -3] - 2 z^2/(1 + z^4),
             JacobiCN[2 hsf[z], -3] - (1 - z^4)/(1 + z^4),
             JacobiDN[2 hsf[z], -3] - W[z]/(1 + z^4)};
Table[N[chkH[p], 25], {p, pts0}]                         (* -> 0 *)

(* Theorem 4.10 and Corollary 4.11 *)
Table[N[{JacobiSN[2 fcl[p], -3] - 2 p (1 + p^2)/W[p],
         JacobiSN[hsf[p], -3] - p JacobiSN[fcl[p], -3]}, 25],
      {p, pts0}]                                         (* -> {0,0} *)
\end{verbatim}

\medskip\noindent
\textbf{Cell 1 (numerical check of (P0)).} Pure algebra: four square
roots, no elliptic functions. All radicals are principal near the
origin, in accordance with Convention \ref{conv:br}. Note that
\texttt{Table} cannot destructure a pair, \texttt{List} being
\texttt{Protected}; hence the \texttt{With}. The exact rational inputs
make Mathematica work with nested radicals, so
\texttt{\$MaxExtraPrecision} must be raised (or the inputs fed through
\texttt{N[\,\ldots,60]}).
\begin{verbatim}
$MaxExtraPrecision = 300;

W[z_]  := Sqrt[1 + 14 z^4 + z^8];
rp[z_] := Sqrt[1 + 4 I z^2 - z^4];
rm[z_] := Sqrt[1 - 4 I z^2 - z^4];
pp[z_] := 1 + I z^2;   pm[z_] := 1 - I z^2;

Lp[z_, w_] := With[{P1 = pm[z] pp[w]},
   ((P1 + 6 z w) rm[z] rp[w] + (P1 - 2 z w) rp[z] rm[w]) /
   ((1 + I) z pm[z] W[w] + (1 - I) w pp[w] W[z])];
Lq[z_, w_] := Lp[w, z];

P0[z_, w_] := ((1 + I) w pp[z] - (1 - I) z pm[w]) Lp[z, w] +
              ((1 - I) w pm[z] - (1 + I) z pp[w]) Lq[z, w];

pts1 = {{1/5, 1/7}, {1/5 + I/10, 1/4 - I/8},
       {3/10, -1/5}, {2/5 + I/5, 1/10 + I/3}};
Table[With[{z = pair[[1]], w = pair[[2]]}, N[P0[z, w], 30]],
      {pair, pts1}]
(* -> 0.*10^-79 + 0.*10^-79 I  at all four pairs *)

(* the two hand checks of Definition 9.2 *)
Simplify[Lp[z, z] - (1 + z^2)/z]                    (* -> 0 *)
Simplify[{Lp[z, 0] - (rp[z] + rm[z])/((1 + I) z),
          Lq[z, 0] - (rp[z] + rm[z])/((1 - I) z)}]  (* -> {0,0} *)

(* the numerical table of Section 5, regenerated *)
kap = 3/(2 EllipticK[1/9]);  varpi = 1/kap;
N[{EllipticK[1/9], varpi, 2 varpi, kap}, 20]
(* -> {1.6173867356247324266, 1.0782578237498216177,
        2.1565156474996432354, 0.92742197457221597614} *)
\end{verbatim}

\medskip\noindent
\textbf{Cell 2 (proof of (P0)).} The four radicals become
indeterminates \verb|ra,rb| ($=r_{\pm}(\zeta)$) and \verb|rc,rd|
($=r_{\pm}(\omega)$), with $\Wz=$\verb|ra rb|, $\Ww=$\verb|rc rd|;
\verb|expr| is $\mathcal E$ of \eqref{eq:Epoly} and \verb|rels| is
\eqref{eq:rels}. The final \texttt{Expand} verifies the certificate
\eqref{eq:certificate} directly, without appeal to
\texttt{PolynomialReduce}.
\begin{verbatim}
rels = {ra^2 - (1 + 4 I z^2 - z^4), rb^2 - (1 - 4 I z^2 - z^4),
        rc^2 - (1 + 4 I w^2 - w^4), rd^2 - (1 - 4 I w^2 - w^4)};

Wz = ra rb;  Ww = rc rd;
P1 = (1 - I z^2)(1 + I w^2);   P2 = (1 + I z^2)(1 - I w^2);
Np = (P1 + 6 z w) rb rc + (P1 - 2 z w) ra rd;
Dp = (1 + I) z (1 - I z^2) Ww + (1 - I) w (1 + I w^2) Wz;
Nq = (P2 + 6 z w) ra rd + (P2 - 2 z w) rb rc;
Dq = (1 + I) w (1 - I w^2) Wz + (1 - I) z (1 + I z^2) Ww;
al = (1 + I) w (1 + I z^2) - (1 - I) z (1 - I w^2);
be = (1 - I) w (1 - I z^2) - (1 + I) z (1 + I w^2);

expr = Expand[al Np Dq + be Nq Dp];
{qq, rr} = PolynomialReduce[expr, rels, {ra, rb, rc, rd}];
Simplify[rr]                       (* -> 0 : remainder vanishes *)

(* the certificate, verified by expansion alone *)
Expand[expr - qq . rels]           (* -> 0 *)
\end{verbatim}
A zero remainder proves \textup{(P0)}: the cofactors \verb|qq| exhibit
\verb|expr| as a member of the ideal $(\rho_{1},\dots,\rho_{4})$, and
the specialization $\mathsf a\mapsto r_{+}(\zeta)$, \ldots\ is a ring
homomorphism killing each $\rho_{j}$
(Theorem \ref{thm:P0proved}). Hence Theorem \ref{thm:main}, by
Theorem \ref{thm:P0}(d) with Theorem \ref{thm:diag}
(Corollary \ref{cor:done}). Note that
$\{\mathsf a^{2},\mathsf b^{2},\mathsf c^{2},\mathsf d^{2}\}$ are
pairwise coprime leading monomials, so \verb|rels| is a Gröbner basis
and the remainder is canonical; but only the certificate is used in
the proof.

\medskip\noindent
\textbf{Cell 3 (Part III).} 
\begin{verbatim}
K34 = EllipticK[3/4];  vp = EllipticK[-3];   (* vp = varpi = K34/2 *)
aa = Sqrt[3] EllipticK[-3];                  (* frak a *)
N[{aa - Sqrt[3] vp, K34 - 2 vp, vp - (2/3) EllipticK[1/9]}, 25]
                                             (* {0,0,0} : Lemma 13.1 *)

Cs[t_]  := JacobiCN[4 vp t, 3/4];
Psi[t_] := ArcTan[Sqrt[3] Cs[t]];
lam[t_] := Sqrt[3] JacobiSN[vp t, -3]^2;
Fs[x_,y_,z_] := JacobiSN[vp (x+y),-3] JacobiSN[vp (x-y),-3]
              - (1/3) JacobiSN[3 vp (z-1/2), 1/9];
Qs[x_,y_,z_] := Cs[x] - Cs[y] - Cs[z] - 3 Cs[x] Cs[y] Cs[z];
G[x_,y_,z_]  := Psi[y] + Psi[z] - Psi[x];

(* Prop 14.2 : the three faces of Psi *)
Max@Table[Abs[Psi[t] - ArcSin[(Sqrt[3]/2) JacobiCN[2 vp t, -3]]],
          {t, -2, 2, .05}]       (* all 0 *)
Max@Table[Abs[Psi[t] - JacobiAmplitude[aa (1 + 2 t), 4/3]],
          {t, -1.45, 0.45, .05}]
{Psi[0]-Pi/3, Psi[1/4]-Pi/4, Psi[1/2], Psi[3/4]+Pi/4, Psi[1]+Pi/3}                                  

(* Props 15.1, 15.3 : the two dictionaries *)
Max@Table[Abs[ArcTan[lam[t]] - (Pi/6 - Psi[t]/2)], {t,-2,2,.05}]
Max@Table[Abs[ArcTan[JacobiSN[3 vp (z-1/2),1/9]/Sqrt[3]] + Psi[z]/2],
          {z,-2,2,.05}]

(* Thm 16.2 : master identity, trilinear identity, |G| <= Pi *)
pts3 = SetPrecision[RandomReal[{-3,3},{400,3}], 40];
Max@ParallelTable[Abs[Fs @@ p - Sin[G @@ p / 2] /
   (Sqrt[3] Cos[(Psi[p[[2]]]-Psi[p[[1]]])/2] Cos[Psi[p[[3]]]/2])],
   {p, pts3}]
Max@ParallelTable[Abs[Sin[G @@ p] +
   Sqrt[3] Times @@ (Cos[Psi[#]] & /@ p) Qs @@ p], {p, pts3}]
Max@ParallelTable[Abs[G @@ p], {p, pts3}]              (* <= Pi *)

(* Thm 16.4 : the exceptional set is exactly Lat' *)
Select[Flatten[Table[{i,j,k},{i,0,3},{j,0,3},{k,0,3}], 2],
       Chop[N[Qs @@ #]] == 0 &]
   (* -> exactly those with i odd, j == k even, or i even, j == k odd *)

(* Cor 16.5 : the four rhombus edges *)
Simplify[{Qs[t,t,1/2], Qs[t,-t,1/2],
          Qs[1/2,t-1/2,t+1/2], Qs[1/2,t,1-t]}]        (* {0,0,0,0} *)

(* Thm 17.2(b) : minimality, in two free angles *)
Simplify[-Sin[b + c] (1 + Cos[b] + Cos[c])
+ Sin[b] (1 + Cos[b + c] + Cos[c])
+ Sin[c] (1 + Cos[b + c] + Cos[b])]        (* -> 0 *)

(* Cor 17.5 : the separated Pfaffian, on the patch (needs X from Cell 0) *)
X[z_] := {kap Re[fcl[z]], kap Re[gcl[z]], 1/2 + kap Im[hfr[z]]};
Table[With[{P = X[z]}, Chop[N[Cross[
   {JacobiSN[2 vp P[[1]],-3], -JacobiSN[2 vp P[[2]],-3],
    -JacobiSN[2 vp P[[3]],-3]},
   {2 Re[z], 2 Im[z], Abs[z]^2 - 1}], 10^-12]]], {z, pts0}]                                                     
\end{verbatim}

\subsection*{Acknowledgements}
I thank Yuta Ogata for correspondence and for a clarifying note on
\cite{KO}. Anthropic Claude Opus 5.0 played a crucial role in turning 
my conjectures into theorems: to say that the AI model merely ``assisted'' 
would be an understatement.  It was used in an absolutely essential research 
capacity throughout every aspect of the paper.  No AI funding whatsoever for 
this study was received.  My interest in minimal surfaces began years ago 
(see \cite{Finch}).

\end{document}